\documentclass[12pt,reqno]{amsart}
\usepackage{amscd}
\usepackage{amssymb}
\usepackage[mathscr]{eucal}

\makeatletter

\def\subsection{\@startsection{subsection}{2}%
  \z@{.5\linespacing\@plus.7\linespacing}{1pt}%
      {\normalfont\bfseries}}

\def\l@section{\@tocline{1}{0pt}{1pc}{}{\bfseries}}
\def\l@subsection{\@tocline{2}{0pt}{3.1em}{5pc}{}}

\makeatother

\allowdisplaybreaks[1]
\newtheorem{thm}{Theorem}[section]
\newtheorem{lem}[thm]{Lemma}
\newtheorem{cor}[thm]{Corollary}
\newtheorem{prop}[thm]{Proposition}
\newtheorem{prob}[thm]{Problem}

\theoremstyle{definition}

\newtheorem{rem}[thm]{Remark}
\newtheorem{defn}[thm]{Definition}

\theoremstyle{remark}

\numberwithin{equation}{section}

\newcommand{\C}{{\mathbb{C}}}
\newcommand{\R}{{\mathbb{R}}}
\newcommand{\Z}{{\mathbb{Z}}}

\newcommand{\N}{{\mathbb{N}}}
\newcommand{\T}{{\mathbb{T}}}
\renewcommand{\ker}{\mathrm{Ker}}

\DeclareMathOperator{\Ad}{Ad}
\DeclareMathOperator{\Tr}{Tr}

\DeclareMathOperator{\tr}{tr}
\DeclareMathOperator{\id}{id}

\DeclareMathOperator{\Projf}{Projf}

\DeclareMathOperator{\ONB}{ONB}
\DeclareMathOperator{\Irr}{Irr}

\DeclareMathOperator{\Mor}{Mor}
\DeclareMathOperator{\End}{End}

\DeclareMathOperator{\Aut}{Aut}
\DeclareMathOperator{\Cnt}{Cnt}
\DeclareMathOperator{\Cnd}{Cnd}
\DeclareMathOperator{\Int}{Int}
\def\oInt{\ovl{\Int}}

\DeclareMathOperator{\mo}{mod}

\DeclareMathOperator{\Ind}{Ind}

\DeclareMathOperator{\supp}{supp}            

\DeclareMathOperator{\diag}{diag}

\def\mF{\mathcal{F}}

\def\meC{\mathscr{C}}

\def\meG{\mathscr{G}}
\def\meH{\mathscr{H}}

\def\meJ{\mathscr{J}}
\def\meK{\mathscr{K}}

\def\meR{\mathscr{R}}

\def\al{\alpha}
\def\be{\beta}
\def\ga{\gamma}
\def\de{\delta}

\def\la{\lambda}

\def\vep{\varepsilon}
\def\ph{{\phi}}
\def\ps{{\psi}}
\def\vph{\varphi}

\def\om{\omega}
\def\si{\sigma}
\def\ta{\tau}
\def\th{\theta}

\def\Ph{\Phi}
\def\Ps{\Psi}
\def\Th{\Theta}
\def\Ga{\Gamma}
\def\De{\Delta}

\def\el{\ell}

\def\ovl{\overline}                 
\def\wdh{\widehat}
\def\wdt{\widetilde}
\def\opp{{\mathrm{opp}}}

\def\hE{{\hat{E}}}

\def\hal{{\widehat{\al}}}

\def\hga{{\hat{\ga}}}

\def\hvph{{\hat{\vph}}}
\def\tvph{{\tilde{\vph}}}

\def\hrho{{\hat{\rho}}}

\def\bal{{\ovl{\al}}}
\def\bbe{{\ovl{\be}}}

\def\tDe{{\tilde{\Delta}}}

\def\tal{{\wdt{\al}}}
\def\tbe{{\wdt{\be}}}
\def\trho{\wdt{\rho}}

\def\tM{{\wdt{M}}}

\def\opi{{\ovl{\pi}}}                
\def\orho{{\ovl{\rho}}}

\def\subs{\subset}                   
\def\Subs{\Subset}
\def\setm{\setminus}
\def\nin{\notin}

\def\per{\perp}

\def\oti{\otimes}                    
\def\rti{\rtimes}                    

\def\col{\colon}
\def\ra{\hspace{-0.5mm}\rightarrow\!}

\def\btr{{\boldsymbol{1}}}                   
\def\bG{\mathbb{G}}

\def\bhG{{\wdh{\mathbb{G}}}}

\def\bhGG{\wdh{\mathbb{G}}\times \wdh{\mathbb{G}}^\opp}

\def\lhG{L^\infty(\bhG)}                    
\def\lhGG{L^\infty(\bhG\times\bhG^\opp)}    
\def\lhGZ{L^\infty(\bhG\times\Z)}

\def\lthG{L^2(\bhG)}                        

\def\lG{L^\infty(\bG)}                      

\def\IG{\Irr(\bG)}

\author[T. Masuda]{Toshihiko Masuda$^1$}
\address{$^1$ 
Graduate School of Mathematics, Kyushu University, 
Hakozaki, Fukuoka\\ \indent \mbox{812-8581},
JAPAN}
\email{masuda@math.kyushu-u.ac.jp}

\author[R. Tomatsu]{Reiji Tomatsu$^{\,2,3}$}
\address{$^{2}$Department of 
Mathematical Sciences,
\hspace{-0.6mm}University of Tokyo,
\hspace{-0.6mm}Komaba,Tokyo\\ \indent \mbox{153-8914},
JAPAN,}
\address{$^3$Department of Mathematics, 
K.U. Leuven, Celestijnenlaan 200B,
B-3001 \ \indent Leuven, 
BELGIUM}
\email{tomatsu@ms.u-tokyo.ac.jp}

\thanks{$^{1,2,3}$ Supported by JSPS}
\begin{document}

\title{Classification of minimal actions of a compact Kac algebra with
amenable dual on injective factors of type III}

\begin{abstract}
We classify a certain class of minimal actions
of a compact Kac algebra with
amenable dual on injective factors of type III.
Our main technical tools are the structural analysis of
type III factors and the theory of canonical extension
of endomorphisms introduced by Izumi.
\end{abstract}

\maketitle

\section{Introduction}\label{sec:intro}
The purpose of this paper is to extend the classification result
of \cite{M-T-CMP} to type III case,
that is,
to show uniqueness of certain minimal actions of a compact Kac algebra
with amenable dual on injective type III factors.

On compact group actions on type III factors,
there are some preceding works relevant with our work.
The complete classification for compact abelian groups
was obtained by
Y.~Kawahigashi and M. Takesaki \cite{Kw-Tak}.
The recent result due to M. Izumi \cite{Iz-can2} is remarkable.
Among other things, he showed the conjugacy result for
certain minimal actions of compact groups.
More precisely,
if minimal actions of a compact group on type III$_0$ factors
are faithful on the flow of weights
and have the common Connes-Takesaki modules
\cite{CT},
then they are conjugate.

In this paper, we classify minimal actions whose dual actions are
approximately inner and centrally free, which can be regarded as the
generalization of classification results for trivial invariant case in
\cite{Kw-Tak}. 
One should notice that these objects are different from Izumi's ones
because minimal actions studied in \cite{Iz-can2}
are duals of free and centrally trivial actions.
Our strategy is on the whole same as \cite{M-T-CMP},
that is,
we mainly handle actions of an amenable discrete Kac algebra $\bhG$
instead of a compact Kac algebra $\bG$,
and obtain our main theorem through a duality argument \cite{ES}.
It seems, however, difficult to
generalize the argument in \cite{M-T-CMP}
to type III McDuff factors because of the lack of traces.

We present a different approach for injective type III factors
starting from the classification for type II$_1$ case \cite{M-T-CMP}.
More precisely,
we extend given actions of $\bhG$ on type III factors
to larger von Neumann algebras, which are
the crossed products by abelian group actions.
Then we classify the composed actions
of the extended actions of $\bhG$ and the dual actions.
Splitting the dual actions and taking the partial crossed products,
we show that all approximately inner
and centrally free actions come from a free action on the injective type
II$_1$ factor.
In these processes, what play crucial roles are
the structure analysis of type III factors \cite{CT,Tak-dual},
Izumi's theory on canonical extension of endomorphisms
\cite{Iz-can2} and the characterization of approximate innerness and
central triviality of endomorphisms \cite{M-T-endo-pre}.

This paper is organized as follows.

In \S \ref{sec:main}, our main results and their applications are stated.

In \S \ref{sec:pre}, we prove the necessary results for the study in
the later sections.
In particular, the relative Rohlin theorem proved in \S \ref{sec:Rel-Rohlin}
plays an important role for our model action splitting argument.

In \S \ref{sec:lamclass},
type III$_\lambda$ case ($0<\lambda<1$) is studied.
Considering the discrete decomposition,
we can reduce our problem to
classifying actions of direct product of $\bhG$
and the integer group $\Z$ on the injective
type II$_\infty$ factor.
Here, the $\Z$-action has non-trivial Connes-Takesaki module,
and the main theorem of \cite{M-T-CMP} is not immediately applicable.
However, we can show
the model action splitting as in \cite{Con-auto}
that enables us to cancel the Connes-Takesaki module
and to use the main theorem of \cite{M-T-CMP}.

In \S \ref{sec:III0},
type III$_0$ case is studied.
We make use of the continuous decomposition, and represent a flow of
weights as a flow built under a ceiling function.
Then all things are reduced to the type II case
as in \cite{Su-Tak-RIMS,Su-Tak-act}.
We classify actions of the direct product of $\bhG$
and an AF ergodic groupoid
on the injective type II$_\infty$ factor
by using
\cite{M-T-CMP} and Krieger's cohomology lemma \cite{JT}.

In \S \ref{sec:III1},
type III$_1$ case is studied.
Following the line of Connes and Haagerup's theory of
classification of injective factors of type III$_1$
\cite{Co-III1,Ha-III1},
we consider the discrete decomposition of the type III$_\la$
factor by the type III$_1$ factor.
Then we classify actions of the direct product of $\bhG$
and the torus coming from the dual action
by showing the model action splitting
in \S \ref{sec:lam-stab} and \ref{sec:1-model}.
The key point here is approximate innerness of modular automorphisms.

In \S \ref{sec:appendix},
we prove some basic results on the canonical extension
in order that readers can smoothly shift
from theory of endomorphisms to
that of actions of discrete Kac algebras.
Most of them directly follow from \cite{M-T-endo-pre}
by making use of the notion of
a Hilbert space in a von Neumann algebra \cite{Ro}.

\vspace{10pt} 
\noindent \textbf{Acknowledgments.}

A part of this work was done while the first named author
stayed at Fields Institute
and the second named author stayed
at Katholieke Universiteit Leuven.
They express gratitude for their warm hospitality.

\section{Notations and Main theorem}\label{sec:main}
Throughout this paper, we treat only separable von Neumann
algebras, except for ultraproduct von Neumann algebras.
We freely use the notations in \cite{M-T-CMP}.
For example, $\bhG=(\lhG,\De,\vph)$ denotes a discrete Kac algebra.
Although some of our results are applicable to a general
discrete Kac algebra,
we always assume the amenability of $\bhG$
before \S \ref{sec:appendix}.
See \cite{M-T-CMP} and the references therein
for the definition of a discrete (or compact) Kac algebra and its amenability.

For a von Neumann algebra $M$,
we denote by $U(M)$ the set of unitary elements in $M$.
By $W(M)$, we denote
the set of faithful normal semifinite weights
on $M$.

By \cite{Co-inj,Kri-erg,Co-III1,Ha-III1},
injective  type III factors are determined by their flow of weights.
We denote by $\meR_0$, $\meR_{0,1}$, $\meR_\lambda$ and
$\meR_\infty$ the injective factor of type II$_1$, type II$_\infty$,
type III$_\lambda$ $(0<\lambda<1)$ and type III$_1$, respectively.

Let $M$ be a factor.
For a finite dimensional Hilbert space $K$,
let $\Mor_0(M,M\otimes B(K))$ be a set of all homomorphisms with finite
index.
When $M$ is properly infinite,
we can identify $\Mor_0(M,M\otimes B(K))$ with
$\End_0(M)$, the set of endomorphisms of $M$
with finite index. (See \S \ref{sec:appendix}.)
By $\Tr_K$ and $\tr_K$, we denote the non-normalized trace and
the normalized trace on $B(K)$, respectively.

\subsection{Actions and cocycle actions}
We recall some definitions and notations used in \cite{M-T-CMP}
for readers' convenience.
Let $M$ be a von Neumann algebra,
$\al\in\Mor(M,M\oti\lhG)$ and $u\in U(M\oti\lhG\oti\lhG)$.
The pair $(\al,u)$ (or simply $\al$)
is called a \emph{cocycle action} of $\bhG$ on $M$
if the following holds:
\begin{enumerate}
\item
$(\al\oti\id)\circ\al=\Ad u\circ (\id\oti\De)\circ\al$;
\item
$(u\oti1)(\id\oti\De\oti\id)(u)=\al(u)(\id\oti\id\oti\De)(u)$;
\item
$u_{\cdot,\btr}=u_{\btr,\cdot}=1$.
\end{enumerate}
Here, $\al(u):=(\al\oti\id\oti\id)(u)$,
which is one of our conventions frequently used in our paper,
that is, we will omit $\id$ when the place where $\al$ acts is apparent.
If $u=1$, $\al$ is called an \emph{action}.
We introduce a left inverse $\Ph_\pi^\al\col M\oti B(H_\pi)\ra M$
of $\al_\pi$ for each $\pi\in\IG$ as follows:
\[
\Ph_\pi^\al(x)=(1\oti T_{\opi,\pi}^*)u_{\opi,\pi}^*
\al_\opi(x)u_{\opi,\pi}(1\oti T_{\opi,\pi})
\quad\mbox{for}\ x\in M\oti B(H_\pi).
\]
Then $\Ph_\pi^\al$ is a faithful normal unital completely positive
map with $\Ph_\pi^\al\circ\al_\pi=\id_M$ \cite[Lemma 2.4]{M-T-CMP}.
In general, a left inverse of $\al_\pi$ is not uniquely determined,
but we only use the left inverse above.
If $M$ is a factor, then $\Ph_\pi^\al$ is standard, that is,
the conditional expectation
$\al_\pi\circ\Ph_\pi^\al\col M\oti B(H_\pi)\ra \al_\pi(M)$
is minimal (see Proposition \ref{prop: standard}).
The other easy but useful remark is the fact that $u$ is evaluated in
$(M^\al)'\cap M$, where $M^\al:=\{x\in M\mid \al(x)=x\oti1\}$ is a
the fixed point algebra.
This means that $(\al|_{(M^\al)'\cap M},u)$ is a cocycle action
on $(M^\al)'\cap M$.

\subsection{Approximate innerness and central freeness}
\label{sec:apcen}

We collect basic notions of homomorphisms and actions
from \cite{M-T-CMP}.

\begin{defn}\label{defn: app-cent}
Let $M$ be a von Neumann algebra
and $\al\in \Mor_0(M, M\oti B(K))$
with the standard left inverse $\Ph^\al$.
We say that
\begin{enumerate}
\item
$\al$ is \emph{properly outer}
if there exists no non-zero element $a\in M\oti B(K)$
such that
$a(x\oti1)=\al(x)a$ for any $x\in M$;

\item
$\alpha$ is \emph{approximately inner}
if there exists a sequence $\{U^\nu\}_\nu\subset U(M\otimes B(K))$
such that 
\[
\lim_{\nu\to\infty}
\|(\vph\oti\tr_K)\circ \Ad (U^\nu)^*-\vph\circ\Ph^\al
\|=0
\quad\mbox{for all}\ \vph\in M_*;
\]

\item
$\alpha$ is \emph{centrally trivial}
if $\al^\om(x)=x\oti1$ for all $x\in M_\om$;

\item
$\alpha$ is \emph{centrally non-trivial}
if $\al$ is not centrally trivial;

\item
$\al$ is \emph{properly centrally non-trivial}
if there exists no non-zero element $a\in M\oti B(K)$
such that
$\al^\om(x)a=(x\oti1)a$ for all $x\in M_\om$.
\end{enumerate}
\end{defn}

\begin{defn}
Let $\al\in\Mor(M,M\oti\lhG)$ be a cocycle action of $\bhG$.
We say that
\begin{enumerate}
\item
$\al$ is \emph{free}
if $\al_\pi$ is properly outer for all $\pi\in\IG\setm\{\btr\}$;

\item
$\al$ is \emph{approximately inner}
if $\al_\pi$ is approximately inner for all $\pi\in\IG$;

\item
$\al$ is \emph{centrally free}
if $\al_\pi$ is properly centrally non-trivial
for all $\pi\in\IG\setm\{\btr\}$.
\end{enumerate}
\end{defn}

Note the following fact.
If $\al$ is a free action of $\bhG$ on a factor
\cite[Definition 2.7]{M-T-CMP},
then $\al_\pi$ is irreducible for each $\pi\in\IG$
\cite[Lemma 2.8]{M-T-CMP}.
If $\al_\pi$ is irreducible,
then central non-triviality is
equivalent to
properly central non-triviality
\cite[Lemma 8.3]{M-T-CMP}.
Hence a free action $\al$ on a factor is centrally free
if and only if $\al_\pi$ is centrally non-trivial
for each $\pi\in\IG\setm\{\btr\}$.

\subsection{Main theorem}

We recall the notion of
the cocycle conjugacy for two (cocycle) actions.
\begin{defn}
Let $M$ and $N$ be von Neumann algebras.
Let $\al\in \Mor(M,M\oti\lhG)$ and $\be\in\Mor(N,N\oti\lhG)$
be cocycle actions of $\bhG$ with 2-cocycles $u$ and $v$,
respectively.
\begin{enumerate}
\item $\al$ and $\be$ are said to be \emph{conjugate}
if there exists an isomorphism $\th\col M\ra N$
such that
\begin{itemize}
\item
$\al=(\th^{-1}\oti\id)\circ\be \circ\th$;
\item
$u=(\th^{-1}\oti\id\oti\id)(v)$.
\end{itemize}
We write $\al\approx\be$ if $\al$ and $\be$ are conjugate.

\item
$\al$ and $\be$ are said to be \emph{cocycle conjugate}
if there exist an isomorphism $\th\col M\ra N$
and $w\in U(M\oti \lhG)$ such that
\begin{itemize}
\item
$\Ad w\circ\al=(\th^{-1}\oti\id)\circ\be\circ \th$;
\item
$w\al(w)u(\id\oti\De)(w^*)=(\th^{-1}\oti\id\oti\id)(v)$.
\end{itemize}
We write $\al\sim\be$ if $\al$ and $\be$ are cocycle conjugate.
\end{enumerate}
\end{defn}
When $\al$ is an action,
$v\in U(M\oti\lhG)$ is called an $\al$-\emph{cocycle}
if $(v\oti1)\al(v)=(\id\oti\De)(v)$.
The following is the main theorem of this paper
which asserts the uniqueness of approximate inner and
centrally free action.

\begin{thm}\label{thm:main}
Let $\bG$ be a compact Kac algebra with amenable dual,
$M$ an injective factor,
$\alpha$ an approximately inner
and centrally free action of $\bhG$ on $M$,
and $\alpha^{(0)}$ a free action of
$\bhG$ on $\meR_0$.
Then $\alpha$ is cocycle
conjugate to $\id_M\otimes \alpha^{(0)}$.
\end{thm}

This implies the following
in view of the duality theorem \cite[Theorem IV.3]{ES}.

\begin{thm}\label{thm:main2}
Let $\bG$ be a compact Kac algebra with amenable dual,
$M$ an injective factor,
$\alpha$ a minimal action of $\bG$ on $M$,
and
$\alpha^{(0)}$ a minimal action of $\bG$ on $\meR_0$.
If the dual
action of $\alpha$ is approximately inner and centrally free,
then $\alpha$ is cocycle conjugate to
$\id_M\otimes \alpha^{(0)}$.
If $\alpha$ is a dual action,
then $\alpha$ and
$\id_M\otimes \alpha^{(0)}$ are conjugate.
\end{thm}

As a corollary, we obtain the following classification of minimal
actions of compact Lie groups on $\meR_\infty$.
\begin{cor}
Let $G$ be a semisimple connected compact Lie group. Then any two minimal
actions of $G$ on $\meR_\infty$ are conjugate to each other.
\end{cor}
\begin{proof}
This follows from Theorem \ref{thm:main2},
\cite[Theorem 3.15,\ 4.12]{M-T-endo-pre}
and \cite[Corollary  5.14]{Iz-can2}. 
\end{proof}

Our main purpose is to prove Theorem \ref{thm:main}.
In \cite[Theorem 7.1]{M-T-CMP}, we have proved that
in type II$_1$ case.
The remaining cases are type II$_\infty$, III$_\la$ ($0<\la<1$),
III$_0$ and III$_1$.
Type II$_\infty$ case is easily shown as follows.

\vspace{10pt}
\noindent
$\bullet$
\textit{Proof of Theorem \ref{thm:main} for $\meR_{0,1}$.}

Let $\al$ be an approximately inner and centrally free action
of $\bhG$ on $\meR_{0,1}$.
Let $\ta$ be a normal trace on $\meR_{0,1}$.
Since $\al$ is approximately inner,
we have $\ta\circ\Ph_\pi^\al=\ta\oti\tr_\pi$ for $\pi\in\IG$
by Corollary \ref{cor: app-cent-infinite}.
Hence $\ta$ is invariant under $\al$.

Let $\{e_{i,j}\}_{i,j=1}^\infty\subs \meR_{0,1}$
be a system of matrix units
with a finite projection $e_{11}$.
Since $(\ta\oti \tr_\pi)(e_{11}\oti1)=(\ta\oti \tr_\pi)(\al_\pi(e_{11}))$
for each $\pi\in\IG$,
we can take $v\in \meR_{0,1}\oti \lhG$ such that
$vv^*=e_{11}\oti1$ and $v^*v=\al(e_{11})$.
Set a unitary
$V=\sum_{i=1}^\infty (e_{i1}\oti1)v\al(e_{1i})$.
Then the perturbed cocycle action $\Ad V\circ\al$
fixes the type I factor $B:=\{e_{i,j}\}_{i,j}''$.
Therefore $\Ad V\circ\al|_{B'\cap \meR_{0,1}}$
is an approximately inner
and centrally free cocycle action
on the injective type II$_1$ factor $B'\cap \meR_{0,1}$.
By \cite[Theorem 6.2]{M-T-CMP},
we can perturb $\Ad V\circ\al|_{B'\cap \meR_{0,1}}$ to be an action.
Then this action is cocycle conjugate to the model action
$\al^{(0)}$.
Therefore we have $\al\sim\id_{B(\el_2)}\oti\al^{(0)}$.
Using $\al^{(0)}\sim \id_{\meR_0}\oti\al^{(0)}$,
we obtain $\al\sim\id_{\meR_{0,1}}\oti \al^{(0)}$.
\hfill$\Box$
\\

By Theorem \ref{thm:main}, any two approximately inner and centrally
free actions $\al$ and $\be$ on an injective factor
$M$ are cocycle conjugate.
This can be more precisely stated as \cite[Theorem 7.1]{M-T-CMP}.

\begin{thm}
Let $M$ be an injective factor and $\bhG$ an amenable discrete Kac algebra.
Let $\al$ and $\be$ be approximately inner and centrally free actions
of $\bhG$ on $M$.
Then there exists $\th\in\oInt(M)$ and an $\al$-cocycle $v\in M\oti \lhG$
such that
\[
\Ad v\circ\al=(\th^{-1}\oti\id)\circ\be\circ\th.
\]
\end{thm}
\begin{proof}
Since $M$ is injective, $M$ is isomorphic to $\meR_0\oti M$.
Fix an isomorphism $\Psi:M\rightarrow M\oti\meR_0$.
Let $\al^{(0)}$ be a free action of $\bhG$ on $\meR_0$.
Set $\ga:=(\Ps^{-1}\oti\id)\circ(\id_M\oti\al^{(0)})\circ\Ps$,
which is an approximately inner and centrally free action on $M$.
By Theorem \ref{thm:main}, we can take
$\theta_0\in \mathrm{Aut}(M)$ and an $\al$-cocycle $v$
such that
$\mathrm{Ad} v\circ\alpha=(\theta_0^{-1}\oti\id)\circ\ga \circ\theta_0$.
To prove the theorem, it suffices to show the statement for $\be=\ga$.

Set $\th_1:=\Ps\circ\th_0\circ \Ps^{-1}\in \Aut(M\oti\meR_0)$.
Note that the core of $M\oti\meR_0$ canonically coincides with
$\tM\oti\meR_0$.
Since 
the module map $\mo\col \Aut(M)\ra \Aut_\th(Z(\tM))$ is surjective
by \cite{ST2},
there exists $\th_2\in\Aut(M)$
such that $\mo(\th_1)=\mo(\th_2\oti \id_{\meR_0})$.
Set $\th_3:=\Ps^{-1}\circ(\th_2\oti\id_{\meR_0})\circ\Ps\in\Aut(M)$.
Then $\th_3^{-1}\th_0=\Ps^{-1}\circ(\th_2^{-1}\oti\id_{\meR_0})\th_1\circ\Ps$
implies $\mo(\th_3^{-1}\th_0)=\id$.
Putting $\th:=\th_3^{-1}\th_0$,
we have $\Ad v\circ\al=(\th^{-1}\oti\id)\circ\ga\circ\th$
because $\th_3$ commutes with $\ga$.
Moreover $\th\in\oInt(M)$ by \cite[Theorem 1(1)]{KST}.
\end{proof}

\section{Preliminaries}
\label{sec:pre}
The results in this section are frequently
used in the later sections.
One of the most important results is the relative Rohlin theorem
(Theorem \ref{thm: jRoh-G2-fix}).

\subsection{Basic results on cocycle conjugacy}

\begin{lem}\label{lem: id-otimes-al}
Let $(\al,u)$ be a cocycle action of $\bhG$ on a properly infinite
von Neumann algebra $M$.
Let $H$ be a Hilbert space.
Then $(\al,u)$ and $(\id_{B(H)}\oti\al,1\oti u)$
are cocycle conjugate.
\end{lem}
\begin{proof}
Take a Hilbert space $\meH\subs M$
with support $1$ and the same dimension $d\leq\infty$ as $H$ \cite{Ro}.
Let $\{\xi_i\}_{i=1}^{d}$ be an orthonormal basis of $\meH$.
Then we have the isomorphism
$\Ps\col B(H)\oti M\ra M$ such that
$\Ps(e_{ij}\oti x)=\xi_i x\xi_j^*$ for all $x\in M$ and $i,j\in\N$,
where $\{e_{ij}\}_{ij}$ is a canonical system of matrix units
of $B(H)$.

Define the unitary $v:=\sum_{i=1}^d(\xi_i\oti1)\al(\xi_i^*)$.
We check that $\Ps$ and $v$ satisfy the statement.
For $x\in M$ and $i,j\in\N$, we have
\begin{align*}
\Ad v\circ \al\circ \Ps(e_{ij}\oti x)
&=
v\al(\xi_i x\xi_j^*)v^*
=
(\xi_i\oti1)\al(x)(\xi_j^*\oti1)
\\
&=
(\Ps\oti\id)\circ(\id\oti\al)(e_{ij}\oti x).
\end{align*}
Hence (1) holds.
On (2), we have
\begin{align*}
&\,(v\oti1)\al(v)u(\id\oti\De)(v^*)
\\
&=
\sum_{i,j=1}^d
(\xi_i\oti1\oti1)(\al(\xi_i^*)\oti1)
\cdot
(\al(\xi_j)\oti1)\al(\al(\xi_j^*))
\cdot u(\id\oti\De)(v^*)
\\
&=
\sum_{i=1}^d
(\xi_i\oti1\oti1)u(\id\oti\De)(\al(\xi_i^*))
(\id\oti\De)(v^*)
\\
&=
\sum_{i=1}^d
(\xi_i\oti1\oti1)u(\xi_i^*\oti1\oti1)
=(\Ps\oti\id\oti\id)(1\oti u).
\end{align*}
\end{proof}

\begin{lem}\label{lem: prop-inf-coboundary}
Let $(\al,u)$ be a cocycle action on a properly infinite
von Neumann algebra $M$.
Then $u$ is a coboundary.
\end{lem}
\begin{proof}
By the previous lemma,
it suffices to prove that $(\id_{B(\lthG)}\oti\al,1\oti u)$
can be perturbed to an action.
Write $\overline{\al}=\id_{B(\lthG)}\oti\al$ and
$\ovl{u}=1\oti u$.
Then we set a unitary
$v:=W_{31} u_{231}^*\in B(\lthG)\oti M\oti \lhG$,
where $W\in \lhG\oti \lG$ is the multiplicative unitary
defined in \cite[Section 2]{M-T-CMP}.
Using the 2-cocycle relation of $u$ and
$\De(x)=W^*(1\oti x)W$ for $x\in \lhG$,
we have
\begin{align*}
&\,v_{123}\ovl{\al}(v)\ovl{u}(\id\oti\id\oti\De)(v^*)
\\
&=
W_{31}u_{231}^*\cdot W_{41}\al(u^*)_{2341}\cdot
u_{234}\cdot (\id\oti\De\oti\id)(u)_{2341}
(\De\oti\id)(W^*)_{341}
\\
&=
W_{31}u_{231}^*\cdot W_{41}
\big{(}
\al(u^*)\cdot(u\oti1)\cdot (\id\oti\De\oti\id)(u)\big{)}_{2341}
(\De\oti\id)(W^*)_{341}
\\
&=
W_{31}u_{231}^*\cdot W_{41}
\big{(}(\id\oti\id\oti\De)(u)\big{)}_{2341}
(\De\oti\id)(W^*)_{341}
\\
&=
W_{31}u_{231}^*\cdot W_{41}
\big{(}W_{34}^*u_{124}W_{34}\big{)}_{2341}
(\De\oti\id)(W^*)_{341}
\\
&=
W_{31}u_{231}^*\cdot W_{41}
\cdot W_{41}^*u_{231}W_{41}\cdot
(\De\oti\id)(W^*)_{341}
\\
&=
W_{31}W_{41}\cdot
(\De\oti\id)(W^*)_{341}
=1.
\end{align*}
\end{proof}

Next we discuss
the cocycle conjugacy of extended actions.
For definition of the canonical extension of a cocycle action,
readers are referred to \cite{Iz-can2} and \S \ref{sec:appendix}.

\begin{lem}\label{lem: second-cocycle}
Let $\al$ be an action of $\bhG$ on a
properly infinite von Neumann algebra $M$.
Then the second canonical extension $\widetilde{\tal}$
on $\tM\rti_\th\R$ is cocycle conjugate to $\al$.
\end{lem}
\begin{proof}
This is immediately obtained
from Lemma \ref{lem: id-otimes-al} and Corollary \ref{cor: second-conjugate}.
\end{proof}

We close this subsection with the following lemma.

\begin{lem}\label{lem: G1-G2-al-be}
Let $\bhG^i$ be a discrete Kac algebra for each $i=1,2$.
Let $\al^i$ and $\be^i$ be actions of $\bhG^i$
on von Neumann algebras $M$
and $N$, respectively.
Assume the following:
\begin{itemize}
\item
$\al^1$ and $\al^2$ commute;

\item
$\be^1$ and $\be^2$ commute;

\item
The $\bhG^1\times \bhG^2$ actions
$\al:=(\al^1\oti\id)\circ\al^2$
and $\be:=(\be^1\oti\id)\circ\be^2$ are cocycle conjugate.
\end{itemize}
Then the action $\al^1$ (resp. $\beta^1$) extends to the action $\ovl{\al}^1$
on $M\rti_{\al^2}\bhG^2$ (resp. $\overline{\beta}^2$ on $M\rti_{\be^2}\bhG^2$). 
Moreover, $\ovl{\al}^1$ and $\ovl{\be}^1$ are cocycle conjugate.
\end{lem}
\begin{proof}
Let $v$ be an $\al$-cocycle and $\Ps\col M\ra N$ be an isomorphism
such that $\Ad v\circ \al=(\Ps^{-1}\oti\id)\circ \be\circ\Ps$.
Set unitaries $v^\el:=v_{\cdot\oti\btr}\in M\oti L^\infty(\bhG^1)$
and $v^r:=v_{\btr\oti\cdot}\in M\oti L^\infty(\bhG^2)$,
which are $\al^1$-cocycle and $\al^2$-cocycle, respectively.
Then we define an isomorphism
$\Theta\col M\rti_{\al^2}\bhG^1\ra N\rti_{\be^2}\bhG^2$
by $\Theta(x)=(\Ps\oti\id)(v^rx(v^r)^*)$.
We set a unitary
$u:=(\al^2\oti\id)(v^\el)\in (M\rti_{\al^2}\bhG^2)\oti L^\infty(\bhG^1)$.
Then $u$ is an $\ovl{\al}^1$-cocycle.
By direct calculation, we have
$\Ad u\circ \ovl{\al}^1=(\Theta^{-1}\oti\id)\circ \ovl{\be}^1\circ\Theta$.
\end{proof}

\subsection{Rohlin property}\label{sec:Roh}
See \cite[Section 3]{M-T-CMP} for notions of ultraproduct algebras
and actions on them.
First we recall the following definition
\cite[Definition 3.4, 3.13]{M-T-CMP}.

\begin{defn}
Let $\ga\in \Mor(M^\om,M^\om\oti\lhG)$ be an action of $\bhG$.
We say that
\begin{enumerate}
\item
$\ga$ is \emph{strongly free}
when
for any $\pi\in\IG\setm\{\btr\}$
and any countably generated von Neumann subalgebra $S\subs M^\om$,
there exists no non-zero $a\in M^\om\oti B(H_\pi)$
such that $\ga_\pi(x)a=a(x\oti1)$ for all $x\in S'\cap M_\om$.

\item
$\ga$ is \emph{semi-liftable}
when
for any $\pi\in\IG$,
there exist elements
$\be^\nu, \be\in \Mor_0(M,M\oti B(H_\pi))$,
$\nu\in\N$,
such that $\be^\nu$ converges to $\be$
and $\ga_\pi((x^\nu)_\nu)=(\be^\nu(x^\nu))_{\nu}$
for all $(x^\nu)_\nu\in M^\om$.
\end{enumerate}
\end{defn}

Note that a cocycle action $\al\in\Mor(M,M\oti \lhG)$
is centrally free if and only if $\al^\om$ is strongly free
\cite[Lemma 8.2]{M-T-CMP}.
For $(x^\nu)_\nu\in M^\om$,
we set $\displaystyle\ta^\om(x):=\lim_{\nu\to\om}x^\nu$.
Then $\ta^\om\col M^\om\ra M$ is a faithful normal conditional expectation.

\begin{defn}
Let $\bhG$ be an amenable discrete Kac algebra
and $\ga\in \Mor(M_\om,M_\om\oti\lhG)$ an action.
We say that $\ga$ has the \emph{Rohlin property}
when
for any central $F\in \Projf(\lhG)$, $\de>0$,
$(F,\de)$-invariant central $K\in\Projf(\lhG)$
with $K\geq e_\btr$,
any countable subset $S\subs M^\om$
and any faithful state $\ph\in M_*$,
there exists a projection $E\in M_\om\oti \lhG$
such that
\begin{enumerate}
\renewcommand{\labelenumi}{(R\arabic{enumi})}
\item $E$ is supported over $K$, that is, $E=E(1\oti K)$;

\item $E$ almost intertwines $\ga$ and $\De$ in the following sense:
\[
|\ga_F(E)-(\id\oti{}_F\De)(E)|_{\ph\circ\ta^\om\oti\vph\oti\vph}
\leq 5\de^{1/2}|F|_\vph;
\]

\item $E$ gives a copy of $\lhG K$, that is,
if we decompose $E$ as 
\[
E=\sum_{\pi\in \supp(K)} \sum_{i,j\in I_\pi} 
d(\pi)^{-1} E_{\pi_{i,j}}\oti e_{\pi_{i,j}}, 
\]
then, for all $i,j\in I_\pi$, $k,\el\in I_\rho$
and $\pi,\rho\in\supp(K)$, we have
\[
E_{\pi_{i,j}} E_{\rho_{k,\el}}=\de_{\pi,\rho}\de_{j,k}E_{\pi_{i,\el}};
\]

\item 
$(\id\oti\vph_\pi)(E)\in S'\cap M_\om$ for any $\pi\in\supp(K)$;

\item $E$ gives a partition of unity of $S'\cap M_\om$, that is,
$(\id\oti\vph)(E)=1$. 

\end{enumerate}
\end{defn}

The above projection $E$ is called a \emph{Rohlin projection}.

\begin{defn}\label{defn: j-Rohlin}
Let $\ga\in \Mor(M^\om,M^\om\oti\lhG)$ be
an action.
Assume that $M_\om$ is globally invariant under $\ga$
and $\ga|_{M_\om}$ has the Rohlin property.
We say that $\ga$ has the \emph{joint Rohlin property}
when 
for any $F\in \Projf(\lhG)$, $\de>0$,
$(F,\de)$-invariant central $K\in\Projf(\lhG)$
with $K\geq e_\btr$,
any countable set $S\subs M^\om$
and any countable family of $\ga$-cocycles $\meC$
which are evaluated in $M^\om$,
there exists a projection $E\in M_\om\oti \lhG$
such that

\begin{enumerate}
\renewcommand{\labelenumi}{(S\arabic{enumi})}
\item
$E$ satisfies (R1), (R2), (R3), (R4) and (R5);

\item
For any $v\in \meC$, a projection $vEv^*$
also satisfies (R3);

\item
For any $v\in \meC$ and $\pi\in\supp(K)$,
we have $(\id\oti\vph_\pi)(vEv^*)=(\id\oti\vph_\pi)(E)$.

\end{enumerate}
\end{defn}

\begin{lem}
If $\ga$ has the joint Rohlin property and $E$ is a projection as above, 
then the element $(\id\oti\vph)(vE)$ is a unitary for all $v\in\meC$.
\end{lem}
\begin{proof}
Set $\mu:=(\id\oti\vph)(vE)$ and $E':=vEv^*$.
Then,
\[
\mu=\sum_{\pi\in \IG}\sum_{i,j\in I_\pi}
v_{\pi_{i,j}}E_{\pi_{j,i}}
=
\sum_{\pi\in \IG}\sum_{i,j\in I_\pi}E'_{\pi_{i,j}}v_{\pi_{j,i}}.
\]

Using (R3) for $E$ and $E'$,
we can check $\mu\mu^*=1=\mu^*\mu$ as follows,
\begin{align*}
\mu\mu^*
&=
\sum_{\pi\in \IG}\sum_{i,j,k,\el\in I_\pi}
v_{\pi_{i,j}}E_{\pi_{j,i}} E_{\pi_{k,\el}}^*v_{\pi_{k,\el}}^*
=
\sum_{\pi\in \IG}\sum_{i,j,k\in I_\pi}
v_{\pi_{i,j}}E_{\pi_{j,k}}v_{\pi_{k,i}}^*
\\
&=
(\id\oti\vph)(E')=1,
\end{align*}
and
\begin{align*}
\mu^*\mu
&=
\sum_{\pi\in \IG}\sum_{i,j,k,\el\in I_\pi}
v_{\pi_{j,i}}^* (E')_{\pi_{i,j}}^*
(E')_{\pi_{k,\el}}v_{\pi_{\el,k}}
=
\sum_{\pi\in \IG}\sum_{j,k,\el\in I_\pi}
v_{\pi_{j,k}}^*(E')_{\pi_{j,\el}}v_{\pi_{\el,k}}
\\
&=
(\id\oti\vph)(v^* E'v)
=(\id\oti\vph)(E)=1.
\end{align*}
\end{proof}
Such an element $(\id\oti\vph)(vE)$ is called a \emph{Shapiro unitary}.

Let $\bhG^1:=\bhG=(\lhG,\De)$, $\bhG^2=(L^\infty(\bhG^2),\De^2)$
be amenable discrete Kac algebras with
the invariant weights $\vph^1:=\vph$ and $\vph^2$, respectively.
The product Kac algebra $\bhG\times\bhG^2$
is naturally constructed.
The invariant weight and the coproduct are denoted by
$\tvph=\vph_{\bhG\times\bhG^2}$ and $\tDe=\De_{\bhG\times\bhG^2}$,
respectively.

\begin{lem}\label{lem: de-invariant}
Take $(F_i,\de_i)$-invariant central projection $K_i\in L^\infty(\bhG^i)$
for $i=1,2$, respectively.
Then $K_1\oti K_2$ is $(F_1\oti F_2,\de_1+\de_2)$-invariant.
\end{lem}
\begin{proof} 
\begin{align*}
&|(F_1\oti F_2\oti1\oti1)\tDe(K_1\oti K_2)
-F_1\oti F_2\oti K_1\oti K_2|_{\tvph\oti\tvph}
\\
&\leq
|(F_1\oti F_2\oti1\oti1)\tDe(K_1\oti K_2)
-(F_1\oti F_2\oti K_1 \oti1) \De^2(K_2)_{24}
|_{\tvph\oti\tvph}
\\
&\quad
+
|(F_1\oti F_2\oti K_1 \oti1) \De^2(K_2)_{24}
-
F_1\oti F_2\oti K_1\oti K_2|_{\tvph\oti\tvph}
\\
&=
|(F_1\oti1)\De(K_1)-(F_1\oti K_1)|_{\vph^1\oti\vph^1}
|(F_2\oti1)\De^2(K_2)|_{\vph^2\oti\vph^2}
\\
&\quad+
|(F_2\oti1)\De^2(K_2)-(F_2\oti K_2)|_{\vph^2\oti\vph^2}
|F_1\otimes K_1|_{\vph^1\oti \vph^1}
\\
&<
(\de_1+\de_2)|F_1\oti F_2|_\tvph |K_1\oti K_2|_\tvph. 
\end{align*}
\end{proof}

The following elementary lemma is useful.

\begin{lem}\label{lem: P-Q}
Let $P$, $Q$ be von Neumann algebras.
Let $\ph\in P_*$ and $\ps\in Q_*$
be faithful positive functionals, respectively.
Let $X, Y\in P\oti Q$ be given.
If $X\in (P\oti Q)_{\ph\oti\ps}$,
then one has 
\[
|(\id\oti\ps)(YX)|_{\ph}
\leq 
\|Y\||X|_{\ph\oti\ps}
.\]
\end{lem}
\begin{proof}
Let $(\id\oti\ps)(YX)=w|(\id\oti\ps)(YX)|$ be the 
polar decomposition. 
Since $X$ commutes with $\ph\oti\ps$, 
we have 
\begin{align*}
|(\id\oti\ps)(YX)|_{\ph}
&=
\ph(w^*(\id\oti\ps)(YX))
=
(\ph\oti\ps)((w^*\oti1)YX)
\\
&\leq
\|(w^*\oti1)Y\||X|_{\ph\oti\ps}
\leq
\|Y\||X|_{\ph\oti\ps}.
\end{align*}
\end{proof}

\subsection{Relative Rohlin theorem}\label{sec:Rel-Rohlin}

Throughout this subsection,
we are given the following:
\begin{enumerate}
\renewcommand{\labelenumi}{(A\arabic{enumi})}
\item
A von Neumann algebra $M$;
\item
A $\bhG$-action $\ga^1$ on $M^\om$ 
and a $\bhG^2$-action $\ga^2$ on $M^\om$,
and they are commuting;

\item
The $\bhG\times \bhG^2$-action
$\ga:=(\ga^1\oti \id)\circ\ga^2$ is strongly free
and semi-liftable;
\label{item: s-s}

\item
$M_\om$ is globally invariant under $\ga$;

\item
$\ta^\om\circ \Ph_{(\pi,\rho)}^\ga
=\ta^\om\oti\tr_\pi\oti\tr_\rho$
on $M_\om\oti B(H_{(\pi,\rho)})$
for all $(\pi,\rho)\in \IG\times \Irr(\bG^2)$.
\label{item: pi-rho}
\end{enumerate}

The assumption (A3) restricts not only $\ga$ but also $M_\om$.
For example, $M_\om=\C$ is excluded.
When $M$ is a factor, (A5) automatically holds.
Indeed by semi-liftability, we can take $(\be^\nu)_\nu$,
a sequence of homomorphisms on $M$ converging to $\be$ and defining
$\ga_{(\pi,\rho)}$,
that is, $\ga_{(\pi,\rho)}(x)=(\be^\nu(x))_\nu$
for $x=(x^\nu)_\nu\in M^\om$.
Then by \cite[Lemma 3.3]{M-T-CMP},
we obtain $\ta^\om\circ\Ph_{(\pi,\rho)}^\ga=\Ph^\be\circ(\ta^\om\oti\id)$
on $M^\om\oti B(H_{(\pi,\rho)})$.
Since $M$ is a factor, $\ta^\om|_{M_\om}$ is a trace.
Hence for $x\in M_\om$ and $y\in B(H_{(\pi,\rho)})$,
we have
\begin{align*}
\ta^\om\circ\Ph_{(\pi,\rho)}^\ga(x\oti y)
&=
\Ph^\be\circ(\ta^\om\oti\id)(x\oti y)
=\ta^\om(x)\Ph^\be(1\oti y)
\\
&=
\ta^\om(x)\ta^\om(\Ph_{(\pi,\rho)}^\ga(1\oti y))
=
\ta^\om(x)(\tr_\pi\oti\tr_\rho)(y).
\end{align*}
This shows $\ta^\om\circ\Ph_{(\pi,\rho)}^\ga=\ta^\om\oti\tr_\pi\oti\tr_\rho$
as desired.

Our aim is to prove the relative Rohlin theorem which assures that
a Rohlin projection for $\ga^1$ can be evaluated
in $M_\om^{\ga^2}$.

Let us take $F_i, K_i$ and $\de_i$ for $i=1,2$
as in Lemma \ref{lem: de-invariant}.
We may assume that $K_i\geq e_\btr$ for each $i$ (see \cite[\S 2.3]{M-T-CMP}).
Set $F=F_1\oti F_2$, $\de=\de_1+\de_2$ and $K=K_1\oti K_2$.
Set $\meK^i=\supp(K_i)$ for each $i$ and $\meK=\supp(K)=\meK_1\times\meK_2$.
We fix a faithful state $\ph\in M_*$
and set $\ps:=\ph\circ\ta^\om$.
Let $\meC$ be a countable family of $\ga^1$-cocycles.
Let $S\subs M^\om$ and
$T\subs (M^\om)^{\ga}$ a countably generated von Neumann subalgebras.
For a projection $E\in M^\om\oti L^\infty(\bhG\times\bhG^2)$,
we denote by $\hE$ the sliced element $(\id\oti\id\oti\vph^2)(E)$.

Define the set $\meJ$ consisting of projections
in $(T'\cap M_\om)\oti L^\infty(\bhG\times\bhG^2)K$
such that $E\in\meJ$ if and only if
$E$ satisfies (R1), (R3) and (R4)
and, in addition, $\hE$ satisfies (S2) and (S3) for $\meC$.
Since $0\in \meJ$, $\meJ$ is non-empty.
Define the following functions $a$, $b$, $c$ and $d$ from $\meJ$ to $\mathbb{R}_+$:
\begin{align*}
a_E
&=
|F|_\vph^{-1}|\ga_F(E)-(\id\oti{}_{F}\tDe)(E)|_{\ps\oti\tvph\oti\tvph};
\\
b_E
&=
|E|_{\ps\oti\tvph};
\\ 
c_E
&=
|F_2|_{\vph^2}^{-1}
|\ga_{F_2}^2(\hat{E})_{132}-\hE\oti F_2|_{\ps\oti\vph^1\oti\vph^1};
\\
d_E
&=
|F_1|_{\vph^1}^{-1}
|\ga_{F_1}^1(E)
-(\id\oti{}_{F_1}\De_{\bhG}\oti \id)(E)|_{\ps\oti\vph^1\oti\tvph}.
\end{align*}

\begin{lem}
Let $E\in\meJ$.
Assume that $b_E<1-\de^{1/2}$.
Then there exists $E'\in \meJ$ such that
\begin{enumerate}
\item 
$a_{E'}-a_E\leq 3\de^{1/2} (b_{E'}-b_E)$;
\item
$0<(\de^{1/2}/2)|E'-E|_{\ps\oti\vph}\leq b_{E'}-b_{E}$;
\item 
$c_{E'}-c_E\leq 4\de_2^{1/2}(b_{E'}-b_E)$;
\item
$d_{E'}-d_E\leq 3\de_1^{1/2} (b_{E'}-b_E)$.
\end{enumerate}
\end{lem}
\begin{proof}
Our proof is similar to the one presented in \cite{Ocn-act}.
We may assume that $S$ contains $T$ and 
the matrix entries of all $v\in \meC$,
and that $S$ is globally $\ga$-invariant.
We add the matrix entries of $E$ to $S$
and denote the new countably generated von Neumann algebra
by $\tilde{S}$.
Again we may and do assume that $\tilde{S}$ is globally $\ga$-invariant.
Take $\de_3>0$ such that $b_E<(1-\de_3)(1-\de^{1/2})$.

Recall our assumptions (A\ref{item: s-s}) and (A\ref{item: pi-rho}).
Then by \cite[Lemma 5.3]{M-T-CMP},
there is a partition of unity $\{e_i\}_{i=0}^q\subs \tilde{S}'\cap M_\om$
such that
\begin{enumerate}
\item $|e_0|_\ps<\de_3$;

\item $(e_i\oti1_\pi\oti 1_\rho)\ga_{(\pi,\rho)}(e_i)=0$
for all $1\leq i\leq q$ and $(\pi,\rho)\in \ovl{\meK}\cdot\meK\setm\{\btr\}$.
\end{enumerate}

Set a projection
$f_i:=(\id\oti\tvph)(\ga_K(e_i))\in \tilde{S}'\cap M_\om$.
We claim that at least one $i$ with $1\leq i\leq q$ satisfies
$|E(f_i\oti 1\oti1)|_{\ps\oti\tvph}<(1-\de^{1/2})|f_i|_\ps$.
Since
\begin{align*}
&\hspace{16pt}|E(f_i\oti 1\oti1)|_{\ps\oti\tvph}
\\
&=
(\ps\oti\tvph)(E(f_i\oti1))
\\
&=
\ps((\id\oti\tvph)(E) (\id\oti\tvph)(\ga_K(e_i)))
\\
&=
\sum_{(\pi,\rho)\in\meK}d(\pi)^2d(\rho)^2
\ps((\id\oti\tvph)(E) \Ph_{(\opi,\orho)}^\ga(e_i\oti1_\opi\oti1_\orho))
\\
&=
\sum_{(\pi,\rho)\in\meK}d(\pi)^2d(\rho)^2
\ps(\Ph_{(\opi,\orho)}^\ga
(\ga_{(\opi,\orho)}((\id\oti\tvph)(E) )(e_i\oti1_\opi\oti1_\orho)))
\\
&=
\sum_{(\pi,\rho)\in\meK}d(\pi)^2d(\rho)^2
(\ps\oti\tr_\opi\oti\tr_\orho)
(\ga_{(\opi,\orho)}((\id\oti\tvph)(E) )(e_i\oti1_\opi\oti1_\orho))
\\
&=
(\ps\oti\tvph)(\ga_{\ovl{K}}((\id\oti\tvph)(E))(e_i\oti1\oti1)), 
\end{align*}
we have the following:
\begin{align*}
\sum_{i=1}^q |E(f_i\oti 1\oti1)|_{\ps\oti\tvph}
&=
(\ps\oti\tvph)(\ga_{\ovl{K}}((\id\oti\tvph)(E))(e_0^\per \oti1\oti1))
\\
&\leq
(\ps\oti\tvph)(\ga_{\ovl{K}}((\id\oti\tvph)(E)))
=
|K|_\tvph (\ps\oti\tvph)(E)
\\
&=
b_E |K|_\tvph
\\
&<
(1-\de_3)(1-\de^{1/2})|K|_\tvph.
\end{align*}
If $|E(f_i\oti 1\oti1)|_{\ps\oti\tvph}\geq
(1-\de^{1/2})|f_i|_\ps(=(1-\de^{1/2})|e_i|_\ps |K|_\tvph)$
for all $1\leq i\leq q$,
then we have
\[
(1-\de^{1/2})|e_0^\per|_\ps |K|_\tvph < (1-\de_3)(1-\de^{1/2})|K|_\tvph. 
\]
This is a contradiction with $|e_0|_\ps<\de_3$.
Hence there exists $f_i$ such that 
$|E(f_i\oti 1\oti1)|_{\ps\oti\tvph}<(1-\de^{1/2})|f_i|_\ps$.
Set $e:=e_i$ and $f:=(\id\oti\tvph)(\ga_K(e))\in \tilde{S}'\cap M_\om$.

Define the projection $E'\in M_\om\oti L^\infty(\bhG\times\bhG^2)$
by
\[
E'=E(f^\per\oti1\oti1)+\ga_K(e).
\]
Since $T\subs (M^\om)^\ga$ and $e\in \tilde{S}'\cap M_\om$,
$E'\in (T'\cap M_\om)\oti L^\infty(\bhG\times\bhG^2)K$.
Then $E'$ satisfies (R1), (R3) and (R4) by \cite[Lemma 5.7]{M-T-CMP}.
We have to check $\hE'$ satisfies (S2) and (S3).
Set a projection $e'=(\id\oti\vph^2)(\ga_{K^2}^2(e))\in \tilde{S}'\cap M_\om$.
If we show $(e'\oti1)\ga_\pi^1(e')=0$
for each $\pi\in\meK^1\cdot\meK^1\setm\{\btr\}$,
then we are immediately done in view of \cite[Lemma 5.7]{M-T-CMP}.
This is verified as follows.
First we compute the following: for $\pi\in \meK^1\cdot\meK^1\setm\{\btr\}$
and $\rho\in\meK^2$,
\begin{align*}
(e\oti1\oti1)\ga_\orho^2(\ga_\pi^1(e'))
&=
(e\oti1\oti1)\ga_\orho^2(\ga_\pi^1((\id\oti\vph^2)(\ga_{K^2}^2(e))))
\\
&=
\sum_{\si\in\meK^2}
(\id\oti\id \oti\id \oti\vph_\si^2)
((e\oti1\oti1\oti1) \ga_\orho^2(\ga_\pi^1(\ga_\si^2(e))))
\\
&=
(\id\oti\id \oti\id \oti\vph_\rho^2)
((e\oti1\oti1\oti1) \ga_\orho^2(\ga_\pi^1(\ga_\rho^2(e))))
\\
&=0,
\end{align*}
where we have used the starting condition for $e$.
Using this, we get
\begin{align*}
(e'\oti1)\ga_\pi^1(e')
&=
((\id\oti\vph^2)(\ga_{K^2}^2(e))\oti1)\ga_\pi^1(e')
\\
&=
\sum_{\rho\in\meK^2}
d(\rho)^2(\Ph_\orho^{\ga^2}(e\oti1)\oti1)
\ga_\pi^1(e')
\\
&=
\sum_{\rho\in\meK^2}
d(\rho)^2(\Ph_\orho^{\ga^2}\oti\id)
((e\oti1\oti1)\ga_\orho^2(\ga_\pi^1(e')))
\\
&=0.
\end{align*}
Therefore $\hE'$ satisfies (S2) and (S3), which means $E'\in\meJ$.

Next we estimate $b_{E'}$ as follows:
\begin{align*}
b_{E'}-b_E
&=
(\ps\oti\tvph)(E'-E)
\\
&=
(\ps\oti\tvph)(-E(f\oti1\oti1)+\ga_K(e))
\\
&>
-(1-\de^{1/2})|f|_\ps
+|f|_\ps
=\de^{1/2}|f|_\ps. 
\end{align*}
Hence
\begin{equation}\label{eq: fb}
\de^{1/2}|f|_\ps<b_{E'}-b_E. 
\end{equation}

We check the inequalities in the statements. 
The first, the second and the fourth ones are derived
in a similar way to the proof of \cite[Lemma 5.11]{M-T-CMP}.
We only present a proof for the third one.
Since
\begin{align*}
&\ga_{F_2}^2(\hat{E}')_{132}-\hat{E}'\oti F_2
\\
&=
\ga_{F_2}^2(\hat{E})_{132}
(\ga_{F_2}^2(f)_{13}^\per-f^\per\oti1\oti F_2)
\\
&\quad
+
(\ga_{F_2}^2(\hat{E})_{132}-(\hE\oti F_2))
(f^\per\oti 1\oti F_2)
\\
&\quad
+
\ga_{K_1}^1
(
(\id\oti\id\oti\vph^2)
((\id\oti\De^2)(\ga^2(e))(1\oti F_2\oti K_2))
)\\
&\quad
-
\ga_{K_1}^1
(
(\id\oti\id\oti\vph^2)
((\id\oti\De^2)(\ga_{K_2}^2(e))(1\oti F_2\oti 1))
), 
\end{align*}
we have 
\begin{align}
&|\ga_{F_2}^2(\hat{E}')_{132}-\hat{E}'\oti F_2|_{\ps\oti\tvph}
\notag
\\
&\leq
|
\ga_{F_2}^2(\hat{E})_{132}
(\ga_{F_2}^2(f)_{13}^\per-f^\per\oti1\oti F_2)
|_{\ps\oti\tvph}
\label{al: e01}\\
&\quad
+
|
(\ga_{F_2}^2(\hat{E})_{132}-(\hE\oti F_2))
(f^\per\oti 1\oti F_2)
|_{\ps\oti\tvph}
\label{al: e02}\\
&\quad
+
|\ga_{K_1}^1
(
(\id\oti\id\oti\vph^2)
((\id\oti\De^2)(\ga_{K_2^\per}^2(e))(1\oti F_2\oti K_2))
)|_{\ps\oti\tvph}
\label{al: f1}\\
&\quad
+
|
\ga_{K_1}^1
(
(\id\oti\id\oti\vph^2)
((\id\oti\De^2)(\ga_{K_2}^2(e))(1\oti F_2\oti K_2^\per))
)
|_{\ps\oti\tvph}.
\label{al: f2}
\end{align}
Then we have the following estimates:
\[
(\ref{al: e02})\leq c_E,
\mbox{ and }(\ref{al: f1}),\ (\ref{al: f2})< \de_2 |F_2|_{\vph^2}|f|_\ps. 
\]
On (\ref{al: e01}), we have 
\begin{align*}
(\ref{al: e01})
&=
(\ps\oti\tvph)
(\ga_{F_2}^2(\hat{E})_{132}
|\ga_{F_2}^2(f)_{13}-f\oti1\oti F_2|
)
\\
&=
(\ps\oti\vph^2)
(\ga_{F_2}^2(
(\id\oti\tvph)(E)
|\ga_{F_2}^2(f)-f\oti F_2|
))
\\
&\leq
(\ps\oti\vph^2)
(
|\ga_{F_2}^2(f)-f\oti F_2|
)
\\
&\leq
(\ps\oti\vph^2)
(
|(\id\oti\vph^1\oti\id\oti\vph^2)
\circ\ga_{K_1}^1
(
(\id\oti{}_{F_2}\De_{K_2}^2)(\ga_{K_2^\per}^2(e)))|
)
\\
&\quad
+
(\ps\oti\vph^2)
(
|(\id\oti\vph^1\oti\id\oti\vph^2)
\circ\ga_{K_1}^1
(
(\id\oti {}_{F_2}\De_{K_2^\per})(\ga_{K_2}^2(e)))|
)
\\
&\leq
\de_2|K_1|_{\vph^1}|F_2|_{\vph^2}|K_2|_{\vph^2}|e|_\ps
+
\de_2|K_1|_{\vph^1}|F_2|_{\vph^2}|K_2|_{\vph^2}|e|_\ps
\\
&=
2\de_2|F_2|_{\vph^2}|f|_\ps. 
\end{align*}
By using (\ref{eq: fb}), we have 
\[
c_{E'}\leq c_E+4\de_2|f|_\ps< c_E+4\de_2^{1/2}(b_{E'}-b_E). 
\]
\end{proof}

Thus we obtain the following as \cite[Theorem 5.9]{M-T-CMP}.

\begin{lem}
Let $\ga=(\ga^1\oti\id)\circ\ga^2$,
$F$, $K$, $S$, $T$ and $\meC$ be as before.
Then the following statements hold:
\begin{enumerate}
\item $\ga$ has the Rohlin property;

\item
In the setting of Definition \ref{defn: j-Rohlin} for $\bhG\times\bhG^2$,
we can take a Rohlin projection $E$ from $(T'\cap M_\om)\oti\lhGG$
such that $\hE$ satisfies (S1), (S2) and (S3) for $\meC$
and
\begin{align*}
&|\ga_{F_2}^2(\hE)_{132}-\hE\oti F_2|_{\ps\oti\vph^1}
<
5\de_2^{1/2}|F_2|_{\vph^2};
\\
&|\ga_{F_1}^1(E)
-(\id\oti{}_{F_1}\De\oti \id)(E)|_{\ps\oti\vph^1\oti\tvph}
<
5\de_1^{1/2}|F_1|_{\vph^1}.
\end{align*}
\end{enumerate}
\end{lem}

Our main theorem in this subsection is the following.

\begin{thm}[Relative Rohlin theorem]
\label{thm: jRoh-G2-fix}
Let $M$ be a von Neumann algebra
and
$\ga=(\ga^1\oti\id)\circ\ga^2$ an action of
$\bhG\times\bhG^2$ on $M^\om$ such that
\begin{itemize}
\item
The $\bhG$-action $\ga^1$ on $M^\om$
commutes with
the $\bhG^2$-action $\ga^2$ on $M^\om$;

\item
The $\bhG\times \bhG^2$-action
$\ga:=(\ga^1\oti \id)\circ\ga^2$ is strongly free
and semi-liftable;

\item
$M_\om$ is globally invariant under $\ga$;

\item
$\ta^\om\circ \Ph_{(\pi,\rho)}^\ga
=\ta^\om\oti\tr_\pi\oti\tr_\rho$
on $M_\om\oti B(H_{(\pi,\rho)})$
for all $(\pi,\rho)\in \IG\times \Irr(\bG^2)$.
\end{itemize}
Then $\ga^1$ has the joint Rohlin property.
Moreover, for any countably generated von Neumann subalgebra
$T\subs (M^\om)^\ga$,
$\ga^1$ has a Rohlin projection
$E\in (T'\cap M_\om^{\ga^2})\oti\lhG$
satisfying (S1), (S2) and (S3).
\end{thm}
\begin{proof}
Let $F_i$, $K_i$ and $\de_i$ $(i=1,2)$ be given as before.
Take a Rohlin projection $E\in (T'\cap M_\om)\oti L^\infty(\bhG\times\bhG^2)$ 
supported over $K_1\oti K_2$ as in the previous lemma.
Then we have 
\begin{align}
&|\ga_{F_2}^2(\hE)_{132}-\hE\oti F_2|_{\ps\oti\vph^1}
\leq
5\de_2^{1/2}|F_2|_{\vph^2};
\label{eqn: ga-hE}\\
&|\ga_{F_1}^1(E)
-(\id\oti{}_{F_1}\De\oti \id)(E)|_{\ps\oti\vph^1\oti\tvph}
\leq
5\de_1^{1/2}|F_1|_{\vph^1}.
\label{eqn: ga1-E}
\end{align}

We set $\hE=(\id\oti\id\oti\vph^2)(E)$.
By (R3), $\hE$ gives a partition of unity
by matrix elements along with $K_1$.
We estimate the equivariance of $\hE$ with respect to $\ga^1$.
\begin{align*}
&|\ga_{F_1}^1(\hE)
-(\id\oti{}_{F_1}\De)(\hE)
|_{\ps\oti \vph^1\oti\vph^1}
\\
&=
|
(\id\oti\id\oti\id\oti\vph^2)
(
\ga_{F_1}^1(E)
-(\id\oti{}_{F_1}\De\oti\id)(E)
)
|_{\ps\oti \vph^1\oti\vph^1}
\\
&\leq
|
\ga_{F_1}^1(E)
-(\id\oti{}_{F_1}\De\oti\id)(E)
|_{\ps\oti \vph^1\oti\tvph}
\quad (\mbox{by Lemma }\ref{lem: P-Q})
\\
&leq
5\de_1^{1/2}|F_1|_{\vph^1}
\quad (\mbox{by}\ (\ref{eqn: ga1-E})).
\end{align*}

Take an increasing sequence
$\{F_2(n)\}_{n=1}^\infty\subset \Projf(Z(L^\infty(\bhG^2)))$ with
 $F_2(n)\rightarrow1$ strongly as $n\rightarrow \infty$.
Next we take $\de_2(n)>0$ such that
$\de_2(n)^{1/2}|F_2(n)|_{\vph^2}\to0$ as $n\to\infty$.
Take a sequence of Rohlin projections $\{E(n)\}_n$ satisfying
the above inequalities (\ref{eqn: ga-hE}) and (\ref{eqn: ga1-E})
for $F_2(n)$ and $\de_2(n)$.
By using the index selection trick \cite[Lemma 3.11]{M-T-CMP}
for $(\hat{E}(n))_n\in\el^\infty(\N,M_\om)$,
we obtain a Rohlin projection 
$E_1\in (T'\cap M_\om^{\ga^2})\oti \lhG$ supported on $K_1$ such that 
\[
|\ga_{F_1}^1(E_1)-(\id\oti_{F_1}\De)(E_1)
|_{\ps\oti\vph^1\oti\vph^1}
\leq 5\de_1^{1/2}.
\]
\end{proof}

\begin{cor}[Rohlin theorem]\label{cor: Rohlin}
Let $M$ be a von Neumann algebra
and $\ga$ an action of $\bhG$ on $M^\om$.
Assume the following:
\begin{itemize}
\item
$\ga$ is strongly free and semiliftable;

\item
$M_\om$ is globally invariant under $\ga$;

\item
$\ta^\om\circ\Ph_\pi^\ga=\ta^\om\oti\tr_\pi$
on $M_\om\oti B(H_\pi)$
for all $\pi\in\IG$.
\end{itemize}
Then
\begin{enumerate}
\item
$\ga$ has the joint Rohlin property;

\item
For any countably generated von Neumann subalgebra
$T\subs (M^\om)^\ga$,
$\ga$ has a Rohlin projection
$E\in (T'\cap M_\om)\oti\lhG$
satisfying (S1), (S2) and (S3);

\item
$\ga$ is stable on $M^\om$,
that is,
for any $\ga$-cocycle $v\in M^\om\oti\lhG$,
there exists a unitary $\mu\in M^\om$
such that $v=(\mu\oti1)\ga(\mu^*)$.
If $v \in M_\om\oti\lhG$, then $\mu$ can be taken from $M_\om$.
\end{enumerate}
\end{cor}
\begin{proof}
In the previous theorem, we put $\bhG^2=\{\btr\}$.
Then (1) and (2) hold.

(3)
Let $\{F_\nu\}_{\nu\in \N}\subs \lhG$ be an
increasing family of finitely supported central projections
such that $F_\nu\to1$ strongly as $\nu\to\infty$.
For each $\nu$, take an $(F_\nu,1/\nu)$-invariant finite projection
$K_\nu\in \lhG$ with $e_\btr\leq K_\nu$.
Let $E^\nu\in M_\om\oti\lhG$ be a Rohlin projection
satisfying (S1), (S2) and (S3) for a faithful state
$\ps=\ph\circ\ta^\om\in (M^\om)_*$,
$F_\nu$, $K_\nu$ and $1/\nu^2$.
Then we get the Shapiro unitary
$\mu^\nu=(\id\oti\vph)(vE^\nu)$, and we have
\begin{align*}
&\hspace{16pt}
v\ga_{F_\nu}(\mu^\nu)-\mu^\nu\oti F_\nu
\\
&=
(\id\oti\id\oti\vph)((v\oti1)\ga_{F_\nu}(vE^\nu))
-
(\id\oti\id\oti\vph)((\id\oti{}_{F_\nu}\De)(vE^\nu))
\\
&=
(\id\oti\id\oti\vph)((\id\oti{}_{F_\nu}\De)(v)\ga_{F_\nu}(E^\nu))
-
(\id\oti\id\oti\vph)((\id\oti{}_{F_\nu}\De)(vE^\nu))
\\
&=
(\id\oti\id\oti\vph)((\id\oti{}_{F_\nu}\De)(v)
(\ga_{F_\nu}(E^\nu)-(\id\oti{}_{F_\nu}\De)(E^\nu))).
\end{align*}
Since the element $\ga_{F_\nu}(E^\nu)-(\id\oti{}_{F_\nu}\De)(E^\nu))$
is in the centralizer of $\ps\oti\vph\oti\vph$,
we can use Lemma \ref{lem: P-Q}, and we get
\begin{align*}
|v\ga_{F_\nu}(\mu^\nu)-\mu^\nu\oti F_\nu|_{\ps\oti\vph}
&\leq
|\ga_{F_\nu}(E^\nu)-(\id\oti{}_{F_\nu}\De)(E^\nu)|_{\ps\oti\vph\oti\vph}
\\
&\leq 5/\nu.
\end{align*}
By using the index selection map for $(\mu^\nu)_\nu\in \el^\infty(\N,M^\om)$,
we get $\mu\in M^\om$ such that $v\ga(\mu)=\mu\oti1$.
When $v$ is evaluated in $M_\om$,
each $\mu^\nu$ is in $M_\om$,
and so is $\mu$ by the property of the index selection map.
\end{proof}

\begin{cor}
\label{cor: gamma-th Rohlin}
Let $M$ be a von Neumann algebra,
$\gamma$ an action of $\bhG$ on $M^\omega$
and $\theta\in\Aut(M^\om)$.
Regard $\th$ as an action of $\Z$ on $M^\om$.
Assume the following:
\begin{itemize}
\item
$\th$ commutes with $\ga$;

\item The $\bhG\times\Z$-action $\ga\circ\th$
is strongly free and semi-liftable;

\item $M_\om$ is globally invariant under $\ga\circ\th$;

\item
$\ta^\om\circ\th=\ta^\om$ on $M_\om$
and
$\ta^\om\circ\Ph_\pi^\ga=\ta^\om\oti\tr_\pi$
on $M_\om\oti B(H_\pi)$
for all $\pi\in\IG$.
\end{itemize}
Then for any $n>0$ and any countably generated von Neumann subalgebra
$T\subs (M^\om)^{\ga\th}$,
there exists a partition of unity 
$\{E_i\}_{i=0}^{n-1}\subset T'\cap M_\omega^\gamma$
such that 
$\theta(E_i)=E_{i+1}$, where $E_n=E_0$.
\end{cor}
\begin{proof}
For $m>0$, set $\delta_m=2/nm$ and  $\meK_m:=\{0,1,2,nm-1\}$. 
Then $\meK_m$ is $(\{1\},\delta_m)$-invariant subset of $\Z$. 
By Theorem \ref{thm: jRoh-G2-fix},
we have a partition of unity $\{F^m_i\}_{i\in \meK_m}$ in 
$M_\omega^\gamma$ such that
$\sum_{i=0}^{nm-2}|\theta(F_i^m)-F_{i+1}^m|_\ps
\leq 5\delta_m^{\frac{1}{2}}$.
For $0\leq i\leq n-1$, set $E^m_i:=\sum_{k=0}^{m-1}F_{kn+i}$. 
Then for $0\leq i\leq n-2 $, we have 
\[
|\theta(E_i^m)-E_{i+1}^m|_\ps\leq 
\sum_{k=0}^{m-1}|\theta(F_{kn+i})-F_{kn+i+1}|_\psi\leq 5
\delta_m^{\frac{1}{2}}.
\] 
By applying the index selection trick to $\{E_i^m\}_{m=1}^\infty$,
$0\leq i \leq n-1$,
we get $\theta(E_i)=E_{i+1}$ for $0\leq i\leq n-2$.
Then $\theta(E_{n-1})=E_0$ follows automatically.
\end{proof}

Recall the following result \cite[Lemma 4.3]{M-T-CMP}.
The statement is slightly strengthened here, but the same proof
is applicable if we replace $M_\om$ with $A'\cap M_\om$.
Note that $A'\cap M_\om$ is of type II$_1$ for any countably
generated von Neumann subalgebra $A\subs M^\om$
when $M_\om$ is of type II$_1$.

\begin{thm}[2-cohomology vanishing]
\label{thm: 2-coho-vanish}
Let $M$ be a von Neumann algebra such that
$M_\om$ is of type II$_1$.
Let $A\subs M^\om$ be a countably generated von Neumann subalgebra.
Let $(\ga,w)$ be a cocycle action of $\bhG$ on $M^\om$.
Assume the following:
\begin{itemize}
\item $A'\cap M_\om$ is globally invariant under $\ga$;

\item $w\in (A'\cap M_\om)\oti\lhG\oti\lhG$;

\item $\ga$ is of the form $\ga=\Ad U\circ\be$,
where $U\in U(M^\om\oti \lhG)$ and $\be\in \Mor(M^\om,M^\om\oti\lhG)$
is semi-liftable.
\end{itemize}
Then the 2-cocycle $w$ is a coboundary in $A'\cap M_\om$.
\end{thm}

\begin{cor}\label{cor: fixed-typeII}
Let $\ga$ be a strongly free and semi-liftable action of $\bhG$ 
on $M_\om$.
Let $S\subs (M^\om)^\ga$ is countably generated von Neumann 
subalgebra.
If $M_\om$ is of type II$_1$,
then the von Neumann algebra $S'\cap M_\om^\ga$ is also of type II${}_1$.
\end{cor}
\begin{proof}
Let $I$ be a finite index set. 
Since $S'\cap M_\om$ is of type II${}_1$, 
we can take a system of matrix units 
$\{e_{i,j}\}_{i,j\in I}$ in $S'\cap M_\om$. 
Let $\pi\in\IG$. 
Let $Q$ be a finite dimensional subfactor generated 
by $\{e_{i,j}\}_{i,j\in I}$. 
Take an index $i_0\in I$. 
Since $\{\ga_\pi(e_{i,j})\}_{i,j\in I}$ and 
$\{e_{i,j}\oti1_\pi\}_{i,j\in I}$ are systems of matrix units 
in $(S'\cap M_\om)\oti B(H_\pi)$, 
$e_{i_0,i_0}\oti1_\pi$ and $\ga_\pi(e_{i_0,i_0})$ 
are equivalent. 
Hence there exists $v_\pi\in (S'\cap M_\om)\oti B(H_\pi)$ 
such that 
$e_{i_0,i_0}\oti1_\pi=v_\pi v_\pi^*$ 
and $\ga_\pi(e_{i_0,i_0})=v_\pi^* v_\pi$. 
Set the unitary
\[
\tilde{v}_\pi
=\sum_{i\in I}(e_{i,i_0}\oti1_\pi)v_\pi\ga_\pi(e_{i_0,i}). 
\]
Then $\tilde{v}_\pi$ is satisfying 
$\tilde{v}_\pi\ga_\pi(e_{i,j})\tilde{v}_\pi^*=e_{i,j}\oti1$. 
Setting $\tilde{v}=(\tilde{v}_\pi)_\pi\in M_\om\oti \lhG$, 
we have 
\[
\tilde{v}\ga(x)\tilde{v}^*=x\oti1 
\quad\mbox{for all}
\ x\in Q. 
\]
Hence the map $\Ad\tilde{v}\circ \ga$ is a cocycle action 
on $Q'\cap (S'\cap M_\om)$. 
Using the previous 2-cohomology vanishing result for $Q'\cap (S'\cap M_\om)$, 
we obtain a unitary
$w\in (Q'\cap (S'\cap M_\om))\oti\lhG$
such that
$w\tilde{v}$ is an $\ga$-cocycle.
Now we have 
\[
w\tilde{v}\ga(x)\tilde{v}^* w^*=x\oti1 
\quad\mbox{for all}
\ x\in Q. 
\]
Since $\ga$ has the joint Rohlin property,
the action $\ga|_{M_\om}$ is stable
by Corollary \ref{cor: Rohlin}.
Hence the $M_\om$-valued 
$\ga$-cocycle $w\tilde{v}$ is of the form
$w\tilde{v}=(\nu^*\oti1)\ga(\nu)$
where $\nu\in U(S'\cap M_\om)$.
This implies that a subfactor $\nu Q\nu^*$ is fixed by $\ga$. 
Hence $S'\cap M_\om^\ga$ contains a subfactor
with arbitrary finite dimension,
and it is of type II${}_1$. 
\end{proof}

\subsection{Approximately inner actions}\label{sec:appr}
Let $M$ be a von Neumann algebra,
$\bhG$ an amenable discrete Kac algebra
and $\Ga$ a discrete amenable group with the neutral element $e$.
In this subsection, we study the following situation:

\begin{itemize}
\item We are given two actions
$\al\in\Mor(M,M\oti\lhG)$, $\th\col \Ga\ra \Aut(M)$
and unitaries $(v_g)_{g\in\Ga}\in U(M\oti\lhG)$
such that
\[
(\th_g\oti\id)\circ \al\circ \th_g^{-1}=\Ad v_g^* \circ\al;
\]

\item
$M_\om$ is of type II$_1$ and $Z(M)\subs M^\th$;

\item
$(v_g)_{g\in\Ga}$ is a $(\th\oti\id)$-cocycle;

\item $v_g^*$ is an $\al$-cocycle for each $g\in\Ga$;

\item $\al$ is approximately inner;

\item
$\al_\pi \th_g$
is properly centrally non-trivial
for each $(\pi,g)\in \IG\times\Ga\setm(\btr,e)$.
\end{itemize}

Take $U_\pi^\nu\in U(M\oti B(H_\pi))$ for $\nu\in\N$
such that $\Ad U_\pi^\nu$
converges to $\al_\pi$ for each $\pi\in\IG$.
Set $U:=(U_\pi)_\pi\in M^\om\oti\lhG$
where $U_\pi:=(U_\pi^\nu)_\nu\in M^\om\oti B(H_\pi)$.
Then $\al=\Ad U$ on $M$.
Our first task is to replace $U$ with a new one which well behaves
to the action $\th^\om$.

\begin{lem}
For each $\pi\in\Irr(\bhG)$ and $g\in\Ga$,
the sequence
$(v_g(\th_g\oti\id)(U_\pi^\nu))_\nu$ approximates $\al_\pi$.
In particular, $U^* v_g (\th_g^\om\oti\id)(U)\in M_\om\oti \lhG$.
\end{lem}
\begin{proof}
Take $\ph\in M_*$.
We verify that $(\ph\oti\tr_\pi)\circ \Ad (\th_g\oti\id)((U_\pi^\nu)^*)v_g^*$
converges to $\ph\circ \Ph_\pi^\al$ as follows:
\begin{align*}
&\lim_{\nu\to\infty}
(\ph\oti\tr_\pi)\circ \Ad (\th_g\oti\id)((U_\pi^\nu)^*)v_g^*
\\
&=
\lim_{\nu\to\infty}
\big{(}(\th_g^{-1}\oti\id)(v_g)U_\pi^\nu(\ph\circ\th_g\oti\tr_\pi)
(U_\pi^\nu)^*(\th_g^{-1}\oti\id)(v_g^*)
\big{)}\circ (\th_g^{-1}\oti\id)
\\
&=
\big{(}(\th_g^{-1}\oti\id)(v_g)
(\ph\circ\th_g\circ\Ph_\pi^\al)(\th_g^{-1}\oti\id)(v_g^*)
\big{)}\circ (\th_g^{-1}\oti\id)
\\
&=
v_g
(\ph\circ\th_g\circ\Ph_\pi^\al\circ(\th_g^{-1}\oti\id))
v_g^*
=
v_g
(\ph\circ\Ph_\pi^{\Ad v_g^*\circ \al})
v_g^*
\\
&=
v_g
(\ph\circ\Ph_\pi^{\al}\circ\Ad v_g)
v_g^*
=\ph\circ\Ph_\pi^{\al}.
\end{align*}
The latter statement follows from
\cite[Lemma 3.6]{M-T-CMP}.
\end{proof}

\begin{lem}
There exists $u\in U(M_\om\oti\lhG)$ such that
$v_g(\th_g^\om\oti\id)(Uu)=Uu$. 
\end{lem}
\begin{proof}
Since the $\Ga$-action $\th^\om$ is strongly free,
it has the joint Rohlin property.
Let $S\subs M^\om$ be a von Neumann subalgebra generated by
all matrix entries of $(\th_g^\om\oti\id)(U)$ and $v_g$
for all $g\in\Ga$.
Let $F\subs \Ga$ be a finite subset and $\de>0$.
Since $\Ga$ is amenable, there exists a finite subset $K\subs \Ga$
such that $\sum_{g\in F}|gK\De K|<\de|F||K|$.
Fix a faithful state $\ph\in M_*$ and set $\ps:=\ph\circ\ta^\om$.
Take a Rohlin projection
$(E_g)_{g\in\Ga}\subs (S'\cap M_\om)$
such that
$E_g=0$ for $g\nin K$ and
$
\sum_{\el\in \Ga}|\th_g^\om(E_\el)-E_{g\el}|_{\ps}
\leq
5\de^{1/2}$.
Define $u\in M^\om\oti \lhG$ by
$u=\sum_{k\in K}U^* v_k(\th_{k}^\om\oti\id)(U)(E_{k}\oti1)$.
By the previous lemma, $u$ is in $M_\om\oti \lhG$.
Then it is easy to see that $u$ is a unitary element,
and for $g\in F$ we have
\begin{align*}
&U^*v_g(\th_g^\om\oti\id)(U)\cdot(\th_g^\om\oti\id)(u)
\\
&=
U^*v_g(\th_g^\om\oti\id)(U)
\sum_{k\in K}
(\th_g^\om\oti\id)(U^*) (\th_g\oti\id)(v_k)
(\th_{gk}^\om\oti\id)(U)(\th_g^\om(E_{k})\oti1)
\\
&=
\sum_{k\in K}
U^*
v_{gk}(\th_{gk}^\om\oti\id)(U)(\th_g^\om(E_{k})\oti1)
\\
&=
\sum_{k\in K}
U^*
v_{gk}(\th_{gk}^\om\oti\id)(U)((\th_g^\om(E_{k})-E_{gk})\oti1)
+
\sum_{\el\in gK}
U^*
v_{\el}(\th_{\el}^\om\oti\id)(U)(E_{\el}\oti1).
\end{align*}

Take a partial isometry $w_g\in M^\om\oti \lhG$
such that
\[
|U^*v_g(\th_g^\om\oti\id)(U)\cdot(\th_g^\om\oti\id)(u)-u|
=w_g^*(U^*v_g(\th_g^\om\oti\id)(U)(\th_g^\om\oti\id)(u)-u).
\]
Let $\chi\in \lhG_*$ be a faithful state.
Then we have
\begin{align}
&|U^*v_g(\th_g^\om\oti\id)(U)\cdot(\th_g^\om\oti\id)(u)-u|_{\ps\oti\chi}
\notag\\
&=
\sum_{k\in K}
(\ps\oti\chi)
\big{(}
w_g^*U^*
v_{gk}(\th_{gk}^\om\oti\id)(U)((\th_g^\om(E_{k})-E_{gk})\oti1)
\big{)}
\label{eq: gk}
\\
&\quad
-
\sum_{\el\in K\setm gK}
(\ps\oti\chi)
\big{(}
w_g^*U^*
v_{\el}(\th_{\el}^\om\oti\id)(U)(E_{\el}\oti1)
\big{)}.
\label{eq: K-gK}
\end{align}
Since $E_k\in (M_\om)_\ps$,
we can use Lemma \ref{lem: P-Q}, and we have
\[
|(\ref{eq: gk})|
\leq
\sum_{k\in K}
|\th_g^\om(E_{k})\oti1-E_{gk}\oti1|_{\ps\oti\chi}
\leq 5\de^{1/2}.
\]
By the assumption of $K$,
we have
\begin{align*}
|(\ref{eq: K-gK})|
&\leq
\sum_{\el\in K\setm gK}|E_\el|_\ps
\leq
\sum_{\el\in K\De gK}|E_\el|_\ps
\\
&=
\bigg{|}\sum_{\el\in gK}E_\el-\sum_{k\in K}E_k
\bigg{|}_\ps
=
\bigg{|}\sum_{k\in K}E_{gk}-1\bigg{|}_\ps
\\
&=
\bigg{|}\sum_{k\in K}(E_{gk}-\th_g^\om(E_k))\bigg{|}_\ps
\leq
\sum_{k\in \Ga}\big{|}E_{gk}-\th_g^\om(E_k)\big{|}_\ps
\leq
5\de^{1/2}.
\end{align*}
Hence we have
\begin{equation}\label{eq: th-u}
|U^*v_g(\th_g^\om\oti\id)(U)\cdot(\th_g^\om\oti\id)(u)-u|_{\ps\oti\chi}
\leq 10\de^{1/2}.
\end{equation}
For each $\nu\in \N$, take $u_\nu\in M_\om\oti \lhG$ satisfying
(\ref{eq: th-u}) for $\de=1/\nu$.
Take an increasing sequence $F_\nu\Subs \Ga$
with $\bigcup_{\nu=1}^\infty F_\nu=\Ga$.
Applying the index selection trick to $(u^\nu)_\nu$,
we get $u\in M_\om\oti \lhG$ with
$U^*v_g(\th_g^\om\oti\id)(Uu)=u$ for all $g\in\Ga$.
\end{proof}

Replacing $U$ with $Uu$,
we may assume that $U=(U^\nu)_\nu$ also satisfies
\[
v_g(\th_g^\om\oti\id)(U)=U.
\]
As in \cite{M-T-CMP},
we consider the cocycle actions $\ga^{-1}=\Ad U^*\circ \al^\om$
and $\ga^{0}=\Ad U^*$ on $M_\om$.
Their 2-cocycles $w^{-1}$ and $w^0$ are given by
\[
w^{-1}=(U^*\oti1)\al^\om(U^*)(\id\oti\De)(U),
\quad
w^0=(U^*\oti1)U_{13}^*(\id\oti\De^\opp)(U).
\]
Here note that $\ga^{-1}$ and $\ga^0$ are cocycle actions
of $\bhG$ and $\bhG^\opp$, respectively.

\begin{lem}\label{lem: ga1-th}
In the above setting,
$\ga^{-1}$ and $\ga^0$ are cocycle actions on $M_\om^{\th_\om}$.
\end{lem}
\begin{proof}
First, we show that $\ga^{-1}$ and $\ga^0$ commute with $\th_\om$.
Using $v_g(\th_g^\om\oti\id)(U)=U$,
we have $(\th_g^\om\oti\id)\circ\ga^{-1}=\ga^{-1}\circ\th_g^\om$.
This equality holds on $M^\om$, and in particular,
$\ga^{-1}$ commutes with $\th_\om$ on $M_\om$.

Let $x\in M_\om$.
Since $v_g$ commutes with $\th_g^\om(x)\oti1$,
we have 
\[
(\th_g^\om\oti\id)(\ga^0(x))=U^* v_g(\th_g^\om(x)\oti1)v_g^* U
=
U^*(\th_g^\om(x)\oti1)U=\ga^0(\th_g^\om(x)).
\]
Hence $\ga^0$ also commutes with $\th^\om$.

Second, we check that the 2-cocycles $w^{-1}$ and $w^0$ are evaluated
in $M_\om^{\th_\om}$.
Since $v_g^*$ is an $\al$-cocycle,
we have
\begin{align*}
(\th_g^\om\oti\id\oti\id)(w^{-1})
&=
((\th_g^\om\oti\id)(U^*)\oti1)\cdot
(\th_g^\om\oti\id\oti\id)(\al^\om(U^*))
\cdot(\th_g^\om\oti\De)(U)
\\
&=
(U^*\oti1)(v_g\oti1)
\cdot
(v_g^*\oti1)(\al^\om((\th_g\oti\id)(U^*)))(v_g\oti1)
\\
&\quad\cdot(\id\oti\De)(v_g^*U)
\\
&=
(U^*\oti1)\al^\om(U^*v_g)(v_g\oti1)(\id\oti\De)(v_g^*)(\id\oti\De)(U)
\\
&=
w^{-1},
\end{align*}
and
\begin{align*}
(\th_g^\om\oti\id\oti\id)(w^{0})
&=
(\th_g^\om\oti\id)(U^*\oti1)
\cdot(\th_g^\om\oti\id)(U^*)_{13}\cdot(\th_g^\om\oti\De^\opp)(U)
\\
&=
(U^*\oti1)(v_g\oti1)
\cdot
U_{13}^*(v_g)_{13}
\cdot
(\id\oti\De^\opp)(v_g^* U)
\\
&=
(U^*\oti1)U_{13}^*\al(v_g)_{132}(v_g)_{13}(\id\oti\De)(v_g^*)_{132}
(\id\oti\De^\opp)(U)
\\
&=
w^0.
\end{align*}
\end{proof}

Define the cocycle action $\ga$ of $\bhG\times\bhG^\opp$ on $M_\om$
by $\ga:=(\ga^{-1}\oti\id)\circ \ga^0$.
Its 2-cocycle $w$ is given by
\[
w:=U_{12}^*\al^\om(U^*)_{123}\al^\om(U_{12}^*\al^\om(U^*))_{1245}
(\id\oti\De_{\bhG\times\bhG^\opp})(\al^\om(U)U_{12}).
\]
By direct computation, we have
\[
w=\ga^{-1}(w_{123}^0 (w_{132}^{-1})^*)_{1234}w_{124}^{-1}
(\id\oti\De\oti\id\oti\id)(\ga^{-1}(w^0))_{12435}.
\]
Hence $w$ is evaluated in $M_\om^{\th_\om}$,
that is, $\ga$ is a cocycle action on $M_\om^{\th_\om}$.

Then we apply Theorem \ref{thm: 2-coho-vanish} to $\ga$
and get $c\in M_\om^{\th_\om}\oti\lhGG$ such that
\[
c_{123}\ga(c)w(\id\oti\De_{\bhGG})(c^*)=1.
\]
Here we note that the proof of \cite[Lemma 4.3]{M-T-CMP}
works in our case by replacing $M_\om$ with $M_\om^{\th_\om}$.
Also note that $M_\om^{\th_\om}$ is of type II${}_1$.

Set the unitaries $c^\el:=c_{\cdot\oti\btr}$ and $c^r:=c_{\btr\oti\cdot}$
in $M_\om^{\th_\om}\oti\lhG$.
Then the proof similar to that of \cite[Lemma 4.6]{M-T-CMP}
shows that
\begin{itemize}
\item 
$c^\el U^*$ is an $\al^\om$-cocycle;

\item
$U(c^r)^*$ is a unitary representation of $\bhG$;

\item
$U(c^r)^*$ is fixed by the perturbed action
$\Ad (c^\el U^*)\circ \al^\om$.
\end{itemize}

Exchanging $U$ with $U(c^r)^*$, we obtain the following.

\begin{lem}\label{lem: U-c-app}
Let $\al$, $\th$ and $(v_g)_{g\in\Ga}$ be as before.
Then there exists $U\in U(M^\om\oti\lhG)$ and $c\in U(M_\om\oti\lhG)$
such that
\begin{enumerate}

\item
$(\Ad U_\pi^\nu)_\nu$ approximates $\al_\pi$ for all $\pi\in \IG$;

\item
$U$ is a unitary representation of $\bhG$, that is,
we have $(\id\oti\De)(U)=U_{12}U_{13}$;

\item 
$cU^*$ is an $\al^\om$-cocycle;

\item 
$U$ is fixed by the perturbed action $\Ad cU^*\circ \al^\om$;

\item 
$v_g(\th_g^\om\oti\id)(U)=U$ and $(\th_g^\om\oti\id)(c)=c$ for all $g\in\Ga$.
\end{enumerate}
\end{lem}

Now we set the following maps on $M^\om$ as before:
\[
\ga^1:=\Ad cU^* \circ \al^\om,
\quad\ga^2:=\Ad U^* (\cdot\oti1),
\]
which are actions of $\bhG$ and $\bhG^\opp$, respectively.
They preserve $M_\om$ and $M_\om^{\th_\om}$.

\begin{lem}\label{lem: vU-phi}
In the above settings, one has the following:
\begin{enumerate}
\item
$v_g^*U$ is a unitary representation of $\bhG$;

\item
For all $\pi\in\IG$ and $X\in M^\om\oti B(H_\pi)$,
$\Ph_\pi^{\ga^2}(X)
=(\id\oti\tr_\pi)(UXU^*)$.
\end{enumerate}
\end{lem}
\begin{proof}
(1)
Since $v_g^*$ is an $\al$-cocycle, we have
\[
(v_g^*U)_{12}(v_g^*U)_{13}
=(v_g^*)_{12}\al(v_g^*)U_{12}U_{13}
=(\id\oti\De)(v_g^*U).
\]

(2)
Let $S_{\opi,\pi}$ be an isometric
intertwiner from $\btr$ into $\opi\otimes\pi$
for $\bhG^\opp$.
For $X\in M^\om\oti B(H_\pi)$,
we have
\begin{align*}
\Ph_\pi^{\ga^2}(X)
&=
(1\oti S_{\opi,\pi}^*)
(U^*)_{12}
X_{13}
(U)_{12}
(1\oti S_{\opi,\pi})
\\
&=
(1\oti S_{\opi,\pi}^*)
(\id\oti\De^\opp)(U^*)
U_{13}
X_{13}
U_{13}^*
\\
&
\hspace{16pt}\cdot(\id\oti\De^\opp)(U)
(1\oti S_{\opi,\pi})
\\
&=
(1\oti S_{\opi,\pi}^*)
U_{13}
X_{13}
U_{13}^*
(1\oti S_{\opi,\pi})
\\
&=
(\id\oti \tr_\pi)(U X U^* ).
\end{align*}
\end{proof}

Our next aim is to replace $U$ with a new one
such that we can retake $c=1$.
\begin{lem}\label{lem: U-cocycle}
There exists $z\in U(M_\om^{\th_\om})$ such that
$UcU^*=(z\oti1)\al^\om(z^*)$.
\end{lem}
\begin{proof}
By definition of $\ga^1$, we have
$\Ph_{(\pi,g)}^{\ga^1\circ\th^\om}
=\th_{g^{-1}}^\om\circ\Ph_\pi^{\al^\om}\circ \Ad Uc^*$.
Since $\th_g\circ\ta^\om=\ta^\om$ on $M_\om$,
we get
$\ta^\om\circ \Ph_{(\pi,g)}^{\ga^1\circ\th^\om}=\ta^\om\oti\tr_\pi$
on $M_\om\oti B(H_\pi)$ for all $(\pi,g)\in\IG\times\Ga$.
By Lemma \ref{lem: ga1-th},
$\ga^1\circ\th^\om$ is a $\bhG\times \Ga$-action.
It is easy to see that $\ga^1\circ\th^\om$
is strongly free.
Since $\Ad (c^\nu U^{\nu*})\circ \al$ converges to
the trivial action, $\ga^1$ is semiliftable.
Hence $\ga^1\circ\th^\om$ has the joint Rohlin property.

Now we have two $\ga^1$-cocycles $Uc^*$ and $U$.
Let $K\in\Projf(Z(\lhG))$ be an $(F,\de)$-invariant projection
with $K\geq e_\btr$.
By Theorem \ref{thm: jRoh-G2-fix},
we can take a Rohlin projection $E\in M_\om^{\th_\om}\oti\lhG K$
for $\meC=\{U, Uc^*\}$.
Set the Shapiro unitaries
$\mu^\de:=(\id\oti\vph)(UE)$ and $\nu^\de:=(\id\oti\vph)(Uc^*E)$.
Then we claim the following:
\\

\noindent\textbf{Claim 1.}
\[
\mu^\de \nu^{\de*}=(\id\oti\vph)(UEcU^*),\quad
\mu^\de \nu^{\de*}\in M_\om^{\th_\om}.
\]
Indeed, the first equality is shown by using (R3)
and $\vph=\oplus_{\pi\in\IG}d(\pi)\Tr_\pi$.
Next we show that
$\mu^\de \nu^{\de*}\in M_\om^{\th_\om}$.
By Lemma \ref{lem: vU-phi},
we have
\begin{align*}
\mu^\de \nu^{\de*}
&=
(\id\oti\vph)(UEcU^*)
=\sum_{\pi\in\IG}
d(\pi)^2 (\id\oti\tr_\pi)(UEcU^*)
\\
&=
\sum_{\pi\in\IG}
d(\pi)^2 \Ph_\pi^{\ga^2}(Ec).
\end{align*}
Since $Ec\in (M_\om)^{\th_\om}\oti B(H_\pi)$,
$\mu^\de \nu^{\de*}$ is in $M_\om$.
Using the commutativity of $\ga^2|_{M_\om}$ and $\th_\om$,
we have
\begin{align*}
\th_g^\om(\mu^\de \nu^{\de*})
&=
\sum_{\pi\in\IG}
d(\pi)^2
\th_g^\om
\big{(}
\Ph_\pi^{\ga^2}(Ec)
\big{)}
=
\sum_{\pi\in\IG}
d(\pi)^2
\Ph_\pi^{\ga^2}((\th_g^\om\oti\id)(Ec))
\\
&=
\sum_{\pi\in\IG}
d(\pi)^2
\Ph_\pi^{\ga^2}(Ec)
=
\mu^\de \nu^{\de*}.
\end{align*}
\\

Next we claim the following:
\\

\noindent\textbf{Claim 2.}
\begin{equation}\label{ineq: UE}
|
U\ga_F^1(\mu^\de)-\mu^\de\oti F
|_{\ps\oti\vph}
\leq5\de^{1/2};
\end{equation}
\begin{equation}\label{ineq: Uc*E}
|
Uc^*\ga_F^1(\nu^\de)-\nu^\de\oti F
|_{\ps\oti\vph}
\leq5\de^{1/2}. 
\end{equation}

Let 
$U\ga_F^1(\mu^\de)-\mu^\de\oti F
=
v|U\ga_F^1(\mu^\de)-\mu^\de\oti F|$ 
be the polar decomposition. 
Then we have 
\begin{align*}
|U\ga_F^1(\mu^\de)-\mu^\de\oti F|
&=
v^*(U\ga_F^1(\mu^\de)-\mu^\de\oti F)
\\
&=
v^*(\id\oti\id\oti\vph)(U_{12}U_{13}\ga^1(E))\\
&\quad
-
v^*(\id\oti\id\oti\vph)(U_{12}U_{13}(\id\oti{}_F\De)(E)))
\\
&=
v^* (\id\oti\id\oti\vph)(U_{12}U_{13}(\ga^1(E)-(\id\oti{}_F\De)(E))). 
\end{align*}
Using Lemma \ref{lem: P-Q}, we have 
\begin{align*}
|U\ga_F^1(\mu^\de)-\mu^\de\oti F|_{\ps\oti\vph}
&=
|
(\id\oti\id\oti\vph)
(v_{12}^*U_{12}U_{13}(\ga^1(E)-(\id\oti{}_F\De)(E)))|_{\ps\oti\vph}
\\
&\leq
|\ga^1(E)-(\id\oti{}_F\De)(E)|_{\ps\oti\vph\oti\vph}
\leq5\de^{1/2}. 
\end{align*}
Similarly we can prove (\ref{ineq: Uc*E}).
\\

Now we use the index selection trick. 
For decreasing $\de_n=1/n \to 0$ and increasing finite rank central
projections $F_n\to1$ in $\lhG$ as $\N\ni n\to \infty$,
we take $\mu(n):=\mu^{1/n}$ and $\nu(n):=\nu^{1/n}$ in $U(M^\om)$
for $n\in\N$.
Set $\tilde{\mu}=(\mu(n))_n$ and $\tilde{\nu}=(\nu(n))_n$. 
From them,
we construct $\mu$ and $\nu$ in $U(M^\om)$ by index selection.
Since $\mu(n)\nu(n)^*\in M_\om^{\th_\om}$, $\mu\nu^*\in M_\om^{\th_\om}$.
By definition of an index selection map (i.e. it commutes with $\ga^1$),
we have
$U\ga^1(\mu)=\mu\oti1$ and $U c^*\ga^1(\nu)=\nu\oti1$.
These imply
\[
\al^\om(\nu\mu^*)
=
Uc^*\ga^1(\nu\mu^*)cU^*
=
(\nu\mu^*\oti1)UcU^*.
\]
Therefore, $z:=\mu\nu^*$ is a desired solution.
\end{proof}

By the previous lemma,
we get $z\in U(M_\om^{\th_\om})$ such that
$UcU^*=(z\oti1)\al^\om(z^*)$.
Then we consider $V=(z^*\oti1)U(z\oti1)$,
which is a representation of $\bhG$ in $M^\om$.
By the previous lemma,
we have
\[
V^*
=
(z^*\oti1)cU^* \cdot Uc^*U^*(z\oti1)
=
(z^*\oti1)cU^* \al^\om(z).
\]
Since $cU^*$ is an $\al^\om$-cocycle,
so is $V^*$.
Moreover we have
\[
v_g(\th_g\oti\id)(V)
=
v_g(z^*\oti1)v_g^*U(z\oti1)
=
(z^*\oti1)U(z\oti1)=V.
\]
Finally we again replace $U$ with $V=(z^*\oti1)U(z\oti1)$,
and we get the following.

\begin{thm}\label{thm: app-cocycle}
Let $M$ be a von Neumann algebra.
Assume the following:
\begin{itemize}
\item We are given two actions
$\al\in\Mor(M,M\oti\lhG)$, $\th\col \Ga\ra \Aut(M)$
and unitaries $(v_g)_{g\in\Ga}\in U(M\oti\lhG)$
such that
\[
(\th_g\oti\id)\circ \al\circ \th_g^{-1}=\Ad v_g^* \circ\al;
\]

\item
$M_\om$ is of type II$_1$ and $Z(M)\subs M^\th$;

\item
$(v_g)_{g\in\Ga}$ is a $(\th\oti\id)$-cocycle;

\item $v_g^*$ is an $\al$-cocycle for each $g\in\Ga$;

\item $\al$ is approximately inner;

\item
$\al_\pi \th_g$
is properly centrally non-trivial
for each $(\pi,g)\in \IG\times\Ga\setm(\btr,e)$.
\end{itemize}

Then there exists $U=(U^\nu)_\nu\in U(M^\om\oti\lhG)$ 
such that 
\begin{enumerate}
\item 
$(\Ad U_\pi^\nu)_\nu$ converges to $\al_\pi$ for all $\pi\in\IG$;

\item 
$U$ is a representation of $\bhG$
that is, 
$(\id\oti\De)(U)=U_{12}U_{13}$;

\item 
$U^*$ is an $\al^\om$-cocycle,
that is, 
$U_{12}^* \al^\om(U^*)=(\id\oti\De)(U^*)$;

\item
$v_g(\th_g^\om\oti\id)(U)=U$ for all $g\in\Ga$.
\end{enumerate}
\end{thm}

\begin{cor}\label{cor: wU-cocycle}
Let $w\in M^\om\oti\lhG$ be an $\al^\om$-cocycle.
Take $U\in U(M^\om\oti\lhG)$ as in the previous theorem.
If $U^* w U$ is in $M_\om^{\th_\om}\oti\lhG$,
then there exists $z\in U(M_\om^{\th_\om})$
such that
$w=(z\oti1)\al^\om(z^*)$.
\end{cor}
\begin{proof}
The proof is  similar to that of Lemma \ref{lem: U-cocycle}.
Let $\ga^1=\Ad U^*\circ\al^\om$, $\ga^2=\Ad U^*(\cdot\oti1)$
and $\ga=(\ga^1\oti\id)\circ\ga^2$ as before.

Now we have two $\ga^1$-cocycles $U$ and $wU$.
Let $K\in\Projf(\lhG)$ be an $(F,\de)$-invariant central projection.
By Theorem \ref{thm: jRoh-G2-fix},
we can take a Rohlin projection $E\in M_\om^{\th_\om}\oti\lhG K$
as in Definition \ref{defn: j-Rohlin} for $\meC=\{U, wU\}$.
Set the Shapiro unitaries
$\mu^\de:=(\id\oti\vph)(UE)$ and $\nu^\de:=(\id\oti\vph)(wUE)$.
Then we have
\[
\mu^\de \nu^{\de*}=(\id\oti\vph)(UEU^*w^*),\quad
\mu^\de \nu^{\de*}\in M_\om^{\th_\om}.
\]

Next we show that
$\mu^\de \nu^{\de*}\in M_\om^{\th_\om}$.
By Lemma \ref{lem: vU-phi},
we have
\begin{align*}
\mu^\de \nu^{\de*}
&=
(\id\oti\vph)(UEU^*w^*)
=\sum_{\pi\in\IG}
d(\pi)^2 (\id\oti\tr_\pi)(UEU^*w^*)
\\
&=
\sum_{\pi\in\IG}
d(\pi)^2 \Ph_\pi^{\ga^2}(EU^*w^*U).
\end{align*}
Since $EU^* w^* U\in (M_\om)^{\th_\om}\oti B(H_\pi)$
by our assumption on $w$,
$\mu^\de \nu^{\de*}$ is in $M_\om$.
Using the commutativity of $\ga^2|_{M_\om}$ and $\th_\om$,
we have
$\th_g^\om(\mu^\de \nu^{\de*})
=
\mu^\de \nu^{\de*}$.
Now we get
$\mu$ and $\nu$ in $U(M^\om)$ by the index selection
as before.
Then $\mu\nu^*\in M_\om^{\th_\om}$.
By definition of an index selection map (i.e. it commutes with $\ga^1$), 
we have
$U\ga^1(\mu)=\mu\oti1$ and $wU \ga^1(\nu)=\nu\oti1$.
These imply
$w \al^\om(\nu\mu^*)
=
wU\ga^1(\nu\mu^*)U^*
=
(\nu\mu^*\oti1)$.
Therefore, $z:=\nu\mu^*$ is a desired solution.
\end{proof}

The previous result yields the following,
which can be also proved by using \cite[Theorem 7.2]{M-III1}.

\begin{cor}\label{cor: modular-app}
Let $M$ be an injective factor
and $\al$ an
 approximately inner and centrally free cocycle action of $\bhG$ on $M$.
Let $\vph\in W(M)$ and $T>0$.
Then there exists a sequence $\{w_n\}_n\subs U(M)$
such that
\begin{itemize}
\item
$\si_T^\vph=\displaystyle\lim_{n\to\infty}\Ad w_n$ in $\Aut(M)$;

\item
$[D\vph\circ\Ph_\pi^\al:D\vph\oti\tr_\pi]_T
=\displaystyle\lim_{n\to\infty}\al_\pi(w_n)(w_n^*\oti1)$
for all $\pi\in\IG$,
\end{itemize}
where the latter limit is taken in the strong* topology.
\end{cor}
\begin{proof}
By \cite[Theorem 6.2]{M-T-CMP} and Lemma \ref{lem: prop-inf-coboundary},
we can perturb $\al$ to be an action.
Considering the chain rule of Connes' cocycles,
we may and do assume that $\al$ is an action.
Applying the previous theorem to $\al$ and $\Ga=\{e\}$,
we can take a unitary $U=(U^\nu)_\nu$ in $M^\om\oti \lhG$
such that
$\Ad U^\nu$ approximates $\al$ and $U^*$ is an $\al^\om$-cocycle.

Take a sequence of unitaries $\{v^\nu\}_\nu\subs M$ such that
$\si_T^\vph=\displaystyle\lim_{\nu\to\infty}\Ad v^\nu$.
This is possible because $\si_T^\vph$ is approximately inner
\cite{Co-III1, Ha-III1, KST}.
We set $v:=(v^\nu)_\nu\in M^\om$.

For $\pi\in\IG$,
we set a unitary
$w_\pi^\nu
:=((v^\nu)^*\oti1)[D\vph\circ\Ph_\pi^\al:D\vph\oti\tr_\pi]_T^*
\al_\pi(v^\nu)$ in $M\oti B(H_\pi)$,
and also set $w^\nu:=(w_\pi^\nu)_\pi\in M\oti\lhG$
and $w=(w^\nu)_\nu\in M^\om\oti\lhG$.
Then by Lemma \ref{lem: Connes-cocycle-alpha},
we see that $w$ is an $\al^\om$-cocycle.
We will check that $U^* w U\in M_\om\oti \lhG$.
Take any $\pi\in \IG$ and $\ps\in M_*$.
Recall the notation $\ph^\nu\sim \ps^\nu$ for sequences
$(\ph^\nu)_\nu,(\ps^\nu)_\nu\subs (M\oti B(H_\pi))_*$
with $\displaystyle\lim_{\nu\to\om}\|\ph^\nu-\ps^\nu\|=0$.
Using $\Ph_\pi^\al\circ\si_T^{\vph\circ\Ph_\pi^\al}=\si_T^\vph\circ\Ph_\pi^\al$
(see \cite[\S 3.2]{M-T-endo-pre}),
we have
\begin{align*}
&\hspace{16pt}(U_\pi^\nu)^*w_\pi^\nu U_\pi^\nu\cdot
(\ps\oti\tr_\pi)\cdot(U_\pi^\nu )^* (w_\pi^\nu)^* U_\pi^\nu
\\
&\sim
(U_\pi^\nu)^*w_\pi^\nu
\cdot(\ps\circ\Ph_\pi^\al)\cdot (w_\pi^\nu)^* U_\pi^\nu
\\
&=
(U^\nu)^*
((v^\nu)^*\oti1)
[D\vph\circ\Ph_\pi^\al:D\vph\oti\tr_\pi]_T^*
\cdot
\al_\pi(v^\nu)
\cdot(\ps\circ\Ph_\pi^\al)
\\
&\hspace{16pt}
\cdot
\al_\pi((v^\nu)^*)
[D\vph\circ\Ph_\pi^\al:D\vph\oti\tr_\pi]_T
(v^\nu\oti1)
U^\nu
\\
&=
(U^\nu)^*
((v^\nu)^*\oti1)
[D\vph\circ\Ph_\pi^\al:D\vph\oti\tr_\pi]_T^*
\cdot
((v^\nu\cdot\ps\cdot (v^\nu)^*)\circ\Ph_\pi^\al)\cdot
\\
&\hspace{16pt}
\cdot[D\vph\circ\Ph_\pi^\al:D\vph\oti\tr_\pi]_T
(v^\nu\oti1)
U^\nu
\\
&\sim
(U^\nu)^*
((v^\nu)^*\oti1)
[D\vph\circ\Ph_\pi^\al:D\vph\oti\tr_\pi]_T^*
\cdot
((\ps\circ\si_{-T}^\vph)\circ\Ph_\pi^\al)\cdot
\\
&\hspace{16pt}
\cdot[D\vph\circ\Ph_\pi^\al:D\vph\oti\tr_\pi]_T
(v^\nu\oti1)
U^\nu
\\
&\sim
(U^\nu)^*
\cdot
\Big{(}
\big{(}
[D\vph\circ\Ph_\pi^\al:D\vph\oti\tr_\pi]_T^*
\cdot
((\ps\circ\si_{-T}^\vph)\circ\Ph_\pi^\al)\cdot
[D\vph\circ\Ph_\pi^\al:D\vph\oti\tr_\pi]_T
\big{)}
\\
&\hspace{60pt}
\circ(\si_T^{\vph\oti\tr_\pi})
\Big{)}
\cdot
U^\nu
\\
&=
(U^\nu)^*
\cdot
\big{(}
(\ps\circ\si_{-T}^\vph)\circ\Ph_\pi^\al)\circ\si_T^{\vph\circ\Ph_\pi^\al}
\big{)}
\cdot
U^\nu
\\
&=
(U^\nu)^*
\cdot
(\ps\circ\Ph_\pi^\al)
\cdot
U^\nu
\\
&\sim
\ps\oti\tr_\pi.
\end{align*}

By \cite[Lemma 3.6]{M-T-CMP},
we see that $U^* w U$ is in $M_\om\oti\lhG$.
Using Corollary \ref{cor: wU-cocycle},
we can take a unitary $z\in M_\om$
such that $w=(z\oti1)\al^\om(z^*)$,
that is,
$[D\vph\circ\Ph_\pi^\al:D\vph\oti\tr_\pi]_T^*
=(vz\oti1)\al^\om(z^*v^*)$.
Then a representing sequence of $vz$ satisfies the desired properties.
\end{proof}

\section{Classification for type III$_\lambda$ case}\label{sec:lamclass}

\subsection{Canonical extension to discrete cores
and the main result}\label{sec:ext}

As explained in Introduction,
our idea in type
III$_\lambda$ case is that we reduce the classification problem
to type II$_\infty$
case by using the discrete decomposition.
For this purpose, we have to consider the canonical extension
of endomorphisms of a type
III$_\lambda$ factor to its discrete core.
This is possible for endomorphisms with trivial
Connes-Takesaki modules as follows \cite[Proposition 4.5]{Iz-can2}.
Readers are referred to \S\ref{sec:appendix} for relations
between the results of \cite{Iz-can2} and \cite{M-T-endo-pre}.

Let $R$ be a type III$_\lambda$ factor, $0<\la<1$,
and $\phi$ a generalized trace,
that is, $\ph(1)=\infty$ and
$\sigma^\phi_T=\id$, $T=-2\pi/\log \lambda$, hold.
Then $R\rti_{\si^\ph}\T$ is called
the discrete core.
We denote by $\la^\ph(t)$ the unitary implementing $\si_t^\ph$ for $t\in\T$.

\begin{defn}
Let $R$ be a type III$_\la$ factor
and $K$ a finite dimensional Hilbert space.
For $\be\in \Mor_0(R,R\oti B(K))$
with the
standard left inverse $\Phi$ and $\mo(\be)=\id$,
we define the \emph{canonical extension}
$\wdt{\be}
\in\Mor(R\rti_{\si^\ph}\T,(R\rti_{\si^{\ph}}\T)\oti B(K))$
by
\begin{enumerate}
\item 
$\wdt{\be}(x)=\be(x)$ 
for all $x\in R$;

\item 
$\wdt{\be}(\la^\ph(t))
=[D\ph\circ \Ph:D\ph\oti\tr_K]_t 
(\la^\ph(t)\oti1)$ for all $t\in\R/T\Z$.
\end{enumerate}
\end{defn}

For a cocycle action $\al\in\Mor(R,R\oti\lhG)$,
we can prove that $\tal:=(\tal_\pi)_\pi$ is
a cocycle action in a similar way to the proof of
Theorem \ref{thm: canonical-ext-action}.

\begin{lem}\label{lem: be-th-app-cent}
If $\be\in\Mor(\meR_\la,\meR_\la\oti\lhG)$
is an approximately inner and centrally free cocycle action,
then
$\tbe$ is also approximately inner and centrally free.
\end{lem}
\begin{proof}
We check $\mo(\tbe_\pi)=\id$ for each $\pi\in\IG$.
Let $\hat{\ph}$ be the dual weight on $M$.
Then $\si_t^{\hat{\ph}}=\Ad \la^\ph(t)$ for $t\in\T$.
Take a positive operator $h$ such that $\la^\ph(t)=h^{-it}$
for $t\in\T$.
Then $\hat{\ph}_{h}$ is a trace on $M:=\meR_\la\rti_{\si^\ph}\T$.
Note that $\Ph_\pi^{\tbe}$ commutes with the dual action $\th$.
Let $T_\th\col M\ra \meR_\la$ be the operator valued weight
obtained by averaging the $\Z$-action $\th$.
Using
$\hat{\ph}\circ\Ph_\pi^{\tbe}=\ph\circ\Ph_\pi^\be\circ(T_\th\oti\id)$,
we can compute as follows:
\begin{align*}
[D\hat{\ph}_{h}\circ \Ph_\pi^{\tbe}: D \hat{\ph}_h\oti\tr_\pi]_t
&=
[D\hat{\ph}_{h}\circ \Ph_\pi^\tbe:D\hat{\ph}\circ \Ph_\pi^\tbe]_t
[D\hat{\ph}\circ \Ph_\pi^\tbe: D\hat{\ph}\oti\tr_\pi]_t
\\
&\hspace{16pt}
\cdot[D\hat{\ph}\oti\tr_\pi: D \hat{\ph}_h\oti\tr_\pi]_t
\\
&=
\tbe_\pi(h^{it})
[D\ph\circ \Ph_\pi^\be\circ(T_\th\oti\id): D\ph\circ T_\th\oti\tr_\pi]_t
(h^{-it}\oti1)
\\
&=
\tbe_\pi(\la^\ph(t)^*)
[D\ph\circ \Ph_\pi^\be: D\ph\oti\tr_\pi]_t
(\la^\ph(t)\oti1)
=1.
\end{align*}
By Corollary \ref{cor: app-cent-infinite},
$\tbe$ is approximately inner.

Next we check the freeness of $\tbe$.
If $\tbe_\pi$ is not properly outer for some $\pi\neq\btr$,
then $\tbe_\pi$ is actually implemented by a unitary.
This fact is proved as in the proof of \cite[Proposition 3.4]{Iz-can2}
because of the irreducibility of $\be_\pi$
\cite[Lemma 2.8]{M-T-CMP}.
Also note Lemma \ref{lem: bijection}.
Using $(\meR_\la)_\om\subs M_\om$
(see the proof of \cite[Lemma 4.11]{M-T-endo-pre}),
we see that $\be_\pi$ is centrally trivial,
and this is a contradiction.

We show that $\tbe$ is centrally free action.
The second canonical extension $\widetilde{\tbe}$ is
cocycle conjugate to $\be$
by Lemma \ref{lem: second-cocycle}.
Hence $\widetilde{\tbe}$ is centrally free on $M\rti_\th\Z$,
and $(\widetilde{\tbe}_\pi)^\om$ is non-trivial on $(M\rti_\th\Z)_\om$
for any $\pi\neq\btr$.
Since $(M\rti_\th\Z)_\om$ is naturally isomorphic to
$M_\om^{\th_\om}$
and $(\widetilde{\tbe})^\om|_{M_\om}=(\tbe)^\om|_{M_\om}$,
$(\tbe_\pi)^\om$ is non-trivial on $M_\om^{\th_\om}$
for any $\pi\neq\btr$.
In particular,
$\tbe$ is a centrally free action because $\tbe$ is free.
\end{proof}

Though the action $\tbe$ is unique up to cocycle conjugacy,
we need to consider the $\Z$-action $\th$ to obtain
the uniqueness of the original $\be$.
Our aim is to classify the $\bhG\times\Z$-action $\tbe\th$
on $\meR_{0,1}$.
The following is our main theorem in this section.

\begin{thm}\label{thm: cocycle-al-th}
Let $M\cong \meR_{0,1}$ with a trace $\ta$,
$\th\in\Aut(M)$,
$\al$ be an action of $\bhG$ on $M$,
$\be$ an action of $\bhG$ on $\meR_0$.
Assume the following:
\begin{itemize}
\item
$\th\in\Aut(M)$ satisfies $\ta\circ\th=\la\ta$, $0<\la<1$;

\item 
$\al$ is approximately inner and centrally free;

\item 
$\al$ and $\th$ commute;

\item
$\be$ is free.
\end{itemize}
Then $\alpha\theta$ is cocycle conjugate to $\th\oti\be$.
\end{thm}

Once proving the above theorem,
we can show Theorem \ref{thm:main} for $\meR_\la$ as follows.
\\

\noindent
$\bullet$\textit{Proof of Theorem \ref{thm:main}
for $\meR_\la$, $0<\lambda<1$.}

Let $\varphi$ be a generalized trace on $\meR_\lambda$,
and $M:=\meR_\lambda\rtimes_{\sigma^\varphi}\mathbb{T}$.
Then $M$ is isomorphic to $\meR_{0,1}$.
Let $\theta$ be a dual action by $\Z$ on $M$, and 
$\tal$ the canonical extension of $\al$.
Then $\tal$ is approximately inner and centrally free
by Lemma \ref{lem: be-th-app-cent}.
Applying the previous theorem to $\tal\th$,
we get $\tal\theta\sim \th\oti\be$. 

By Lemma \ref{lem: G1-G2-al-be},
the second extension $\wdt{\tal}$ on $M\rti_\th\Z$
is cocycle conjugate to $\id\oti\be$ on
$(M\rti_\th\Z)\oti\meR_0$.
By Lemma \ref{lem: second-cocycle},
$\wdt{\tal}$ is
cocycle conjugate to $\alpha$.
Hence $\alpha$ is cocycle conjugate
to $\id_{\meR_\la}\oti\be$.
\hfill$\Box$
\\

\subsection{Model action splitting}\label{sec:lam}

The rest of this section is devoted
to prove Theorem \ref{thm: cocycle-al-th}.
Let $M$, $\ta$, $\al$ and $\th$ be as in that theorem.
We also take a faithful normal state $\ph$ on $M$.
We fix their notations from here.
Since $\th$ scales the trace, the $\bhG\times\Z$-action $\al\th$
is not approximately inner.

\begin{lem}
The $\bhG\times\Z$-action $\al\th$ is centrally free.
\end{lem}
\begin{proof}
Since $\ta\circ\th=\la\ta$ with $\la\neq1$,
$\th$ is centrally free.
Assume that $\al_\pi\th^n$ is centrally trivial
for
some $\pi\in\IG$ and $n\in\Z$.
Then the map $\al_\pi\th^n$ is implemented by
a unitary by Corollary \ref{cor: app-cent-infinite},
but we have $\mo(\al_\pi\th^n)=\mo(\th^n)$
because $\mo(\al_\pi)=\id$.
Hence $n=0$, and $\pi=\btr$ by central freeness of $\al$.
\end{proof}

Take $U\in U(M^\om\oti\lhG)$
as in Theorem \ref{thm: app-cocycle} with $\Ga=\Z$ and $v_g=1$.
Define the $\bhGG$-action $\ga=(\ga^1\oti\id)\circ\ga^2$
as before,
where
\[
\ga^1(x)=U^*\al^\om(x)U,
\quad
\ga^2(x)=U^* (x\oti1)U
\quad \mbox{for}\ x\in M^\om.
\]
Since $U$ is fixed by $\th^\om$, $\ga$ commutes with $\th^\om$
on $M^\om$.
Hence $\ga\th^\om$ is a $\bhGG\times\Z$-action on $M^\om$.
Applying Corollary \ref{cor: gamma-th Rohlin}
to the strongly free and semi-liftable action
$\ga^1\circ \th$ and
the set $T=\{U_{\pi_{i,j}}\}_{i,j,\pi}''$,
we have the following result.
Note that $T'\cap (M^\om)^{\ga^1}=(M^\om)^\ga$.

\begin{lem}
For any $n\in \N$, there exists 
a partition of unity $\{E_i\}_{i=0}^{n-1}$ in $M_\om^\ga$ 
such that $\th_\om(E_i)=E_{i+1}$ for $0\leq i\leq n-1$ ($E_n:=E_0$). 
\end{lem}

As in \cite{Con-auto}, we obtain the following stability result by using
the above lemma.
\begin{lem}\label{lem: th-stab}
The $\Z$-action $\th_\om$ on $M_\om^\ga$ is stable,
that is,
for any $u\in U(M_\om^\ga)$,
there exists $w\in U(M_\om^\ga)$ 
such that $u=w\th_\om(w^*)$. 
\end{lem}

\begin{lem}\label{lem: smu}
For any $n\in\N$, there exists a system of matrix units 
$\{f_{ij}\}_{i,j=0}^{n-1}
\subset M_\omega^\gamma $ 
with $\theta_\om(f_{ij})=\mu^{i-j}f_{ij}$, 
where $\mu=e^{2\pi \sqrt{-1}/n}$.
\end{lem}
\begin{proof}
By Corollary \ref{cor: fixed-typeII} for $\ga^1\th^\om$,
we see that $(T'\cap M_\om)^{\ga^1\th^\om}=M_\om^{\ga\th^\om}$
is of type II${}_1$.
Hence
we can take a system of matrix units
$\{e_{ij}\}_{i,j=0}^{n-1}\subset M_\om^{\gamma\theta^\om}$.
Set $u:=\sum_{i=0}^{n-1}\mu^i{e_{ii}}$,
and by Lemma \ref{lem: th-stab}, we have
$w\in U(M_\om^\ga)$ such that $u=w\theta_\om(w^*)$. 
Set $f_{ij}:=w^*e_{ij}w\in M^{\ga}$. 
Then we have 
\[
\theta_\om(f_{ij})
=\theta_\om(w^*)e_{ij}\theta_\om(w)=w^*ue_{ij}u^*w=\mu^{i-j}f_{ij}. 
\]
\end{proof}

Recall the following result \cite[Proposition 7.1]{Ocn-act}.

\begin{lem}\label{lem: pi}
Let $e,f$ be projections in $M^\omega$ with $v^*v=e$. $vv^*=f$ for an
element $v\in M^\om$.
Let
$e=(e(\nu))_\nu$ and $f=(f(\nu))_\nu$ be representing sequences
such that $e(\nu)$ and $f(\nu)$ are equivalent for each $\nu\in\N$.
Then we can choose a representing sequence of $v$,
$v=(v(\nu))_\nu$
so that $v^*(\nu)v(\nu)=e(\nu)$ and $v(\nu)v(\nu)^*=f(\nu)$
for each $\nu\in\N$.
\end{lem}

\begin{lem}\label{lem: single-matrix-model}
Let $n\in\N$ and $\mu=e^{2\pi\sqrt{-1}/n}$.
Then for any $\mF\Subset \IG$,
$\Psi \Subset (M_*)_+$, and $\epsilon>0$,
there exists a unitary $u\in M\oti L^\infty(\bhG)$,
a unitary $w\in M$
and a system of matrix units $\{f_{ij}\}_{i,j=0}^{n-1}$ in $M$
such that 
\begin{enumerate}
\renewcommand{\labelenumi}{(\roman{enumi})}
\item 
$\|u_{\pi}-1\|_{\ph\oti\tr_\pi}^\#<\epsilon$
for all $\pi\in \mF$;

\item
$\|w-1\|_{\ph}^\#<\epsilon$;

\item 
$\|[f_{ij},\psi]\|<\epsilon$
for all $\psi\in\Psi$ and $0\leq i,j \leq n-1$;

\item
$\Ad u\circ\alpha(f_{ij})=f_{ij}\oti1$
for all $0\leq i,j \leq n-1$;

\item
$\Ad w\circ\th(f_{ij})=\mu^{i-j}f_{ij}$
for all $0\leq i,j \leq n-1$.
\end{enumerate}
\end{lem}
\begin{proof}
Let $\{f_{ij}\}_{i,j=0}^{n-1}$ be a system of matrix units
in $M_\om^\ga$ as in Lemma \ref{lem: smu}.
Then
$\gamma(f_{ij})=f_{ij}\otimes 1$
implies $\alpha^\om(f_{ij})=f_{ij}\otimes 1$.
Take a representing sequence of $f_{ij}$, $(f_{ij}(\nu))_{\nu}$
such that $\{f_{ij}(\nu)\}_{i,j=0}^{n-1}$ is a system of matrix units
in $M$ for all $\nu$.

By Lemma \ref{lem: pi},
for each $\pi\in\IG$,
there exists $v_\pi(\nu)\in M\oti B(H_\pi)$
such that $v_\pi(\nu) v_\pi(\nu)^*=f_{00}(\nu)\otimes 1$,
$v_\pi(\nu)^* v_\pi(\nu)=\alpha_\pi(f_{00}(\nu))$ and
$(v_\pi(\nu))_\nu=f_{00}\otimes 1$.
Set a unitary
$u_\pi(\nu):=\sum_{i=0}^{n-1}
(f_{i0}(\nu)\otimes1)v_\pi(\nu)\alpha_\pi(f_{0i}(\nu))$.
Then
$\Ad u_\pi(\nu)\circ\alpha_\pi(f_{ij}(\nu))=f_{ij}(\nu)\otimes 1$ 
holds.
We have $(u_\pi(\nu))_\nu=1$ in $M^\om \oti B(H_\pi)$.
Indeed,
\[
(u_\pi(\nu))_\nu
=\sum_{i=0}^{n-1}(f_{i0}(\nu)\otimes1)_\nu (v_\pi(\nu))_\nu 
\alpha_\pi^\om((f_{0i}(\nu))_\nu)=
\sum_{i=0}^{n-1}(f_{i0}\otimes 1)(f_{00}\otimes 1)(f_{0i}\otimes 1)
=1.
\]
Set a unitary $u(\nu)=(u_\pi(\nu))_{\pi}$ in $M\oti\lhG$.

Next we construct $w$.
Applying  Lemma \ref{lem: pi} to $\theta_\om(f_{00})=f_{00}$,
there exists $v(\nu)\in M$
such that $v(\nu) v(\nu)^*=f_{00}$,
$v(\nu)^* v(\nu)=\theta(f_{00}(\nu))$ and
$(v(\nu))_\nu=f_{00}$.
Set a unitary
$w(\nu):=\sum_{i=0}^{n-1}\mu^{i}f_{i0}(\nu)
v(\nu)\theta(f_{0i}(\nu))$. 
Then
$\Ad w(\nu)\circ\theta(f_{ij}(\nu))=\mu^{i-j}f_{ij}(\nu)$ holds
for all $0\leq i,j\leq n-1$ and $\nu\in\N$.
We can show $w(\nu)\to1$ strongly* as $\nu\to\om$ as above.

Hence we can choose $\nu\in\N$ such that $u=u(\nu)$, $w=w(\nu)$
and $f_{ij}(\nu)$ satisfy the desired conditions.
\end{proof}

Let $\Psi_n\Subset M_*$ be an increasing subset such that
$\Psi=\bigcup_{n=1}^\infty\Psi_n$ is total in $M_*$.
We recall the following result due to Connes \cite[Lemma 2.3.6]{Con-auto}.
\begin{lem}\label{lem:con-split}
Let $M_1,M_2,\cdots M_n\subset M$ be mutually commuting
finite dimensional subfactors.
Denote $\bigvee_{k=1}^\infty M_k:=N$.
If $\sum_{k=1}^\infty\|\psi\circ E_{M_k'\cap M}-\psi\|< \infty$
for all $\psi\in \Psi$,
then $N$ is a hyperfinite subfactor of type II$_1$ and
we have $M=N\vee N'\cap M\cong N\otimes N'\cap M $.
Here $E_{M_k'\cap M}=\tr_{M_k}\otimes \id_{M_k'\cap M}$.
\end{lem}

Let $\{n_k\}_{k=1}^\infty\subset \mathbb{N}$ be
such that any $n\in \mathbb{N}$ appears infinitely many times.
Set $\mu_k:=e^{2\pi\sqrt{-1}/n_k}$.
For a system of $n_k\times n_k$ matrix units
$\{e_{ij}\}_{i,j=0}^{n_k-1}$,
set $u_{n_k}:=\sum_{j=0}^{n_k-1}\mu_k^j e_{jj}$, 
and 
$\sigma:=\bigotimes_{k=1}^\infty \Ad u_{n_k}$. 
Then $\sigma$ is an aperiodic automorphism on
$\bigotimes_{k=1}^\infty M_{n_k}(\mathbb{C })\cong \meR_0$. 
We will prove the following model action splitting result.

\begin{lem}\label{lem: alth-id}
The action $\al\th$ is cocycle conjugate
to the action $\si\oti\al\th$.
\end{lem}
\begin{proof}
\textbf{Step 1.} 
Let $\{\epsilon_k\}_{k=1}^\infty$ 
be a decreasing sequence of positive numbers with 
$\sum_{k=1}^\infty\epsilon_k< \infty$.
Let $\{\mF_m\}_{m=1}^\infty$
be a family of increasing finite subsets of
$\IG$
such that $\bigcup_{m=1}^\infty\mF_m=\IG$.

Recall that we have fixed a faithful normal state $\phi\in M_*$.
We will construct the following families:
\begin{enumerate}
\item
Matrix units, 
$\{f_{ij}^{k}\}_{i,j=0}^{n_k-1}\subs M$ for $k\in\N$
such that
they are mutually commuting for $k$ and satisfy
$\|[\psi, f_{ij}^k]\|\leq \epsilon_k/n_k$
for all $0\leq i,j\leq n_k$, $\psi\in \Psi_k$ and $k\in\N$.
\\
We set $M_{n_k}:=(\{f_{ij}^k\}_{i,j=0}^{n_k-1})''$ and
$E_m:=\bigvee_{k=1}^m M_{n_k}$;

\item
Unitaries $u^{m}\in (E_{m-1}'\cap M)\otimes \lhG$
and $w^m\in E_{m-1}'\cap M$
satisfying the following for each $m\in\N$:
\begin{itemize}
\item
$\|u_{\pi}^m-1\|_{\phi\oti\tr_\pi}^\#<\epsilon_m$
and
$\|w^m-1\|_{\phi}^\#<\epsilon_m$
for all $\pi\in\mF_m$;

\item
We set
$\bar{u}^m:=u^{m}u^{m-1}\cdots u^1$
and $\bar{w}^m:=w^m w^{m-1}\cdots w^1$.
Then we have, for all $0\leq i,j\leq n_k-1$ and $1\leq k\leq m$,
\\
$\Ad\bar{u}^m\circ \al(f_{ij}^k)
=f_{ij}^k\oti1$;
\\
$\Ad \bar{w}^m\circ\th(f_{ij}^k)
=\mu_k^{(i-j)}f_{ij}^k$.
\end{itemize}
\end{enumerate}

Assume we have constructed up to $k=m$.
Set $\alpha^m:=\Ad \bar{u}^m\circ \alpha$,
and $\theta_{(m)}:=\Ad \bar{w}^m\circ\th$.
Since $\al^m$ fixes $E_m$, $\al^m$ is a cocycle action
on $E_m'\cap M$.

Let $\{\hat{e}^{i}\}$ a basis for $E_m^*$.
Let us decompose 
$\psi\in \Psi_{m+1}$ as
$\psi=\sum_{i=1}^{\dim(E_m)} \hat{e}^{i}\oti \psi_{i}$,
$\psi_{i}\in (E'_m\cap M)_*$, and
denote by $\hat{\Psi}_{m+1}$ the set of all such $\psi_{i}$.
Fix $\delta_{m+1}>0$ so that $\delta_{m+1}\leq
\epsilon_{m+1}\left(n_{m+1}\dim E_m\right)^{-1}$.
\\

\noindent\textbf{Claim.}
There exist the following elements:
\begin{enumerate}
\item
A system of matrix units
$\{f_{ij}^{m+1}\}_{i,j=0}^{n_{m+1}-1}\subset E_m'\cap M$
such that
$\|[\psi, f_{ij}^{m+1}]\|\leq \delta_{m+1}$
for $\psi\in \hat{\Psi}_{m+1}$.
\\
Set $M_{m+1}:=(\{f_{ij}^{m+1}\}_{i,j=0}^{n_{m+1}-1})''$
and $E_{m+1}=E_m\vee M_{m+1}$.
\item
Unitaries $u^{m+1}\in (E_{m}'\cap M)\oti\lhG$
and $w^{m+1}\in E_{m}'\cap M$
satisfying the following
\begin{itemize}
\item
$\|u_{\pi}^{m+1}-1\|_{\phi\oti\tr_\pi}^\#<\epsilon_{m+1}$
and $\|w^{m+1}-1\|_\phi^\#<\epsilon_{m+1}$
for all $\pi\in\mF_{m+1}$;

\item
$\Ad u^{m+1}\circ\alpha^m(f_{ij}^{m+1})
=f_{ij}^{m+1}$
for all $0\leq i,j\leq n_{m+1}-1$;
\\
$\Ad w^{m+1}\circ \th_{(m)}(f_{ij}^{m+1})
=\mu_{m+1}^{\el(i-j)} f_{ij}^{m+1}$
for all $0\leq i,j\leq n_{m+1}-1$.
\end{itemize}
\end{enumerate}

Indeed, we can prove this as follows.
By the natural isomorphism 
$(E_m'\cap M)^\om=E_m'\cap M^\om$, 
we have 
\begin{equation}\label{eq: rel-com}
\big{(}
(E_m'\cap M)_\om\subs (E_m'\cap M)^\om
\big{)}
=
\big{(}
M_\om\subs E_m'\cap M^\om
\big{)}.
\end{equation}
On $E_m'\cap M$,
we have a $\bhG$-cocycle action $\al^m$
and a $\Z$-action $\th_{(m)}$.
Using Lemma \ref{lem: smu}, we take 
a system of matrix units $\{f_{ij}\}_{i,j=0}^{n_{m+1}-1}\subs M_\om^\ga$
such that $\th_\om(f_{ij})=\mu_{m+1}^{i-j}f_{ij}$
for $0\leq i,j\leq n_{m+1}-1$.
Then we get
$\th_{(m)}^\om(f_{ij})
=\bar{w}^m \th_\om(f_{ij})(\bar{w}^m)^*=\mu_{m+1}^{i-j}f_{ij}$.
Since $f_{ij}$ is fixed by $\ga$,
$\al^\om(f_{ij})=f_{ij}\oti1$ as before.
Hence we have
$(\al^m)^\om(f_{ij})=\bar{u}^m(f_{ij}\oti1)(\bar{u}^m)^*=f_{ij}\oti1$.
By using (\ref{eq: rel-com}), we can represent
$\{f_{ij}\}_{i,j=0}^{n_{m+1}-1}$ as sequences
$\{(f_{ij}(\nu))_\nu\}_{i,j=0}^{n_{m+1}-1}$ in $E_m'\cap M$.
Then we can take a desired elements in the Claim
as in Lemma \ref{lem: single-matrix-model}.
\\

Now the condition (1) in the Claim implies
$\|[\psi, f_{ij}^{m+1}]\|\leq \epsilon_m/n_m$ for $\psi\in \Psi_{m+1}$.
Thus we complete induction.
We have constructed families $u^m$, $w^m$ and $E^m$ for $m\in\N$.
Since for $\ps\in \Psi_k$ we have
\begin{align*}
\|\psi\circ
E_{M_{n_k}'\cap M}-\psi\| 
&=
\left\|\frac{1}{n_k}\sum_{ij=0}^{n_k-1}f_{ij}^k\psi f_{ji}^k-\psi\right\|
= 
\left\|\frac{1}{n_k}\sum_{ij=0}^{n_k-1}f_{ij}^k [\psi,f_{ji}^k]\right\| \\
&\leq \frac{1}{n_k} \sum_{ij=0}^{n_k-1}\left\|[\psi,f_{ji}^k]\right\|
\leq \frac{1}{n_k}\cdot n_k^2 \cdot \frac{\epsilon_k}{n_k}
=\epsilon_k,
\end{align*}
we can check
$\sum_{k=1}^\infty\|\psi\circ E_{M_{n_k}'\cap M}-\psi\|< \infty$
for $\psi\in \Psi$.
Then Lemma \ref{lem:con-split} implies
that $E:=\bigvee_{k=1}^\infty E_k$ is isomorphic to $\meR_0$
and yields a tensor product splitting
$M=E\vee (E'\cap M)\cong E\otimes (E'\cap M)$.

\textbf{Step 2.}
From the condition (2),
the strong* limits $\displaystyle\ovl{u}^\infty=\lim_{m\to\infty}\ovl{u}^m$ 
and $\displaystyle\ovl{w}^\infty=\lim_{m\to\infty}\ovl{w}^m$ exist,
and together with (1),
we have
$\Ad \ovl{u}^\infty\circ\al(x)=x\oti1$
and $\Ad \ovl{w}^\infty\circ\th(x)=\si(x)$
for $x\in E$.
Extend $\ovl{w}^\infty$ to the $\th$-cocycle naturally
and denote it also by $\ovl{w}^\infty\in M\oti \el^\infty(\Z)$.
Then we get the perturbation from the $\bhG\times\Z$-action
$\al\th$ to the $\bhG\times\Z$-cocycle action
$(\Ad\ovl{u}^\infty\al(\ovl{w}^\infty)\circ\al\th,v)$.
Set $\be:=\Ad\ovl{u}^\infty\al(\ovl{w}^\infty)\circ\al\th$.
Then $\be$ is of the form $\si\oti \be'$ on $E\oti (E'\cap M)$,
where $\be'=\be|_{E'\cap M}$.
We claim that $v$ is evaluated in $E'\cap M$,
and $(\be',v)$ is a cocycle action.

By definition of $v$,
$(\be\oti\id)\circ \be=\Ad v \circ(\id\oti\De_{\bhG\times\Z})\circ\be$.
Let $k\in\N$ and $0\leq i,j\leq n_k-1$.
Then we have the following:
\[
(\be_{(\pi,\el)}\oti\id)(\be_{(\rho,m)}(f_{ij}^k))
=
\mu_k^{m(i-j)} \be_{(\pi,\el)}(f_{ij}^k\oti1_\rho)
=
\mu_k^{(\el+m)(i-j)}(f_{ij}^k\oti1_\pi\oti1_\rho)
\]
and 
\begin{align*}
(\id\oti\De_{\bhG\times\Z})(\be(f_{ij}^k))_{(\pi,\el),(\rho,m)}
&=
(\id\oti\De)(\be(f_{ij}^k)_{(\cdot,\el+m)})_{\pi,\rho}
\\
&=
\mu_k^{(\el+m)(i-j)}(\id\oti\De)(f_{ij}^k\oti1)_{\pi,\rho}
\\
&=
\mu_k^{(\el+m)(i-j)}(f_{ij}^k\oti1_\pi\oti1_\rho). 
\end{align*}
Hence $v$ is evaluated in $M_k'\cap M$ for any $k\in\N$, and hence in  $E'\cap M$.

We have shown that $\al\th$ is cocycle conjugate to the
cocycle action $\si\oti \be'$.
Since $E'\cap M$ is type III, we can perturb $(\be',v)$
to a $\bhG\times\Z$-action $\be''$
by Lemma \ref{lem: prop-inf-coboundary}.
Hence $\al\th\sim \si\oti\be''$.
Since $\si\oti\si\approx\si$,
we get $\al\th\sim \si\oti \si\oti\be''\sim\si\oti \al\th$.
\end{proof}

\begin{rem}
We can use the Jones-Ocneanu cocycle argument in
\cite[Lemma 2.4]{Ocn-act} to obtain cocycle conjugacy
$\al\th\sim\si\oti \al\th$ in Step 2 above.
We set $\nu:=\ovl{u}^\infty\al(\ovl{w}^\infty)$.
Then we have $\Ad \nu\circ\al\th=\be=\si\oti\be'$.
Since $\si$ is conjugate to $\si\oti\si$,
there exists an isomorphism $\gamma$ from $E\otimes E$ onto $E$
with $\gamma^{-1}\circ \si \circ \gamma=\si\otimes \si$.
So we have
\begin{align*}
(\gamma^{-1}\oti\id\oti\id_{\lhGZ})\circ \Ad \nu\circ\alpha\theta
\circ (\gamma\oti\id)
&=
\gamma^{-1}\circ\sigma \circ \gamma\oti \be'
=\si\otimes \si \oti\be'
\\
&=\si\oti \Ad \nu\circ\al\th.
\end{align*}
Then the following holds:
\[
\Ad (\ga\oti\id\oti\id)(1\oti\nu^*)\nu\circ\al\th
=
(\ga\oti\id\oti\id)\circ(\si\oti \al\th)
\circ (\ga^{-1}\oti\id).
\]
We will verify that $(\ga\oti\id\oti\id)(1\oti\nu^*)\nu$
is an $\alpha\theta$-cocycle.
Here note that $(\gamma\oti\id\oti\id)(1\oti v)=v$ holds
because the 2-cocycle $v$ is evaluated in $E'\cap M$.
Then the following holds:
\begin{align*}
&\,
((\ga\oti\id\oti\id)(1\oti\nu^*)\nu\oti1)
\cdot\alpha\theta((\ga\oti\id\oti\id)(1\oti\nu^*)\nu)
\\
&=
((\ga\oti\id\oti\id)(1\oti\nu^*)\oti1)
\cdot
(\si\oti \be')((\ga\oti\id\oti\id)(1\oti\nu^*))
(\nu\oti1)\alpha\theta(\nu) \\ 
&=
(\gamma\oti\id\oti\id)
((1\oti\nu^*\oti1)(\si\otimes \si\oti\be')(1\oti\nu^*))
v(\id\otimes \Delta)(\nu) \\
&=
(\gamma\oti\id\oti\id)
(1\oti \alpha\theta(\nu^*)\nu^*)
v(\id\otimes \Delta)(\nu) \\
&=
(\gamma\oti\id\oti\id)
(1\oti (\id\otimes \Delta)(\nu^*)v^*)
v(\id\otimes \Delta)(w) \\
&=
(\id\otimes \Delta)((\ga\oti\id\oti\id)(1\oti \nu^*)\nu).
\end{align*}
Hence $\alpha\theta$ and $\alpha\theta\otimes\si$ are cocycle conjugate.
\end{rem}

\noindent
$\bullet$\textit{ Proof of Theorem \ref{thm: cocycle-al-th}.}

Note that $\theta\otimes \theta^{-1} $ is cocycle conjugate to
$\id_{B(\el_2)}\oti \sigma$ by Connes \cite{Con-auto}. 
Then the following holds:
\begin{align*} 
\alpha\theta 
&\sim \id_{B(\el_2)}\oti \al\th
\quad(\mbox{by}\ \mbox{Lemma}\ \ref{lem: id-otimes-al})\\
&\sim \id_{B(\el_2)}\oti \si\oti \al\th
\quad(\mbox{by}\ \mbox{Lemma}\ \ref{lem: alth-id})\\
&\sim\theta \otimes \theta^{-1}\oti \alpha\theta.  
\end{align*}
Since the action 
$\th^{-1}\oti\alpha\theta$ preserves the trace of 
$\meR_{0,1}$,
it is approximately inner. 
The central freeness is clear.
Then $\th^{-1}\oti\alpha\theta$ is
cocycle conjugate to $\id_{B(\el_2)}\oti \si\oti \be$
by Theorem \ref{thm:main} for type II$_\infty$ case,
and the following holds:
\begin{align*}
\theta \otimes \theta^{-1}\oti \alpha\theta
&\sim 
\th\oti \id_{B(\el_2)}\oti \si \oti \beta 
\\
&\sim
\th \oti \si \oti \beta
\\
&\sim 
\th\oti \beta.
\end{align*}
Therefore we get $\al\th\sim\th\oti\be$.
\hfill$\Box$
\\

We close this section with the following lemma
which is used in Section \ref{sec:III1}.

\begin{lem}\label{lem: gen-tr-cocycle}
Let $N$ be a type III$_\la$ factor with $0<\la<1$
and
$\al$ an approximately inner action of $\bhG$ on $N$.
Let $\ps$ be a generalized trace on $N$.
Then there exists a $\bhG$-action $\be$ on $N$
such that
\begin{itemize}
\item
$\be\sim\al$;

\item
$\ps\circ\Ph_\pi^\be=\ps\oti\tr_\pi$
for all $\pi\in\IG$.
\end{itemize}
\end{lem}
\begin{proof}
Since $\al$ is approximately inner,
we see that $\ps\circ\Ph_\pi^\al$ is a generalized trace
for all $\pi\in\IG$.
Hence there exists a unitary $v_\pi\in N\oti B(H_\pi)$
such that
$\ps\circ\Ph_\pi^\al=(\ps\oti\tr_\pi)\circ\Ad v_\pi$.
Set $v=(v_\pi)_\pi\in N\oti\lhG$, and consider 
the cocycle action $\de:=\Ad v\circ\al$,
whose 2-cocycle is given by $u:=(v\oti1)\al(v)(\id\oti\De)(v^*)$.
Then we have $\ps\circ\Ph_\pi^\de=\ps\oti\tr_\pi$,
and $\si^\ps$ and $\de$ commute in particular.

We check that $u$ is evaluated in $N_\ps$ as follows:
for $\pi,\rho\in\IG$,
\begin{align*}
&u^*(\ps\oti\tr_\pi\oti\tr_\rho)u
\\
&=
u^*
\cdot(\ps\circ\Ph_\rho^\de\circ(\Ph_\pi^\de\oti\id))\cdot
u
\\
&=
(\id\oti\De)(v)\al_\pi(v_\rho^*)(v_\pi^*\oti1)
\cdot(\ps\circ\Ph_\rho^\de\circ(\Ph_\pi^\de\oti\id))\cdot
(v_\pi\oti1)\al_\pi(v_\rho)(\id\oti\De)(v^*)
\\
&=
(\id\oti\De)(v)\al_\pi(v_\rho^*)
\cdot(\ps\circ\Ph_\rho^\de\circ(\Ph_\pi^\al\oti\id))\cdot
\al_\pi(v_\rho)(\id\oti\De)(v^*)
\\
&=
(\id\oti\De)(v)
\cdot\Big{(}
\big{(}v_\rho^*\cdot(\ps\circ\Ph_\rho^\de)\cdot v_\rho\big{)}
\circ(\Ph_\pi^\al\oti\id)
\Big{)}\cdot
(\id\oti\De)(v^*)
\\
&=
(\id\oti\De)(v)
\cdot\big{(}
\ps\circ\Ph_\rho^\al
\circ(\Ph_\pi^\al\oti\id)
\big{)}\cdot
(\id\oti\De)(v^*)
\\
&=
\sum_{\si\prec\pi\oti\rho}
\sum_{S\in\ONB(\si,\pi\oti\rho)}
\frac{d(\si)}{d(\pi)d(\rho)}
(\id\oti\De)(v)
(1\oti S)\cdot(\ps\circ\Ph_\si^\al)\cdot(1\oti S^*)
(\id\oti\De)(v^*)
\\
&=
\sum_{\si\prec\pi\oti\rho}
\sum_{S\in\ONB(\si,\pi\oti\rho)}
\frac{d(\si)}{d(\pi)d(\rho)}
(1\oti S)v_\si\cdot(\ps\circ\Ph_\si^\al)\cdot v_\si^*(1\oti S^*)
\\
&=
\sum_{\si\prec\pi\oti\rho}
\sum_{S\in\ONB(\si,\pi\oti\rho)}
\frac{d(\si)}{d(\pi)d(\rho)}
(1\oti S)\cdot(\ps\circ\Ph_\si^\de)\cdot(1\oti S^*)
\\
&=
\sum_{\si\prec\pi\oti\rho}
\sum_{S\in\ONB(\si,\pi\oti\rho)}
\frac{d(\si)}{d(\pi)d(\rho)}
(1\oti S)\cdot(\ps\oti\tr_\si)\cdot(1\oti S^*)
\\
&=\ps\oti\tr_\pi\oti\tr_\rho.
\end{align*}
Hence $u\in N_\ps\oti\lhG\oti\lhG$, and $(\de|_{N_\ps},u)$
is a cocycle action on the type II$_\infty$ factor $N_\ps$.
By Lemma \ref{lem: prop-inf-coboundary},
there exists a unitary $w\in N_\ps\oti\lhG$ perturbing
$(\de,u)$ to the action $(\Ad w\circ\de,1)$.
Then $wv$ is an $\al$-cocycle and we set $\be:=\Ad wv\circ \al$.
We check that $\be$ satisfies the second condition.
Since $\Ph_\pi^\be=\Ph_\pi^\de\circ\Ad w_\pi^*$ and $w\in N_\ps\oti\lhG$,
we have
\[
\ps\circ\Ph_\pi^\be
=\ps\circ\Ph_\pi^\de\circ\Ad w_\pi^*=(\ps\oti\tr_\pi)\circ\Ad w_\pi^*
=\ps\oti\tr_\pi.
\]
\end{proof}

\section{Groupoid actions and type III$_0$ case}\label{sec:III0}
Let $M$ be an injective factor of type III$_0$
and $\{\tM, \theta, \tau_{\tM} \}$ the canonical core of $M$.
Let $(X, \nu, \mathcal{F}_t)$ be the flow of weights for $M$,
that is,
$Z(\tM)=L^\infty(X,\nu)$,
$\theta_t(f)(x)=f(\mathcal{F}_{-t}x)$
and $\nu$ is a measure on $X$.
We represent $(X,\nu,\mathcal{F}_t)$
as a flow built under the ceiling function, that is, there
exist a measure space $(Y,\mu)$, $f\in L^\infty(Y,\mu)$ with
$f(x)\geq R$ for some $R>0$,
and a nonsingular transformation $T$ on $(Y, \mu)$
such that $X$ is identified with
$\{(y,t)\mid y\in Y,0\leq t < f(y)\}$, $\nu =\mu\times dt$,
and
$\mathcal{F}_t(y,s)=(y,t+s)$ where we identify $(y,f(y))$ and $(Ty,0)$.
Then we have two kinds of measured groupoids,
$\widetilde{\meG}:=  \R\ltimes_{\mathcal{F}} X$ and
$\meG:=\mathbf{Z} \ltimes_T Y$.
In fact, $\meG$ is characterized
as $\meG=\{\gamma\in \widetilde{\meG}\mid s(\gamma),r(\gamma)\in Y\}$.
Here for a $\Gamma$-space $Z$, the groupoid $\Gamma \ltimes Z$ is
defined as $(g,hx)(h,x)=(gh, x)$ for $g,h\in \Gamma$ and $x\in Z$. 
The source map $s$ and the range map $r$ are defined by $s(g,x)=x$ and
$r(g,x)=gx$, respectively. 

Let $\al$ be an approximately inner action
of $\bhG$ on $M$.
Then $\mo(\al)=\id$ by Theorem \ref{thm: app-cent-prop},
that is, the canonical extension $\tal$ fixes $L^\infty(X,\nu)$.
We first discuss
the reduction of the study of $\bhG\times\R$-action $\tal\th$
to the groupoid actions.

Let $\tM=\int_{X}^{\oplus} \tM(x) dx$ be the central decomposition.
Since $\tM(x)$ are injective for almost every $x\in X$,
$M(x) \cong \meR_{0,1}$ holds for almost every $x\in X$.
As in \cite{Su-Tak-act},
we obtain a family of actions
$\{\tal_x\}_{x\in X}$ of $\bhG$ on $\tM(x)$ determined by 
\[
\tal(a)=\int_{X}^{\oplus}\tal(a)(x)d\mu(x)
=\int_{X}^{\oplus}\tal_x(a(x))d\mu(x),
\]
and an action $\{\theta_{\gamma}\}_{\ga\in\wdt{\meG}}$
of
$\widetilde{\meG}$ by 
\[
\theta_t(a)=\int_{X}^{\oplus}\theta_t(a)(x)d\mu(x)=\int_X^{\oplus} 
\theta_{\gamma}(a(\mathcal{F}_{-t}x))d\mu(x),
\]
where $\gamma=(t, \mathcal{F}_{-t}x)$.
Of course $\theta_\gamma$ is an isomorphism
from $\tM(s(\gamma))$ onto $\tM(r(\gamma))$.
Then $\theta_\gamma$ and $\tal_x$
commute in the following sense:
$\tal_{r(\gamma)}\circ\theta_{\gamma}
=\left(\theta_{\gamma}\otimes \id\right)\circ \tal_{s(\gamma)}$.
Since
$\tal$ preserves $\tau_{\tM}$ by Lemma \ref{lem: trace-inv},
each $\tal_x$ preserves $\ta_x$, a trace on $\tM(x)$.
We denote
the $\pi$-component of $\alpha_x$ by $\alpha_{\pi,x}$.
We introduce the notion of a $\bhG$-$\meG$-action.

\begin{defn}
Let $R$ be a von Neumann algebra,
$\bhG$ a discrete Kac algebra, and $\meG$ a groupoid.
\begin{enumerate}
\item
Let $\{\alpha_x\}_{x\in\meG^{(0)}}$ be a family of actions of $\bhG$ and
$\{\alpha_\gamma\}_{\ga\in\meG}$ an action of $\meG$ on $R$.
We say that $\alpha $ is a $\bhG$-$\meG$-\emph{action}
if
$\alpha_{r(\gamma)}\circ\alpha_\gamma=(\alpha_\gamma\otimes \id)
\circ
\alpha_{s(\gamma)}$
for all $\ga$.
We denote $\alpha_{r(\gamma)}\alpha_\gamma$
and
$\alpha_{\pi,r(\gamma)}\alpha_\gamma$ 
by
$\alpha_{\cdot, \gamma}$ and $\alpha_{\pi,\gamma}$ for simplicity.
We call $\{\alpha_x\}_{x\in \meG^{(0)}}$
and
$\{\alpha_\gamma\}_{\gamma\in \meG}$
the $\bhG$-\emph{part} and the $\meG$-\emph{part} of $\alpha$, respectively.

\item
For two $\bhG$-$\meG$ actions $\alpha$ and $\beta$ on $R$,
we say that $\alpha$ and $\beta$ are \emph{cocycle conjugate}
if there exist a Borel function
$\sigma\col X\ra\Aut(R)$,
a $\beta_x$-cocycle $u^x$ for $x\in X$ and
a $\beta_\gamma$-cocycle $u_\gamma$ for $\ga\in\meG$
satisfying, for all $x\in X$ and $\ga\in\meG$,
\[(\sigma_{r(\gamma)}\otimes \id)
\circ \alpha_{\cdot,\gamma}\circ \sigma^{-1}_{s(\gamma)}
=\Ad (u^{r(\gamma)}\beta_{r(\gamma )}(u_\gamma))\circ
\beta_{\cdot,\gamma}
\]
and
\[
u^{r(\gamma)}\beta_{r(\gamma)}(u_\gamma)
=
(u_\gamma\otimes 1)(\beta_\gamma\otimes \id)(u^{s(\gamma)}).
\]
In this case, we simply say that
$u^{r(\gamma)}\beta_{r(\gamma)}(u_\gamma)$ is a $\beta$-cocycle.
\end{enumerate}
\end{defn}

The following can be shown as \cite[p.430]{Su-Tak-act}.

\begin{lem}\label{lem: al-be-cocycle}
Let $\alpha$, $\beta$ be actions of $\bhG$
on a type III$_0$ injective factor $M$.
Suppose that $\mo(\al)=\id=\mo(\be)$.
\begin{enumerate}
\item
The $\bhG\times\R$-actions
$\tal\theta$ and $\tbe\theta$ on $\tM$ are cocycle
conjugate if and only if
the $\bhG$-$\wdt{\meG}$-actions $\tal_{r(\gamma)}\th_\ga$
and
$\tbe_{r(\gamma)}\theta_\gamma$ on $\meR_{0,1}$
are cocycle conjugate.

\item
If the $\bhG$-$\meG$-actions
$\tal_{r(\gamma)}\theta_\gamma$
and $\tbe_{r(\gamma)}\theta_\gamma $ on $\meR_{0,1}$
are cocycle conjugate,
then they are also cocycle conjugate
as the $\bhG $-$\widetilde{\meG}$-actions.
\end{enumerate}
\end{lem}

Hence we only have to classify two $\bhG$-$\meG$
actions $\tal_{r(\gamma)}\theta_\gamma$
and $\tbe_{r(\gamma)}\theta_\gamma$ on $\meR_{0,1}$.
Here the $\bhG$-parts preserve the trace,
and the $\meG$-parts come from $\theta$,
which are independent from $\alpha$ and $\beta$.
Now we consider the following situation:

\begin{itemize}
\item
We are given two $\bhG$-$\meG$-actions $\alpha$ and $\beta$
on $\meR_{0,1}$;

\item
The $\bhG$-parts of $\alpha$ and $\beta$ are free actions;

\item
The $\bhG$-parts of $\alpha$ and $\beta$ preserve the trace on $\meR_{0,1}$;

\item
$\mathrm{mod}(\alpha_\gamma)=\mathrm{mod}(\beta_\gamma)$
for $\gamma\in \meG$.
\end{itemize}

Note that $\meG$ is an ergodic,
approximately finite (AF), orbitally discrete principal
groupoid, and the following
Krieger's cohomology lemma provides a powerful tool
for study of actions of such groupoids \cite{Kri-erg}.
(Also see \cite[Appendix]{JT}.)

\begin{thm}\label{lem:Krieger}
Let $G$ be a Polish group, and $N$ a normal subgroup.
Let $\meG$
be an ergodic AF orbitally discrete principal groupoid. Let $\theta^1$
and $\theta^2$ be  homomorphisms from $\meG$ to $G$ with
$\theta^1_\gamma \equiv \theta^2_\gamma \mod \ovl{N}$.
Then there exist Borel maps
$\sigma \col\meG^{(0)}\rightarrow \ovl{N}$ and 
$u \col \meG \rightarrow N$ 
such that
$\sigma_{r(\gamma)}\theta^1_\gamma
\sigma_{s(\gamma)}^{-1}
= u_\gamma \theta^2_\gamma$.
\end{thm}

We need some preparations as in \cite{JT,Su-Tak-RIMS}.
Let $\sigma$ be a trace
preserving free action of $\bhG$ on $\meR_{0,1}$.
Let $C^{(1)}_\sigma$ be the set of pairs $(\theta, v)$,
where $\theta\in
\overline{\Int}(\meR_{0,1})$ and $v$ is a $\sigma$-cocycle
such that $\Ad v\circ\sigma=(\theta\otimes \id)\circ \sigma \circ \theta^{-1}$.
We define the multiplication on
$C^{(1)}_\sigma$ by
$(\theta_1, v_1)(\theta_2,v_2)
:=(\theta_1\theta_2, (\theta_1\otimes \id)(v_2)v_1)$.
Let $\Aut_{\hat{\sigma}}(\meR_{0,1}\rtimes_\sigma \bhG)$
be the set of all
automorphisms which commute with the dual action of $\bG$.
Then we have 
$C^{(1)}_\sigma \subset \Aut_{\hat{\sigma}}(\meR_{0,1}\rtimes_\sigma \bhG)$
in a canonical way, and $C^{(1)}_\sigma$ is a Polish group.
In fact,
$C^{(1)}_\sigma=\Aut_{\hat{\sigma}}(\meR_{0,1}\rtimes_\sigma \bhG)\cap
\mathrm{Ker(\mo)}$ holds.
Let $C^{(0)}_\sigma:=\{(\Ad v, (v\otimes 1)\sigma (v^*))\mid v\in
U(\meR_{0,1})\}$.
Then $C^{(0)}_\sigma$ is a normal subgroup of $C^{(1)}_\sigma$. 

\begin{lem}\label{lem: c0-dense}
$C^{(0)}_\sigma$ is dense in $C^{(1)}_\sigma$.
\end{lem}
\begin{proof}
Since $\si$ is trace preserving and free,
$\si$ is approximately inner and centrally free
by Corollary \ref{cor: app-cent-infinite}.
Then we can take a unitary
$U=(U^\nu)_\nu\in \meR_{0,1}^\omega\oti\lhG$
as in Theorem \ref{thm: app-cocycle} with $\Ga=\{e\}$.

Take $(\theta,v)\in C^{(1)}_\sigma$ and
choose $\{v^\nu\}_\nu\subs U(\meR_{0,1})$
with
$\theta=\lim\limits_{\nu\rightarrow\infty}\Ad v^\nu$.
Then
\begin{align*}
\Ad v \circ\sigma
&=
(\theta\otimes \id)\circ \sigma \circ \theta^{-1}
=
\lim_{\nu\rightarrow \infty}
\Ad (v^\nu\otimes 1)\circ \sigma \circ \Ad (v^\nu)^* 
\\
&=
\lim_{\nu\rightarrow \infty}
\Ad (v^\nu\otimes 1)\sigma((v^\nu)^*)\circ \sigma.
\end{align*}
Set $V:=(v^\nu)_\nu\in \meR_{0,1}^\omega$.
Then $w:=v^*(V\otimes 1)\sigma^\om(V^*)$
is a $\sigma^\om$-cocycle,
and
$U^*wU \in (\meR_{0,1})_\omega \otimes L^\infty(\bhG)$.
By Corollary \ref{cor: wU-cocycle},
there exists $z\in (\meR_{0,1})_\omega$
such that
$w=(z\otimes 1)\sigma^\om(z^*)$.
This implies
$(z^*V\otimes 1)\sigma^\om(V^*z)=v$.
Let $(\mu^\nu)_\nu $ be a representing
sequence of $z^*V$.
Then $\theta=\lim\limits_{\nu\rightarrow\omega}\Ad \mu^\nu$ and
$v=\lim\limits_{\nu\rightarrow \om}(\mu^\nu\otimes 1)\sigma((\mu^\nu)^*)$.  
\end{proof}

\begin{lem}
Suppose that
$\beta_x$ is constant,
that is, $\be_x=\beta_{x_0}$ for some $x_0\in\meG^{(0)}$.
Then there exist Borel families of automorphisms
$\{\sigma_x\}_{x\in\meG^{(0)}}\subset  \overline{\Int}(\meR_{0,1})$
and $\beta_x$-cocycles  $\{w^x\}_{x\in\meG^{(0)}}\subs U(\meR_{0,1}\oti\lhG)$
such that
$(\sigma_x\otimes \id)\circ\alpha_x\circ \sigma_{x}^{-1}
=\Ad w^x\circ \beta_x$.
\end{lem}
\begin{proof}
Set $N:=\meR_{0,1}\rtimes_{\be_{x_0}} \bhG$ and 
$N(x):=\meR_{0,1}\rtimes_{\al_x}\bhG$ for each $x\in X$.
Note that $N$ and $N(x)$ act on the
common Hilbert space $L^2(\meR_{0,1})\otimes L^2(\bhG)$.

Let $B_x$ be the set of pairs $(\sigma, v)$,
where $\sigma\in\Aut(\meR_{0,1})$ and $v$ is a 1-cocycle
for $\al_x$ such that 
$(\sigma^{-1}\otimes \id)\circ\Ad v \circ \al_x \circ \sigma=\be_{x_0}$.
Then $B_x$ is non-empty because of
Theorem \ref{thm:main} for $\meR_{0,1}$
and it is identified with the set of isomorphisms
from $N$ onto $N(x)$ preserving $\meR_{0,1}$.
Moreover, $B_x$ is a Polish group
because it is
identified with a closed subset of unitary maps $L^2(N)$ onto $L^2(N(x))$
which intertwine $N$ and $N(x)$,
preserve positive cones and $L^2(\meR_{0,1})$
and commute with modular conjugation\cite{Ha-st}.
Then thanks to
the measurable cross section theorem \cite[Theorem A.16, vol.I]{Tak-book},
we can
choose a Borel family $(\sigma_x, v^x)\in B_x$
as in the proof of \cite[Theorem IV.8.28, Proposition IV.8.29]{Tak-book}.
\end{proof}

\begin{thm}\label{thm:class-groupoid}
Let $\alpha$ and $\beta$ be $\bhG$-$\meG$-actions on $\meR_{0,1}$
as before.
Assume that $\beta_x$ is constant.
Then $\alpha$ and $\beta$ are cocycle conjugate as $\bhG$-$\meG$ actions.
\end{thm}
\begin{proof}
By the previous lemma, we can take Borel families
$\{\sigma_x\}_{x\in\meG^{(0)}}\subset  \overline{\Int}(\meR_{0,1})$
and
$\{w^x\}_{x\in\meG^{(0)}}\subs U(\meR_{0,1}\oti\lhG)$, $\beta_x$-cocycles
such that
$(\sigma_x\otimes \id)\circ\alpha_x\circ \sigma_{x}^{-1}
=\Ad w^x\circ \beta_x$.
By replacing
$\alpha_{r(\gamma)}\alpha_\gamma$ with
$(\sigma_{r(\gamma)}\otimes\id)\circ\alpha_{r(\gamma)}
\alpha_\gamma\circ \sigma_{s(\gamma)}^{-1}$,
we
may assume $\alpha_x=\Ad w^x\circ\beta_x$
and
$\mathrm{mod}(\alpha_\gamma)=\mathrm{mod} (\beta_\gamma)$.
Since $(\beta_\gamma\otimes \id)\beta_{s(\gamma)}
=\beta_{r(\gamma)}\beta_\gamma$,
we can regard $\beta_\gamma$ as a homomorphism from $\meG$
to
$\Aut_{\hat{\sigma}}(\meR_{0,1}\rtimes_{\beta_{x_0}} \bhG)$
by
$\gamma\rightarrow (\beta_\gamma,1)$.
We also have
\begin{align*}
(\alpha_\gamma\otimes \id)\beta_{s(\gamma)}
&=
(\alpha_\gamma\otimes \id)\circ \Ad w^{s(\gamma)*}\circ
\alpha_{s(\gamma)}\\ 
&= \Ad (\alpha_\gamma\otimes \id)(w^{s(\gamma)*})\circ
(\alpha_\gamma\otimes \id)
\alpha_{s(\gamma)}\\ 
&=
\Ad (\alpha_\gamma\otimes \id)(w^{s(\gamma)*})\circ
\alpha_{r(\gamma)}\alpha_\gamma\\ 
&=
\Ad (\alpha_\gamma\otimes \id)(w^{s(\gamma)*})w^{r(\gamma)}
\circ\beta_{r(\gamma)}\alpha_\gamma,
\end{align*}
where
$(\alpha_\gamma\otimes \id)\left(w^{s(\gamma)*}\right)
w^{r(\gamma)}$ is a $\beta_{r(\gamma)}$-cocycle.
So we can regard $\alpha$ as a homomorphism from $\meG$
to $\Aut_{\hat{\sigma}}(\meR_{0,1}\rtimes_{\beta_{x_0}} \bhG)$
by 
$\gamma\rightarrow
(\alpha_\gamma,(\alpha_\gamma\otimes \id)(w^{s(\gamma)*})w^{r(\gamma)})$.
Here note
that $C^{(1)}_{\beta_x}=C^{(1)}_{\beta_{x_0}}$
because $\beta_x$ is constant.

We next show that
$\alpha_\gamma\equiv\beta_\gamma
\,\,
\mathrm{mod}
(C^{(1)}_{\beta_{r(\gamma)}})$.
Since $\mathrm{mod}(\alpha_\gamma)=\mathrm{mod}(\beta_\gamma)$,
it is clear that 
$\alpha_\gamma\beta_\gamma^{-1}
\in
\overline{\Int}(\meR_{0,1})$.
By the above computation, we also have the following:
\begin{align*}
(\alpha_\gamma\beta^{-1}_\gamma\otimes \id)\circ\beta_{r(\gamma)}
&=
(\alpha_\gamma\otimes \id)
\circ\beta_{s(\gamma)}\beta^{-1}_\gamma\\
&=
\Ad (\alpha_\gamma\otimes \id)(w^{s(\gamma)*})w^{r(\gamma)}
\circ\beta_{r(\gamma)}\alpha_\gamma\beta_\gamma^{-1} .
\end{align*}
Hence 
$\alpha_\gamma\beta_\gamma^{-1}\in C^{(1)}_{\beta_{r(\gamma)}}$.
Applying Theorem \ref{lem:Krieger} and Lemma \ref{lem: c0-dense}
to the two maps $\al,\be\col \meG\ra C_{\be_{x_0}}^{(1)}$
and $C_{\be_{x_0}}^{(0)}$,
we get Borel maps
$\meG^{(0)}\ni x\mapsto(\sigma_x, v^x)\in C^{(1)}_{\beta_x}$
and $u\col \meG\ni\ga\mapsto u_\gamma\in U(\meR_{0,1})$
such that
\begin{align*}
\lefteqn{(\Ad u_\gamma, u_\gamma\beta_{r(\gamma)}(u_\gamma^*))
\cdot(\beta_\gamma, 1)} \\
&=
(\sigma_{r(\gamma)}, v^{r(\gamma)})
\cdot
(\alpha_\gamma,
(\alpha_\gamma\otimes \id)
(w^{s(\gamma)*})w^{r(\gamma)})
\cdot
(\sigma_{s(\gamma)}^{-1},
\sigma_{s(\gamma)}^{-1}(v^{s(\gamma)*})).
\end{align*}
The left hand side is equal to
$(\Ad u_\gamma\circ \beta_\gamma , u_\gamma\beta_{r(\gamma)}(u_\gamma^*))$.
We compute the right hand side.
For simplicity we write $\alpha_\gamma$ for $\alpha_\gamma\otimes \id$
and so on.
\begin{align*}
\lefteqn{
(\sigma_{r(\gamma)}, v^{r(\gamma)})
\cdot
(\alpha_\gamma,
\alpha_\gamma(w^{s(\gamma)*})w^{r(\gamma)})
\cdot
(\sigma_{s(\gamma)}^{-1},
\sigma_{s(\gamma)}^{-1}(v^{s(\gamma)*}))}
\\
&= 
\left(\sigma_{r(\gamma)}\alpha_\gamma,
\sigma_{r(\gamma)}\left(\alpha_\gamma(w^{s(\gamma)*})w^{r(\gamma)}\right)
v^{r(\gamma)}
\right)
\cdot
\left(\sigma_{s(\gamma)}^{-1},
\sigma_{s(\gamma)}^{-1}(v^{s(\gamma)*})\right) \\
&= 
\left(\sigma_{r(\gamma)}\alpha_\gamma\sigma_{s(\gamma)}^{-1}, 
\sigma_{r(\gamma)}\alpha_\gamma\sigma_{s(\gamma)}^{-1}(v^{s(\gamma)*})
\sigma_{r(\gamma)}\left(\alpha_\gamma(w^{s(\gamma)*})w^{r(\gamma)}\right)
v^{r(\gamma)}
\right).
\end{align*}

By comparing the first component,
we have $\Ad u_\gamma \circ\beta_\gamma =
\sigma_{r(\gamma)}\circ \alpha_\gamma \circ \sigma^{-1}_{s(\gamma)}$.
Since $\meG$ is generated by a single transformation,
we may assume that
$u_\gamma$ is a $\beta$-cocycle.  

The second component is computed as follows:
\begin{align*}
\lefteqn{ \sigma_{r(\gamma)}
\alpha_\gamma\sigma_{s(\gamma)}^{-1}(v^{s(\gamma)*})
\sigma_{r(\gamma)}\left(\alpha_\gamma(w^{s(\gamma)*})w^{r(\gamma)}\right)v^{r(\gamma)}}\\
 &=
 \Ad u_\gamma\beta_\gamma(v^{s(\gamma)*})\cdot
\sigma_{r(\gamma)}\alpha_\gamma(w^{s(\gamma)*})
\sigma_{r(\gamma)}(w^{r(\gamma)})v^{r(\gamma)} \\
 &=
 \Ad u_\gamma\beta_\gamma\left(v^{s(\gamma)*}
\sigma_{s(\gamma)}(w^{s(\gamma)*})\right)
\sigma_{r(\gamma)}(w^{r(\gamma)})v^{r(\gamma)}.
\end{align*}
Set $u^x:=\sigma_x(w^x)v^x$.
By using
$\sigma_x\circ \be_x \circ
\sigma_x^{-1}=\Ad v^x \circ\be_x$,
it follows that $u^x$ is a $\beta_x$-cocycle and
$\sigma_x\circ \alpha_x \circ\sigma_x^{-1}
=\Ad u^x \circ\beta_x$.
By comparing the second component, we have 
$\beta_\gamma(u_\gamma^*)=
\beta_\gamma(u^{s(\gamma)*})u_\gamma^*u^{r(\gamma)}$,
and equivalently
$u^{r(\gamma)}\beta_\gamma(u_\gamma)=
u_\gamma \beta_\gamma(u^{s(\gamma)})$.
This shows that
$u_{\cdot,\gamma}$ is a $\beta$-cocycle, and
$\sigma_{r(\gamma)}\circ \alpha_{\cdot,\gamma}
\circ\sigma_{s(\gamma)}
=\Ad u_{\cdot,\gamma}\circ\beta_{\cdot,\gamma}$.
Thus
$\alpha$ and $\beta$ are cocycle conjugate.
\end{proof}

\noindent
$\bullet$
\textit{Proof of Theorem \ref{thm:main} for type III$_0$ factors.}
Let $M$, $\alpha$ and $\alpha^{(0)}$ be as in Theorem \ref{thm:main}.
Then $\tal$ and $\wdt{\id_M\oti \al^{(0)}}=\id_\tM\oti \al^{(0)}$
act on $Z(\tM)$ trivially
and free on $\tM$ by Theorem \ref{thm: app-cent-prop}.
By using an isomorphism $\meR_{0,1}\cong\meR_{0,1}\oti \meR_0$,
we see that $(\tal_x)_{x\in X}$
and $(\id_{\tM(x)}\oti\al^{(0)})_{x\in X}$ are satisfying
the condition of Theorem \ref{thm:class-groupoid}.
Then the two $\bhG$-$\meG$-actions on $\meR_{0,1}$
arising from
$\tal\th$ and $\th\otimes \alpha^{(0)}$
are cocycle conjugate.
This implies
the cocycle conjugacy of the $\bhG\times\R$-actions
$\tal\theta$ and $(\th\otimes \alpha^{(0)})$
by Lemma \ref{lem: al-be-cocycle}.
Considering the
partial crossed product by $\theta$,
we get $\wdt{\tal}\sim\id_{\widetilde{\tM}}\otimes \alpha^{(0)}$
as in the proof of Lemma \ref{lem: G1-G2-al-be}. 
Thus $\al$ and $\id_{M}\oti\al^{(0)}$ are cocycle conjugate
by Lemma \ref{lem: second-cocycle}.
\hfill$\Box$

\begin{rem}
In general, there may appear some obstructions
in combining the $\bhG$-part and the $\meG$-part.
In \cite{JT, Su-Tak-RIMS, Kw-Tak},
model actions absorbing obstructions are constructed.
In our case, however,
we are treating only free actions, and no obstructions appear.
Hence we do not need such model actions.
\end{rem}

\section{Classification for type III$_1$ case}\label{sec:III1}
\subsection{Basic results on canonical extensions}
In \S \ref{sec:lamclass}, we obtained the classification of
approximately inner and centrally free actions of an amenable discrete Kac
algebra on the injective factor of type III$_\lambda$.
Using this result
together with ideas of \cite{Co-III1,Ha-III1}
(also see \cite{M-III1}),
we classify actions on the injective factor of type III$_1$.

Let $M\cong\meR_\infty$ and $\varphi$ be a faithful normal state on $M$.
Fix $T>0$.
Set $N:=M \rtimes_{\sigma^\varphi_T}\Z$,
which is an injective factor of
type III$_\lambda$, $\lambda:=e^{-\frac{2\pi}{T} }$,
and let $U\in N$ be the unitary implementing $\si_T^\vph$.
The dual action of the torus $\T=\R/2\pi\Z$ is denoted by $\th$,
which acts on $U$ by $\th_t(U)=e^{-\sqrt{-1}t}U$ for $t\in\T$.
Using the averaging expectation $E_\th\col N\ra M$ by $\th$,
we extend $\vph$ to $\hvph:=\vph\circ E_\th$.
Throughout this section, we keep these notations.

Now we introduce the extension $\hat{}\,\col \End_0(M)\ra \End_0(N)$
defined by
\begin{align*}
&\hrho(x)=\rho(x)\quad\mbox{for all }x\in M;
\\
&\hrho(U)=d(\rho)^{iT}[D\vph\circ\ph_\rho:D\vph]_T U.
\end{align*}
Note that $\hrho$ is one of the variants of the canonical extension.
Indeed, regarding $N\subs \tM$ by $U=\la^\vph(T)$,
we see that $\hrho=\trho|_N$.

\begin{lem}\label{lem: rho-hrho-app}
For any $\rho\in\End_0(M)$, $\mo(\hrho)=\id$.
\end{lem}
\begin{proof}
Since $\si_T^\hvph=\Ad U$.
We can take a positive operator $h$ affiliated with $N_\hvph$
such that $U=h^{iT}$.
We set $\ps:=\hvph_{h^{-1}}$, whose modular automorphism
has the period $T$.
Note that $E_\th\circ\ph_\hrho=\ph_\hrho\circ E_\th=\ph_\rho\circ E_\th$
because $\ph_{\hrho}|_M=\ph_\rho$
(see Theorem \ref{thm: extension-app-cent} (2)).
Then we can compute $[D\ps\circ\ph_\hrho:D\ps]_T$ as follows:
\begin{align*}
[D\ps\circ\ph_\hrho:D\ps]_T
&=
[D\ps\circ\ph_\hrho:D\hvph\circ\ph_\hrho]_T
[D\hvph\circ\ph_\hrho:D\hvph]_T
[D\hvph:D\ps]_T
\\
&=
\hrho([D\ps:D\hvph]_T)
[D\hvph\circ\ph_\hrho:D\hvph]_T
[D\hvph:D\ps]_T
\\
&=
\hrho(U^*)
[D\vph\circ\ph_\rho\circ E_\th:D\vph\circ E_\th]_T
U
\\
&=
\hrho(U^*)
[D\vph\circ\ph_\rho:D\vph]_T
U
=d(\rho)^{-iT}.
\end{align*}
By \cite[Theorem 2.8]{Hi}
 $d(\rho)=d(\hrho)$, so
the above equality means $\mo(\hrho)=\id$.
\end{proof}

We denote by $\End_0^\th(N)$ the set of endomorphisms with finite indices
on $N$ which commute with $\th$,
and by $\ker(\mo)$ the set of 
endomorphisms with finite indices in $\End(N)_{\rm CT}$
with trivial Connes-Takesaki modules.
Note that $\hrho\in\End_0^\th(N)$ for all $\rho\in\End_0(N)$.
We will analyze the relative commutant $\hrho(N)'\cap N$,
which admits the torus action $\th$.
Define the following linear space for each $n\in\Z$:
\[I_n
:=\{a\in
\hat{\rho}(N)'\cap N\mid \theta_t(a)=e^{\sqrt{-1}nt}a \mbox{ for all }
t\in \mathbb{T}\}.
\]

\begin{lem}\label{lem: hrho-relcom}
For each $n\in\Z$, one has
$I_n=U^{-n}(\rho,\sigma^{\varphi}_{nT}\rho)$.
\end{lem}
\begin{proof}
Take $a\in I_n$.
Then $\theta_t(U^n a)=U^n a$ for $t\in\T$, and $b:=U^n a\in M$.
We check $b\in (\rho,\si_{nT}^\vph\rho)$ as follows:
for $x\in M$,
\[
b\rho(x)=U^n\rho(x)a=U^n\rho(x)U^{n*}U^na=\sigma^\varphi_{nT}(\rho(x))b.
\]
Hence $I_n\subs U^{-n}(\rho,\sigma^{\varphi}_{nT}\rho)$.

Next we show the converse inclusion.
Set a unitary
$
u:=d(\rho)^{iT}\left[D\varphi\circ \ph_\rho: d\varphi\right]_T
$.
Take $b\in (\rho,\sigma^{\varphi}_{nT}\rho)$.
By direct computation, we see that $U^{-n}b\in I_n$
if and only if
$b=\si_{nT}^\vph(u)\si_T^\vph(b)u^*$ holds.
Consider the map
$\mu\col(\rho,\sigma^{\varphi}_{nT}\rho)\ni b
\mapsto
\si_{nT}^\vph(u)\si_T^\vph(b)u^*
\in (\rho,\sigma^{\varphi}_{nT}\rho)$.
Then $\mu$ is a well-defined unitary, here
the inner product is given by
$\langle a,b \rangle=\ph_\rho(b^*a)$
for $a,b\in(\rho,\sigma^\varphi_T\rho)$.
Hence it suffices to prove
that $\mu$ is actually an identity map.
Since $(\rho,\sigma^{\varphi}_{nT}\rho)$ is finite dimensional,
it is spanned by eigenvectors of $\mu$.
Let $b$ be an eigenvector
$\mu(b)=e^{\sqrt{-1}s}b$ for some $s\in[0,2\pi)$.
We claim that $U^{-n}b\in (\theta_{-s}\hat{\rho},\hat{\rho})$.
For $x\in M$, we have the following.
\[
U^{-n}b \theta_{-s}(\hrho(x))
=U^{-n}\si_{nT}^\vph(\rho(x))b  
=\rho(x)U^{-n} b
=\hrho(x)U^{-n}b.
\]
We also have the following.
\begin{align*}
U^{-n}b\hat{\theta}_{-s}(\hat{\rho}(U))
&=U^{-n}b \hat{\theta}_{-s}(uU)
= U^{-n}b \cdot e^{{\sqrt{-1}}s}uU
=U^{-n}\mu(b)uU\\
&=uU^{-n}\si_T^\vph(b)U
=uU^{1-n}b
=\hat{\rho}(U) U^{-n}b.
\end{align*} 
Thus we have verified the claim.
By the Frobenius reciprocity,
$\mathrm{dim}(\theta_{-s}\hat{\rho},\hat{\rho})
=\mathrm{dim}(\theta_{-s},\hat{\rho}\bar{\hat{\rho}})$,
and hence
$\hat{\rho}\bar{\hat{\rho}}$
contains $\theta_{-s}$ as an irreducible component.
However by the previous lemma,
$\hat{\rho}$ has trivial Connes-Takesaki module,
and $\mo(\th_{-s})=\mo(\hat{\rho}\bar{\hat{\rho}})=\id$.
This is possible only if $s=0$.
Therefore $\mu=\id$.
\end{proof}

\begin{thm}\label{thm: extension-app-cent}
Let $\rho\in\End_0(M)$.
Then one has the following:
\begin{enumerate}
\item
$\hrho$ is irreducible if and only if
$\rho$ is irreducible.
In this case, the inclusion $\rho(M)\subs N$ is irreducible;

\item
The standard left inverse $\ph_\hrho$ is given by:
\[
\ph_\hrho(xU^n)=d(\rho)^{-inT}\ph_\rho(x[D\vph\circ\ph_\rho:D\vph]_{nT}^*)U^n
\quad\mbox{for all }x\in M, n\in\Z;
\]

\item
The extension $\hat{\cdot}$ is a bijection
from $\End_0(M)$ onto $\End_0^\th(N)\cap\ker(\mo)$;

\item
$\hrho\in\Cnd(N)$ if and only if $\rho\in \Cnd(M)$.
\end{enumerate}
\end{thm}
\begin{proof}
(1) If $\hat{\rho}$ is irreducible, then 
$I_0=\mathbb{C}$, and 
$(\rho,\rho)=\mathbb{C}$ follows from the previous lemma.
Conversely if $\rho$ is irreducible,
then $\rho\bar{\rho}$ contains
no nontrivial modular automorphisms
because the $T$-set $T(M)$ is trivial.
This means
$(\rho,\rho)=\mathbb{C}$,
and $(\si_{nT}^\vph \rho,\rho)=0$ for $n\ne 0$.
Hence $I_0=\mathbb{C}$, and $I_n=0$ for $n\neq 0$.
Since $\hat{\rho}(N)'\cap N$
is densely spanned by $\{I_n\}_{n\in\Z}$,
$\hat{\rho}$ is irreducible.

We prove the latter statement in (1).
Take $x\in \rho(M)'\cap N$
and let $x=\sum_{n\in\Z}x_n^* U^n$ be the formal decomposition.
Then for each $n\in\Z$,
$x_n\in (\rho,\si_{nT}^\vph \rho)$.
From the above argument, $x_0\in\C$ and $x_n=0$ for $n\neq0$.
Hence $\rho(M)'\cap N=\C$.

(2)
By \cite[Lemma 3.5]{M-T-endo-pre},
the map $\ph_\hrho$ is well-defined.
By \cite[Theorem 2.8]{Hi}, $\hat{\rho}\phi_{\hrho}$ is the minimal
conditional expectation, and it
follows that $\phi_{\hrho}$ is standard.

(3)
Let $\ps$ be a periodic weight constructed as in
the proof of Lemma \ref{lem: rho-hrho-app}.
By Lemma \ref{lem: rho-hrho-app}, we see that
$\hrho\in \End_0^\th(N)\cap \ker(\mo)$.
So, the given map is well-defined.
We show that the map is a bijection.
Clearly it is injective, and it suffices to show the surjectivity.
Let $\si\in \End_0^\th(N)\cap \ker(\mo)$.
Since $\mo(\si)=\id$, we have
$d(\si)^{iT}[D\ps\circ\ph_\si:D\ps]_T=1$.
This is equivalent to
\begin{equation}\label{eq: siU}
\si(U)=d(\si)^{iT}[D\vph\circ\ph_\si|_M:D\vph]_T U.
\end{equation}
Set $\rho=\si|_M$.
The action $\th$ of $\T$ on $\si(N)$ is dominant,
and $d(\sigma)=d(\rho)$ follows from \cite[Theorem 2.8 (2)]{Hi}.
In the proof of \cite[Theorem 2.8 (2)]{Hi},
it is also shown that $\si\circ\ph_\si|_M$ is the minimal expectation
from $M$ onto $\rho(M)$.
Hence $\ph_\rho=\ph_\si|_M$.
Then the equality (\ref{eq: siU}) yields $\si=\wdh{\si|_M}$.

(4)
Let $\rho\in\Cnd(M)$.
We may and do assume that $\rho$ is irreducible.
Then by \cite[Theorem 4.12]{M-T-endo-pre},
there exists $t\in\R$ such that $[\rho]=[\si_t^\vph]$.
Then $[\hrho]=[\wdh{\si_t^\vph}]=[\si_t^\hvph]$, and $\hrho\in\Cnd(N)$.

Conversely we assume that $\hrho\in\Cnd(N)$.
Thanks to (1),
we may and do assume that $\hrho$ is irreducible.
By \cite[Theorem 4.12]{M-T-endo-pre},
there exist $t\in \R$ and $u\in U(N)$
such that $\hrho=\Ad u\circ\si_t^\hvph$.
Considering the formal decomposition of $u$,
we see that $(\si_{nT+t}^\vph,\rho)\neq 0$ for some $n\in\Z$.
Since $\rho$ is irreducible by (1),
this means $[\rho]=[\si_{nT+t}^\vph]$, and $\rho\in\Cnd(M)$.
\end{proof}

Let $K$ be a finite dimensional Hilbert space.
Following the procedure introduced in \S\ref{sec:appendix},
we define the canonical extension $\bbe\in\Mor(N,N\oti B(K))$
for $\be\in \Mor_0(M,M\oti B(K))$ by
\begin{align*}
&\bbe(x)=\be(x)\quad\mbox{for all }x\in M;
\\
&
\bbe(U)=d(\be)^{iT}[D\vph\circ\Ph^\be:D\vph\oti \Tr_K]_T (U\oti1).
\end{align*}

By $\Mor_0^\th(N,N\oti B(K))$,
we denote the set of homomorphisms in $\Mor_0(N, N\oti B(K))$
commuting with $\th$.
The following is a direct consequence of the previous theorem.
The fourth statement follows from the third one and
Theorem \ref{thm: app-cent-prop}.

\begin{lem}\label{lem:canonical}
Let $K$ be a finite dimensional Hilbert space.
Then one has the following:
\begin{enumerate}
\item
Let $\beta\in \Mor_0(M, M\otimes B(K))$.
Then $\bbe$ is irreducible if and only if $\be$ is irreducible.
In this case, the inclusion $\be(M)\subs N\oti B(K)$ is irreducible;

\item
Let $\beta\in \Mor_0(M, M\otimes B(K))$.
Then $d(\bbe)=d(\be)$ and the standard left inverse $\Ph^\bbe$ is
given by the following equality:
for $x\in M\oti B(K)$ and $n\in\Z$,
\[
\Ph^\bbe(x (U^n\oti1))
=
d(\be)^{-inT}\Ph^\be(x[D\vph\oti\Tr_K:D\vph\circ\Ph^\be]_{nT})
U^n;
\]

\item
The extension $\ovl{\cdot}$ is a bijection from $\Mor_0(M,M\oti B(K))$
onto $\Mor_0^\th(N,N\oti B(K))\cap \ker(\mo)$;

\item
Let $\beta\in \Mor_0(M, M\otimes B(K))$.
If $d(\be)=\dim(K)$,
then $\bbe\in\oInt(N,N\oti B(K))$;

\item
Let $\beta\in \Mor_0(M, M\otimes B(K))$.
Then $\be\in \Cnt(M,M\oti B(K))$
if and only if
$\bbe\in\Cnt(N,N\oti B(K))$.
\end{enumerate}
\end{lem}

\subsection{Reduction to the classification of 
actions on $\meR_\la$}\label{sec:1-lambda}

Let $\alpha$ a centrally free cocycle action of $\bhG$
on $M\cong\meR_\infty$.
Then
$\alpha$ is automatically approximately inner
from Corollary \ref{cor: app-cent-infinite}.
For each $\pi\in\IG$, we consider the canonical extension
$\bal_\pi\in \Mor_0(N,N\oti B(K))$ as before.
Then
$\bal$ is a cocycle action on $N$ with the same 2-cocycle.

\begin{prop}\label{prop:ext}
Let $\alpha$ be a centrally free cocycle action of
$\bhG$ on $M$.
Then
$\bal$ is an approximately inner and
centrally free cocycle action of $\bhG$ on $N$.
\end{prop}
\begin{proof}
For each $\pi\in\IG\setm\{\btr\}$,
$\bal_\pi$ is approximately inner
and
centrally nontrivial by Lemma \ref{lem:canonical} (4) and (5).
Since $\al_\pi$ is properly outer,
$\al_\pi$ is irreducible \cite[Lemma 2.8]{M-T-CMP}.
Hence so is $\bal_\pi$ by Lemma \ref{lem:canonical} (1).
Then by \cite[Lemma 8.3]{M-T-CMP},
$\bal_\pi$ is properly centrally non-trivial.
Thus the cocycle action $\bal$ is centrally free.
\end{proof}

Our main theorem of this section is the following:

\begin{thm}\label{thm:class1}
Let $\alpha$ be a centrally free action of $\bhG$ on $M$.
Then the $\bhG\times\T$-action $\bal\theta$ on $N$
is cocycle conjugate to
$\theta\otimes \alpha^{(0)}$,
where $\al^{(0)}$ is a free action of $\bhG$ on $\meR_0$.
\end{thm}

\noindent
$\bullet$
\textit{Proof of Theorem \ref{thm:main} for $\meR_\infty$.}

Since the natural extension of $\bal$ to $N\rti_\th\T$
is cocycle conjugate to $\al$ by Takesaki duality,
we see that Theorem \ref{thm:class1} implies
Theorem \ref{thm:main}
considering the partial crossed product by $\th$ as before.
\hfill$\Box$
\\

The rest of this section is devoted to show Theorem \ref{thm:class1}.
The essential part of our proof is the model
action splitting result Proposition \ref{prop:split}.
The following lemma shows that
the canonical extension well behaves to cocycle perturbations.

\begin{lem}\label{lem: can-cocycle-stable}
For $i=1,2$,
let $M^i$ be a type III$_1$ factor,
$\vph^i\in W(M^i)$
and $(\al^i,u^i)$ be a cocycle action of $\bhG$ on $M^i$.
We set $N^i:=M^i\rti_{\si_T^{\vph^i}}\Z$ and
the dual action $\th^i:=\wdh{\si_T^{\vph^i}}$.
If
$(\al^1,u^1)$ is cocycle conjugate to $(\al^2,u^2)$,
then
there exists an isomorphism $\Ps\col N^1\ra N^2$
and a unitary $v\in M^2\oti \lhG$
such that
\begin{itemize}
\item
$\Ps\circ\th_t^1=\th_t^2\circ\Ps$ for all $t\in\T$;

\item
$(\Ps\oti\id)\circ\bal^1\circ\Ps^{-1}
=
\Ad v \circ\bal^2$;

\item
$(\Ps\oti\id\oti\id)(u^1)=(v\oti1)\al^2(v)u^2(\id\oti\De)(v^*)$.
\end{itemize}
In particular,
the $\bhG\times\T$-cocycle action
$\bal^1\th^1$ is cocycle conjugate to $\bal^2\th^2$.
\end{lem}
\begin{proof}
Since $(\al^1,u^1)\sim(\al^2,u^2)$,
there exists an isomorphism $\Ps_0\col M^1\ra M^2$
and $v\in M^2\oti\lhG$ such that
\begin{itemize}
\item
$(\Ps_0\oti\id)\circ\al^1\circ\Ps_0^{-1}
=
\Ad v \circ\al^2$;

\item
$(\Ps_0\oti\id\oti\id)(u^1)=(v\oti1)\al^2(v)u^2(\id\oti\De)(v^*)$.
\end{itemize}
We set $\ps^2:=\vph^1\circ\Ps_0^{-1}\in W(M^2)$.
Then there exists an isomorphism
$\Ps\col N^1\ra M^2\rti_{\si_T^{\ps^2}}\Z$
such that
$\Ps(xU^{\vph^1})=\Ps_0(x)U^{\ps^2}$,
where $U^{\vph^1}$ and $U^{\ps^2}$ are the implementing unitaries
for $\si_T^{\vph^1}$ and $\si_T^{\ps^2}$, respectively.
Then $\Ps$ intertwines the dual actions.
Regard $M^2\rti_{\si_T^{\ps^2}}\Z=N^2$ in the core $\wdt{M}^2$.
It suffices to show the second equality holds on
the implementing unitary $U^{\ps^2}$.
This is checked as follows:
for $\pi\in\IG$, we have
\begin{align*}
&\hspace{16pt}(\Ps\oti\id)\circ\bal_\pi^1\circ\Ps^{-1}(U^{\ps^2})
\\
&=
(\Ps\oti\id)(\bal_\pi^1(U^{\vph^1}))
\\
&=
(\Ps\oti\id)
([D\vph^1\circ\Ph_\pi^{\al^1}:D\vph^1\oti\tr_\pi]_T
(U^{\vph^1}\oti1))
\\
&=
[D\vph^1\circ\Ph_\pi^{\al^1}\circ(\Ps_0^{-1}\oti\id)
:D\vph^1\circ\Ps_0^{-1}\oti\tr_\pi]_T
(U^{\ps^2}\oti1)
\\
&=
[D\ps^2\circ\Ps_0\circ\Ph_\pi^{\al^1}\circ(\Ps_0^{-1}\oti\id)
:D\ps^2\oti\tr_\pi]_T
(U^{\ps^2}\oti1)
\\
&=
[D\ps^2\circ\Ph_\pi^{(\Ps_0\oti\id)\circ\al^1\circ\Ps_0^{-1}}
:D\ps^2\oti\tr_\pi]_T
(U^{\ps^2}\oti1)
\\
&=
[D\ps^2\circ\Ph_\pi^{\Ad v\circ\al^2}
:D\ps^2\oti\tr_\pi]_T
(U^{\ps^2}\oti1)
\\
&=
[D\ps^2\circ\Ph_\pi^{\al^2}\circ\Ad v_\pi^*
:D\ps^2\oti\tr_\pi]_T
(U^{\ps^2}\oti1)
\\
&=
v_\pi\si_T^{\ps\circ\Ph_\pi^{\al^2}}(v_\pi^*)
[D\ps^2\circ\Ph_\pi^{\al^2}
:D\ps^2\oti\tr_\pi]_T
(U^{\ps^2}\oti1)
\\
&=
v_\pi
[D\ps^2\circ\Ph_\pi^{\al^2}:D\ps^2\oti\tr_\pi]_T
\si_T^{\ps^2\oti\tr_\pi}(v_\pi^*)
(U^{\ps^2}\oti1)
\\
&=
v_\pi
[D\ps^2\circ\Ph_\pi^{\al^2}:D\ps^2\oti\tr_\pi]_T
(U^{\ps^2}\oti1)
v_\pi^*
\\
&=
\Ad v_\pi\circ\bal^2(U^{\ps^2}).
\end{align*}
\end{proof}

The following lemma is an equivariant version of
\cite[Lemma I.2]{Co-III1}.
Recall that $\al^{(0)}$ is a free action of $\bhG$
on $\meR_0$.

\begin{lem}
One has the following:

\begin{enumerate}
\item
Let $\de$ be an action of $\bhG$ on $\meR_\la$
and $\ga\in\Aut(\meR_\la)$ such that
\begin{itemize}
\item
$\de$ commutes with $\ga$;
 
\item
$\meR_\la\rtimes_\gamma \Z \cong \meR_\la$;

\item
The natural extension $\ovl{\de}$ of $\de$
to $\meR_\la\rtimes_\gamma \Z$ is
approximately inner and centrally free;

\item
The $\bhG\times \T$-action $\ovl{\de} \hat{\gamma}$
is centrally free on $\meR_\la\rti_\ga\Z$.
\end{itemize}
Then $\bhG\times \Z$-action $\de\gamma$ on $\meR_\la$
is cocycle conjugate to
$\id_{\meR_\la}\otimes \gamma^{(0)}\otimes \alpha^{(0)}$, 
where $\gamma^{(0)}$ is an aperiodic automorphism on $\meR_0$.

\item
Let $\de$ be an action of $\bhG$ on $\meR_\la$,
and $\beta$ an action of $\T$ on $\meR_\la$
such that
\begin{itemize}

\item $\de$ is approximately inner and centrally free;

\item
$\de$ commutes with $\be$;

\item
The $\bhG\times\T$-action $\de\beta$
is centrally free on $\meR_\la$;

\item
$\meR_\la\rtimes_\beta \T\cong \meR_\la$.
\end{itemize}
Then
the $\bhG\times\T$-action $\de\beta$ is cocycle conjugate to
$\id_{\meR_\la}\otimes \widehat{\gamma^{(0)}} \otimes \alpha^{(0)}$.
\end{enumerate}
\end{lem}
\begin{proof}
(1)
Set $R:=\meR_\la$
which admits the $\bhG\times\Z$-action $\de\ga$.
Let $W\in R\rti_\ga\Z$ be the unitary implementing $\ga$.

\noindent{\bf Step 1.}
We show that $\ga$ is approximately inner and centrally free.

This follows from \cite[Lemma I.2]{Co-III1}.
Also see \cite[Lemma XVIII.4.18]{Tak-book}.

\noindent{\bf Step 2.}
We show that the $\bhG\times\Z$-action $\de\ga$
is approximately inner.

It is known that $R$ and $R\rtimes_\gamma \Z$ have the common flow of
weights \cite{Kw-Tak, Se-flow}.
Since $\ovl{\de}$ is approximately inner on $R\rtimes_\gamma \Z$,
$\mo(\ovl{\de})=\mo(\de)=\id$.
Hence $\de$ is approximately inner on $R$
by Theorem \ref{thm: app-cent-prop},
and so is $\de\gamma$.

\noindent{\bf Step 3.}
We show that the $\bhG\times\Z$-action $\de\ga$
is centrally free.

Fix a generalized trace $\ps$ on $R$.
Note that our assumption of (1) is satisfied for
any perturbed actions of $\de\ga$.
By Lemma \ref{lem: gen-tr-cocycle},
we may and do assume that $\ps$ is invariant by
$\de\gamma$.

For each $\pi\in\IG$,
we set
$Q_\pi:=\de_\pi(R)'\cap (R\rti_\ga\Z\oti B(H_\pi))$.
We can show that $Q_\pi$ is finite dimensional
in a similar way to the proof
of Theorem \ref{thm: extension-app-cent} (1),
where the freeness of $\ga$ is crucial.
Also we can show that $\Ad (W\oti1)$ ergodically acts on $Q_\pi$,
and the torus action $\hga$ preserves $Q_\pi$.
Therefore, there exist atoms $\{p_i\}_{i=1}^m\subs Q_\pi$
such that $p_i\in R\oti B(H_\pi)$, $\ga(p_i)=p_{i+1}$
for $1\leq i\leq m-1$ and
\begin{equation}\label{eqn: Q-p}
Q_\pi=\de_\pi(R)'\cap (R\oti B(H_\pi))=\C p_1+\cdots+\C p_m.
\end{equation}
Take an isometry $V_1\in N\oti B(H_\pi)$ such that $V_1V_1^*=p_1$.
Set $V_i:=(\ga^{i-1}\oti\id)(V_1)$ for $1\leq i\leq m$.
Then we have $V_i V_i^*=p_i$.

Now assume that
$\de_\pi\gamma^n$ is not properly centrally non-trivial
on $R$ for some $\pi \in \IG$ and $n\in \Z$.
Set $\be_i:=V_i^* \de_\pi(\ga^n(\cdot))V_i$ for each $i$.
Then $\be_i\in \Mor_0(R,R\oti B(H_\pi))$ is irreducible
and $\de_\pi\ga^n=\sum_{i=1}^m V_i \be_i(\cdot)V_i^*$.
Then $\be_i$ is not properly centrally non-trivial
for some $i$.
We may and do assume $i=1$.
Since $\be_1$ is irreducible,
$\be_1$ is centrally trivial \cite[Lemma 8.3]{M-T-CMP}.
Then by Corollary \ref{cor: app-cent-infinite},
we see that $\be_1=\Ad u\circ \si_{t_0}^\ps$
for some $u\in U(R)$ and $t_0\in\R$.

So we have $\de_\pi (\gamma^n(x))V_1 u=V_1 u \sigma^\ps_{t_0}(x)$
for $x\in R$.
Applying $\ga^{i-1}$ to the both sides,
we have $\de_\pi\big{(}\ga^n(\ga^{i-1}(x))\big{)}V_i \ga^{i-1}(u)
=V_i\ga^{i-1}(u) \si_{t_0}^\ps(\ga^{i-1}(x))$
for $x\in R$, where we have used
the fact that $\ga$ commutes with $\si^\ps$.
By definition of $\be_i$,
we obtain
$\be_i(\ga^{i-1}(x))\ga^{i-1}(u)
=\ga^{i-1}(u)\si_{t_0}^\ps(\ga^{i-1}(x))$,
that is, $\be_i=\Ad \ga^{i-1}(u)\circ\si_{t_0}^\ps$.
Hence $\{\be_i\}_{i=1}^m$ define the equivalent sectors.
By (\ref{eqn: Q-p}), this is possible when $m=1$,
that is, $\de_\pi\ga^n$ is irreducible.
Hence we may assume that $\de_\pi\ga^n=\Ad u\circ \si_{t_0}^\ps$.

Since $\ps$ is invariant under
$\de\gamma$, $u\in R_\ps$, and $\gamma(u)=e^{\sqrt{-1}s_0}u$
for some $s_0\in\R$.
We can check that
$\ovl{\de}_\pi\circ\Ad W^n
=\Ad u\circ\si_{t_0}^{\hat{\ps}}\circ\hat{\ga}_{-s_0}$
holds on $R\rti_\ga\Z$ by direct computation.
So $\ovl{\de}_\pi\hat{\ga}_{s_0}$ is centrally trivial,
and the assumption (1) yields $\pi=\btr$ and $s_0=0$,
and $\ga^n=\Ad u\circ \si_{t_0}^\ps$.
Then we get $n=0$ from central freeness of $\ga$.

\noindent{\bf Step 4.}
We use the classification result for actions on $\meR_\la$.

The $\bhG\times\Z$-action $\de\gamma$ on $\meR_\la$
is an approximately inner and centrally free.
So $\de\gamma$ is
cocycle conjugate to
$\id_N\otimes
\gamma^{(0)}\otimes \alpha^{(0)}$
by Theorem \ref{thm:main} for $\meR_\la$.

(2)
Let $N=\meR_\la\rti_\be \T$ and $\ga=\bbe$.
Extend the action $\de$ to $N$, which is also denoted by $\de$.
Using the Takesaki duality \cite{Tak-dual},
we see that all the assumptions of (1) are fulfilled.
Then we get $\de\bbe\sim \id_{\meR_\la}\oti\ga^{(0)}\oti\al^{(0)}$.
Comparing the crossed products by $\bbe$ and $\ga^{(0)}$,
we obtain $\de\be\sim \id_{\meR_\la}\oti\wdh{\ga^{(0)}}\oti\al^{(0)}$.
\end{proof}

\begin{lem}\label{lem:conj}
Let $M\cong\meR_\infty$, $N=M\rti_{\si_T^\vph}\Z$ as before,
and $\alpha$ a
centrally free action of $\bhG$ on $M$.
Then the $\bhG\times\T$-action
$\theta_{-t}\oti \bal\theta_t$ on $N\oti N$
is cocycle conjugate to
$\id_N\otimes \widehat{\gamma^{(0)}}_t\otimes \alpha^{(0)}$.
\end{lem}
\begin{proof}
We can identify $(N\otimes N)\rtimes_{\theta_t\otimes \theta_{-t}} \T$ 
with $(M\otimes M)
\rtimes_{\sigma^\vph_T\otimes \sigma^\varphi_T }\Z$
\cite[Lemma 1 (b)]{Co-III1}.
Hence
$(N\otimes N)\rtimes_{\theta_t\otimes \theta_{-t}}\T$
is a factor of type III$_\lambda$.
By Proposition \ref{prop:ext}, $\bal$ is approximately inner
and centrally free, hence so is $\id\oti \bal$.
It is obvious that
$\theta_{-t}\oti \bal\theta_t$ is a centrally free action.
Then the previous lemma can be applied.
\end{proof}

\begin{prop}\label{prop:split}
Let $M$, $N$, $\alpha$, $\theta$ be as above.
Let $\beta_t$ be a
product type action of $\T$
on $\meR_0\cong\bigotimes_{i=1}^\infty M_2(\C)$ 
given by
$\beta_t=\bigotimes_{i=1}^\infty \Ad 
\left(\begin{array}{cc}
1&0 \\
0&e^{\sqrt{-1}t} \\
\end{array}\right)$
for $t\in\R$.
Then $\bal\theta$ is
cocycle conjugate to
$\id_{\meR_\lambda}\otimes \beta\otimes \bal\theta$.
\end{prop}

The proof of Proposition \ref{prop:split} will be presented
in the sequel subsections.
Here we prove Theorem \ref{thm:class1}
assuming Proposition \ref{prop:split}.

\vspace{10pt}
\noindent
$\bullet$
\textit{Proof of Theorem \ref{thm:class1}.}

Note that $\beta$ is a minimal action of $\T$,
hence is dual, and conjugate to $\widehat{\gamma^{(0)}}$. 
Since $\theta_{-t}\otimes \theta_t$ (resp. $\theta_t$) is
cocycle
conjugate to $\id_{\meR_\lambda}\otimes \beta_t$ (resp.  
$\theta_t\otimes \beta_t\otimes \id_{\meR_\lambda}$) by the theory of 
Connes \cite[Lemma 5]{Co-III1} and Haagerup \cite{Ha-III1}, 
we have
\begin{align*}
\bal\theta_t&\sim
\id_{\meR_\lambda}\otimes \beta_t \otimes \bal\theta_t 
\mbox{\hspace{10pt}(by Proposition \ref{prop:split})}\\
&\sim\theta_t\otimes \theta_{-t}\otimes \bal\theta_t \\
&\sim\theta_t\otimes \beta_t\otimes \id_{\meR_\lambda}\otimes
 \alpha^{(0)} \mbox{\hspace{10pt}(by Lemma \ref{lem:conj})}
\\
&\sim\theta_t\otimes \alpha^{(0)}.
\end{align*}
Hence $\bal\theta_t$ is cocycle conjugate to $\theta_t\otimes
\alpha^{(0)}$.
Taking the crossed product by $\th$,
we see that $\al$ is cocycle conjugate to $\id_{\meR_\infty}\oti\al^{(0)}$.
\hfill$\Box$
\\

Therefore the proof of Theorem \ref{thm:main} has been reduced to
proving Proposition \ref{prop:split}.
We will
show that
$\bal\theta\sim
\id_{\meR_\lambda}\otimes \bal\theta$ in Corollary \ref{cor:lam-split},
and
$\bal\theta\sim \beta\otimes \bal\theta$
in Theorem \ref{thm:modelsplit},
and complete the proof of Proposition \ref{prop:split}.

\subsection{$\lambda$-stability}\label{sec:lam-stab}
As an analogue of the property $L'_a$ in \cite{Ar-lambda},
we introduce the following notion.

\begin{defn}\label{def:L'}
Let $\bhG$ be a discrete Kac algebra, $P$ a factor, and
$\alpha$ a cocycle action of $\bhG$ on $P$.
For $0<\lambda < 1$, 
set $a=\frac{\lambda}{1+\lambda}$.
We say that $(P,\alpha)$
satisfies the \emph{property} $L'_a$ if we have the following:

For any
$\varepsilon>0$,
any finite sets $\mF\Subset  \IG$ and
$\Psi_\pi \Subset (P\otimes B(H_\pi))_{*}$ for $\pi\in\mF$,
there exists a partial isometry 
$u\in P$ such that
for $\psi \in \Psi_\pi$, $\pi\in \mF$,
\begin{align*}
&uu^*+u^*u=1,
\quad u^2=0;\\
&\|(u\otimes 1)\cdot\psi-\lambda
\psi\cdot(u\otimes 1)\|<\varepsilon;\\
&\|(u\otimes 1-\alpha_\pi(u))\cdot \psi\|<\varepsilon;\\
&\|\psi\cdot(u\otimes 1-\alpha_\pi(u))\|<\varepsilon.
\end{align*}
\end{defn}

Note that the property $L'_a$ is
stable under perturbations of a cocycle action.

\begin{lem}\label{lem:L'}
Let $\alpha$ be a
centrally free cocycle action
of $\bhG$ on $\meR_\infty$.
Then $(\meR_\infty,\alpha)$ has the property
$L'_a$, $a=\frac{\lambda}{1+\lambda}$,
for any $0<\lambda<1$.
\end{lem}
\begin{proof}
Since $M:=\meR_\infty$ is properly infinite,
we may and do assume that $\al$ is an action.
Take $\pi \in \IG$, and set $\pi':=d(\pi) \mathbf{1}\oplus \pi$,
a direct sum representation of $\bG$.
Consider an inclusion $\alpha_{\pi'}(M)\subset M\otimes B(H_{\pi'})$.
We can identify $M\otimes B(H_{\pi'})$ with $M_2(M\otimes B(H_\pi))$
and
\[
\alpha_{\pi'}(M)=\left\{
\left(
\begin{array}{cc}
x\otimes 1_\pi&0_\pi \\
0_\pi&\alpha_\pi(x) \\
\end{array}\right)
\mid x\in M\right\}.
\]
Then $\alpha_{\pi'}(M)\subset M\otimes B(H_{\pi'})$
is an inclusion of injective factors of type III$_1$
with the minimal index $4d(\pi)^2$.
The minimal expectation $E^\pi$ is given by
\[
E^\pi\left(
\left(\begin{array}{cc}
a&b \\
c&d \\
\end{array}\right)
\right)
=\frac{1}{2}\alpha_{\pi'}
\left((\id\otimes \tr_\pi)
(a)+\Phi_\pi(d)
\right).
\]

For a fixed $0<\la<1$,
we construct the type III$_\la$ factor $N:=M\rti_{\si_T^\vph}\Z\subs \tM$,
$T=-2\pi/\log\la$ as before.
The implementing unitary is denoted by $U^\vph=\la^{\vph}(T)$.
Set
$\gamma:=\sigma^{\varphi\circ\Ph_{\pi'}^\al}_T$,
where $\Ph_{\pi'}^\al=\alpha^{-1}_{\pi'}\circ E^\pi$.
Then $\gamma$ globally preserves the inclusion
$\alpha_{\pi'}(M)\subset M\otimes B(H_{\pi'})$.
\\

\noindent\textbf{Claim 1.}
We show that the inclusion
$\alpha_{\pi'}(M)\rti_\ga\Z\subset (M\otimes B(H_{\pi'}))\rti_\ga\Z$
is isomorphic to $\bal_{\pi'}(N)\subs N\oti B(H_{\pi'})$.

We identify $(M\otimes B(H_{\pi'}))\rti_\ga\Z$ with
$(M\otimes B(H_{\pi'}))\rti_{\si_T^{\vph\oti\tr_{\pi'}}}\Z$
in the core algebra $Q$ of $M\otimes B(H_{\pi'})$.
Then
\begin{align*}
\alpha_{\pi'}(M)\rti_\ga\Z
&=
\alpha_{\pi'}(M)\vee\{\la^{\vph\circ\Ph_{\pi'}^\al}(T)\}''
\\
&=
\alpha_{\pi'}(M)\vee
\{[D\vph\circ\Ph_{\pi'}^\al:D\vph\oti\tr_{\pi'}]_T
\la^{\vph\oti\tr_{\pi'}}(T)\}''.
\end{align*}
The canonical isomorphism $\Psi\col Q\ra \tM\oti B(H_{\pi'})$
satisfies $\Ps|_{M\oti B(H_{\pi'})}=\id$ and
$\Ps(\la^{\vph\oti\tr_{\pi'}}(T))=\la^\vph(T)\oti1$.
Hence
\[\Ps(\la^{\vph\circ\Ph_{\pi'}^\al}(T))
=[D\vph\circ\Ph_{\pi'}^\al:D\vph\oti\tr_{\pi'}]_T
(\la^{\vph}(T)\oti1)=\bal_{\pi'}(\la^\vph(T)).
\]
Then we have
$\Ps(\alpha_{\pi'}(M)\rti_\ga\Z)=\bal_{\pi'}(N)$
and $\Ps((M\otimes B(H_{\pi'}))\rti_\ga\Z)=N\oti B(H_{\pi'})$.
\\

\noindent\textbf{Claim 2.}
We show that the inclusion $\bal_{\pi'}(N)\subs N\oti B(H_{\pi'})$
is relatively $\la$-stable.

Since $\bal$ is approximately inner and centrally free on
$N$ by Proposition \ref{prop:ext},
$\bal$ is cocycle conjugate to
$\id_{\meR_\lambda}\otimes \bal$
by Theorem \ref{thm:main} for type III$_\la$ case.
Hence the inclusion $\bal_{\pi'}(N)\subset N\otimes B(H_{\pi'})$
is relatively $\lambda$-stable
in the sense that
$\bal_{\pi'}(N)\subset N\otimes B(H_{\pi'})
\cong \meR_\lambda\oti \bal_{\pi'}(N)
\subset (\meR_\lambda\oti N)\otimes B(H_{\pi'})$.
\\

\noindent\textbf{Claim 3.}
We show that $\ga$ is an approximately inner automorphism
on the subfactor $\al_{\pi'}(M)\subs M\oti B(H_{\pi'})$.

By Corollary \ref{cor: modular-app}, 
we can choose $\{w_n\}_n\subset U(M)$
such that $\displaystyle\sigma_T^\varphi=\lim_{n\rightarrow \infty}\Ad w_n$ 
and 
$\displaystyle
[D\varphi\circ\Phi_\pi:D\varphi\otimes \tr_\pi ]_T
=\lim_{n\to\infty}
\alpha_\pi(w_n)(w_n^*\otimes 1)$ for all $\pi\in \IG$.
Since 
$2\varphi\circ \alpha^{-1}_{\pi'}\circ E^\pi $ is nothing but 
a balanced functional
$\varphi\otimes \tr_\pi\oplus\varphi\circ \Phi_\pi$, 
\[
\ga=
\Ad \left(
\begin{array}{cc}
1_\pi&0_\pi \\
0_\pi& [D\varphi\circ \Phi_\pi: D\varphi\otimes \tr_\pi]_T
\end{array}
\right)\circ (\sigma^\varphi_T\otimes \id_\pi).
\]
Thus
$\displaystyle\gamma=\lim_{n\rightarrow\infty }
\Ad\alpha_{\pi'}(w_n)$,
and $\gamma$ is approximately inner in a subfactor sense.
\\

By the previous three claims,
we can show that
the inclusion
$\alpha_{\pi'}(M)\subset M\otimes B(H_{\pi'})$
is relatively $\lambda$-stable.
Indeed, the proof is similar to that of 
\cite[Corollary II.3]{Co-III1}.
(Also see \cite[Theorem 3.6]{M-III1}.)
Hence for any $\varepsilon>0$ and
any $\{\psi_i\}_{i=1}^n\subs (M\otimes B(H_{\pi'}))_*$,
there exists $u\in M$ such that 
$u^2=0$, $uu^*+u^*u=1$ and
\[
\|\alpha_{\pi'}(u)\cdot \psi_i-
\lambda \psi_i\cdot\alpha_{\pi'}(u)\|<\varepsilon,
\quad \mbox{for all }1\leq i\leq n.
\]
For $\psi\in (M\otimes B(H_\pi))_*$, define $\psi_{ij}\in 
(M\otimes B(H_{\pi'}))_*$ by 
$\psi_{ij}(a)=\psi(a_{ij})$ via identification of
$M\otimes B(H_{\pi'})$ with $M_2(M\otimes B(H_\pi))$ and
$\alpha_{\pi'}(x)=\mathrm{diag}(x\otimes 1_\pi,\alpha_\pi(x))$
for $x\in M$.
Assume we have chosen $u$ so that
\[
\|\alpha_{\pi'}(u)\cdot \psi_{ij}-
\lambda \psi_{ij}\cdot\alpha_{\pi'}(u)\|<\varepsilon
\quad\mbox{for all }i,j=1,2.
\]
Then we obtain the following four inequalities.
\begin{align*}
&\|(u\otimes 1_\pi)\cdot \psi-
\lambda \psi\cdot (u\otimes 1)\|<\varepsilon,
\quad
\|(u\otimes 1)\cdot \psi-
\lambda \psi\cdot \alpha_{\pi}(u)\|<\varepsilon,
\\
&\|\alpha_{\pi}(u)\cdot \psi-
\lambda \psi\cdot (u\otimes 1)\|<\varepsilon,
\quad
\|\alpha_{\pi}(u)\cdot \psi-
\lambda \psi\cdot \alpha_{\pi}(u)\|<\varepsilon.
\end{align*}
It is easy to deduce that $u$ satisfies the condition in Definition 
\ref{def:L'} for $\psi$.
So far, we have considered a single element $\pi\in \IG$.
For a finite subset 
$\mF\Subset\IG $, define $\Pi:=\bigoplus_{\pi\in \mF}\pi' $, 
and consider the similarly defined inclusion
$\alpha_\Pi(M)\subset M\otimes B(H_\Pi)$.
Then the same argument is applicable.
\end{proof}

\begin{lem}\label{lem:perturb}
Let $P$ be a properly infinite factor,
$H$ a finite dimensional Hilbert space,
$\alpha \in \Mor_0(P,P\otimes B(H))$
and $\Phi\Subs(P\otimes B(H))_*$
a finite set of faithful states.
Let $0<\varepsilon<1$ and $0<\lambda\leq 1$.
Assume that
there exists $u\in P$ such that $uu^*+u^*u=1$, $u^2=0$
and for all $\varphi\in \Phi$,
\begin{align*}
&\|(u\otimes 1)\cdot\varphi-\lambda \varphi \cdot
(u\otimes 1)\|\leq \lambda\varepsilon,
\quad
\|\varphi\cdot (u\otimes 1) -\lambda^{-1} (u\otimes 1)\cdot \varphi \|
\leq 
\lambda\varepsilon;
\\
&\|(u\otimes 1-\alpha(u))\cdot\varphi\|\leq \lambda\varepsilon,
\quad
\|\varphi\cdot (u\otimes 1-\alpha(u))\|\leq \lambda\varepsilon.
\end{align*}
Then there exists a unitary $v\in P\otimes B(H)$ such that 
$\Ad v\circ \alpha =\id$ on the type I$_2$ subfactor $\{u\}''$ and
$\|v -1\|_{\varphi}^\#< 12 \sqrt[4]{\varepsilon}$
for all $\varphi\in \Phi$. 
\end{lem}

\begin{proof}
In the following, we frequently use the inequalities
$\|x\|_\varphi^2\leq \|x\|\|x\cdot \varphi\|$,
$\|x\cdot\varphi\|\leq \sqrt{\|\varphi\|}\|x\|_\varphi$. 
First we show $uu^*\otimes 1$ and $\alpha(uu^*)$ are close as follows:
\begin{align*}
&\hspace{16pt}\|(uu^*\otimes 1-\alpha(uu^*))\cdot \varphi\|\\
&\leq
\|(uu^*\otimes 1)\cdot\varphi-
\lambda^{-1} \alpha(u)\cdot\varphi\cdot(u\otimes 1))\|\\
&\quad+ 
\|\lambda^{-1} \alpha(u)\cdot\varphi\cdot(u^*\otimes 1)
-\alpha(uu^*)\cdot \varphi\|\\
&=
 \|(uu^*\otimes 1)\varphi-\lambda^{-1}(u\otimes 1)\cdot\varphi \cdot
(u^*\otimes 1) 
\|  \\ &\quad + \|
\lambda^{-1}(u\otimes 1)\cdot\varphi\cdot( u^*\otimes 1) -
\lambda^{-1} \alpha(u)\cdot\varphi\cdot(u^*\otimes 1))\|
\\
&\quad +
\|\lambda^{-1} \varphi\cdot (u^*\otimes 1)
-(u^*\otimes 1)\cdot \varphi
\| + \|
(u^*\otimes 1)\varphi
-\alpha(u^*)\cdot \varphi\|\leq4\varepsilon.
\end{align*}
Since $\|x\|_\varphi^2\leq \|x\varphi\|\|x\|$, 
we have $\|uu^*\otimes1- \alpha(uu^*)\|_{\varphi}^2\leq 
8\varepsilon $. In the same way, we have
$\|u^*u\otimes1- \alpha(u^*u)\|_{\varphi}^2\leq 
8\varepsilon$. Hence we have 
$\|uu^*\otimes1- \alpha(uu^*)\|_{\varphi}^\#\leq 
2\sqrt{2\varepsilon}$ and
$\|u^*u\otimes1- \alpha(u^*u)\|_{\varphi}^\#\leq 
2\sqrt{2\varepsilon}$.

By \cite[Lemma 1.1.4]{Con-auto} and \cite[Lemma 8.1.1]{Ocn-act}, 
there exists a partial isometry 
$w \in P\otimes B(H)$
with $ww^*=uu^*\otimes 1$, $w^*w=\alpha(uu^*)$,
$\|w-uu^*\otimes 1\|_{\varphi}^\#\leq 
7\|uu^*\otimes 1-\alpha(uu^*)\|_{\varphi}$ for $\varphi\in \Phi$.
Hence we have
$\|w-uu^*\otimes 1\|_{\varphi}^\#
\leq 14\sqrt{2\varepsilon}$.

Set $v:=(uu^*\otimes 1)w\alpha(uu^*)+(u^*\oti1)w\alpha(u)$.
It is standard to see $\Ad v\circ \alpha(x)=x\otimes 1$ for $x\in\{u\}''$.
We estimate $\|(v-1)\cdot \varphi\|$ and
$\|\varphi\cdot(v-1)\|$.
Since $\|x\|_\varphi \leq \sqrt{2}\|x\|_\varphi^\#$, and 
$\|x\varphi\|\leq \sqrt{\|\varphi\|}\|x\|_\varphi$,
we have
\[\|(w-uu^*\otimes 1)\cdot\varphi\|\leq 
\|w-uu^*\otimes 1\|_{\varphi}\leq \sqrt{2}
\|w-uu^*\otimes 1\|_{\varphi}^\#\leq 28\sqrt{\varepsilon}.
\]

Since $\|[uu^*\otimes 1,\varphi]\|\leq 2\varepsilon$, we get
\begin{align*}
&\hspace{14pt}
\left\|\left((uu^*\otimes 1)w
\alpha (uu^*)-uu^*\otimes 1\right)\cdot\varphi
\right\|\\
&\leq
\left\|\left(w\alpha (uu^*)-uu^*\otimes 1\right)
\cdot\varphi
\right\| \\ 
&\leq
\left\|\left(w\alpha (uu^*)-w(uu^*\otimes 1)\right)
\cdot\varphi
\right\| 
+
\left\|\left(w(uu^*\otimes 1)-uu^*\otimes 1\right)\cdot
\varphi
\right\| \\ 
&\leq 4\varepsilon+ 
\|(w-uu^*\otimes 1)[\varphi, uu^*\otimes 1]
\|+
\|(w-uu^*\otimes 1)\cdot\varphi\cdot(uu^*\otimes 1)\|
 \\
&\leq 4\varepsilon+ 4\varepsilon +
28\sqrt{\varepsilon}
\leq 36\sqrt{\varepsilon}
\end{align*}
and
\begin{align*}
&\hspace{14pt}\left\|\left((u^*\otimes 1)w
\alpha (u)-u^*u\otimes 1\right)\cdot\varphi
\right\|\\ 
&\leq
\left\|\left(w\alpha (u)-u\otimes 1\right)
\cdot\varphi
\right\| \\ 
&\leq
\left\|\left(w\alpha (u)-w(u\otimes 1)\right)
\cdot\varphi
\right\| 
+
\left\|\left(w(u\otimes 1)-u\otimes 1\right)
\cdot\varphi
\right\| \\ 
&\leq \varepsilon+ 
\|(w-uu^*\otimes 1)(u\varphi-\lambda\varphi u)
\|+
\|(w-uu^*\otimes 1)\cdot\lambda \varphi\cdot(u\otimes 1)\|
 \\
&\leq \varepsilon+ 2\varepsilon +
28\sqrt{\varepsilon}
\leq 31\sqrt{\varepsilon}.
\end{align*}

Hence $\|(v-1)\cdot \varphi\|\leq 36\sqrt{\varepsilon}+31\sqrt{\varepsilon}=
67\sqrt{\varepsilon}$, and $\|v-1\|_\varphi^2\leq 134\sqrt{\varepsilon}$ holds.

Next we estimate $\|\varphi\cdot(v-1)\|$ as follows:
\begin{align*}
&\hspace{14pt}\left\|\varphi\cdot\left((uu^*\otimes 1)w
\alpha (uu^*)-uu^*\otimes 1\right)
\right\|
\\ 
&\leq
 \left\|\varphi\cdot(uu^*\otimes 1)\left(w
\alpha (uu^*)-uu^*\otimes 1\right)
\right\| \\
&\leq
 \left\|\left[\varphi,uu^*\otimes 1\right]
\left(w
\alpha (uu^*)-uu^*\otimes 1\right)
\right\| \\
&\quad+ \left\|(uu^*\otimes 1)\cdot\varphi\cdot\left(w
\alpha (uu^*)-uu^*\otimes 1\right)
\right\| 
\\
&\leq
4\varepsilon
+ \left\|(uu^*\otimes 1)\cdot\varphi\cdot\left(w
\alpha (uu^*)-uu^*\otimes 1\right)
\right\| 
\\
&\leq
4\varepsilon
+ \left\|\varphi\cdot\left(w-\alpha(uu^*)\right)\alpha(uu^*)
\right\| +
\|\varphi\cdot\left((uu^*\otimes 1)\alpha (uu^*)-
uu^*\otimes 1\right)\|
\\
&\leq
4\varepsilon
+ 28\sqrt{\varepsilon}+
\|\varphi\cdot(uu^*\otimes 1)\left(\alpha (uu^*)-uu^*\otimes 1\right)
\|
\\
&\leq 32\sqrt{\varepsilon}+4\varepsilon+
\|\varphi\cdot\left(\alpha (uu^*)-uu^*\otimes 1\right)
\| \\
&\leq 32\sqrt{\varepsilon}+4\varepsilon+4\varepsilon
\leq40\sqrt{\varepsilon}
\end{align*}
and
\begin{align*}
&\hspace{14pt}
\left\|\varphi\cdot\left((u^*\otimes 1)w
\alpha (u)-u^*u\otimes 1\right)
\right\|
\\ 
&\leq
 \left\|(\varphi \cdot(u^*\otimes 1)-\lambda (u^*\otimes 1)\cdot\varphi)
\cdot\left(w
	   \alpha (u)-u\otimes 1\right)
\right\| \\
&\quad+
 \left\|\lambda u^*\varphi\cdot\left(w
	   \alpha (u)-u\otimes 1\right)
\right\|\\
&\leq
2\varepsilon
+
 \left\|\varphi\cdot\left(w
	   \alpha (u)-u\otimes 1\right)
\right\|\\
&\leq
2\varepsilon
+
 \left\|\varphi\cdot(w-uu^*\otimes 1)\alpha(u)\|+\|
	 \varphi\cdot\left((uu^*\otimes 1)  \alpha (u)-u\otimes 1\right)
\right\|\\
&\leq
2\varepsilon
+
28\sqrt{\varepsilon}
+ \|
 \varphi\cdot(uu^*\otimes 1)\cdot\left(  \alpha (u)-u\otimes 1\right)
\| \\
&\leq
30\sqrt{\varepsilon}+4\varepsilon
+ \|
 \varphi\cdot\left(  \alpha (u)-u\otimes 1\right)
\|
\leq 35\sqrt{\varepsilon}.
\end{align*}
Hence $\|\varphi\cdot (v-1)\|\leq 75\sqrt{\varepsilon}$, and 
$\|v^*-1\|_\varphi^2\leq 150\sqrt{\varepsilon}$ holds.
This implies that $\|v-1\|_\varphi^{\#2}=\frac{1}{2}(\|v-1\|_\varphi^2+
\|v^*-1\|_\varphi^2)\leq 142\sqrt{\varepsilon}$,
and
$\|v-1\|_\varphi^\#\leq 12\sqrt[4]{\varepsilon}$.
\end{proof}

\begin{thm}\label{thm:lam-split}
Let $\alpha$ be a
centrally free action of $\bhG$ on $\meR_\infty$.
Then $\alpha$ is cocycle
conjugate to $\id_{\meR_\lambda}\otimes \alpha$ for all $0< \lambda <1$.
\end{thm}
\begin{proof}
Set $M:=\meR_\infty$,
$\varepsilon_n:=16^{-n}$.
Let $\{\mF_n\}_{n=1}^\infty$ be an increasing sequence of
finite sets of $\IG$ with $\bigcup_{n=1}^\infty\mF_n=\IG$. 
Let $\{\psi_n\}_{n=1}^\infty\subset (M_*)_+$
be a countable dense subset such that
$\psi_1$ is a faithful state.
For each $k\in\N$, we will
construct a mutually commuting sequence of $2\times 2$-matrix units 
$\{e_{ij}(k)\}_{i,j=1}^2$, and unitaries $v^k, 
\bar{v}^k\in M\otimes \lhG $
with the following five conditions:
\begin{align*}
&\bar{v}^n=v^n v^{n-1}\cdots v^1;
\\
&\Ad \bar{v}_\pi^n\circ \alpha_\pi(e_{ij}(k))=e_{ij}(k)\otimes 1_\pi, 
\quad i,j=1,2,\ 1\leq k\leq n,\ \pi\in\mF_n;
\\
&\|v_\pi^n -1\|_{\psi_1\otimes \tr_\pi}^\#<12\sqrt[4]{\varepsilon_n}, 
\quad\pi\in \mF_n;
\\
&\|v_\pi^n -1\|_{(\psi_1\otimes \tr_\pi)\circ 
\Ad \bar{v}^{n-1*}_\pi}^\#<12\sqrt[4]{\varepsilon_n}, 
\quad\pi\in \mF_n;
\\
&\|\psi_k\cdot e_{ij}(n)-\lambda^{i-j}e_{ij}(n)\cdot\psi_k\|< 
2\varepsilon_n,
\quad1\leq k\leq n.
\end{align*}

Since $(M,\alpha)$ has the property $L'_a$ 
for any $0<a< 1/2$ by Lemma \ref{lem:L'}, we can choose 
$u\in M$ such that
\begin{align*}
&uu^*+u^*u=1,\quad
u^2=0;
\\
&\|(u\otimes 1-\alpha_\pi(u))\cdot (\psi_1\otimes \tr_\pi)\|<\lambda
\varepsilon_1,
\quad\pi \in \mF_1;
\\
&\|(\psi_1\otimes \tr_\pi)\cdot
(u\otimes 1-\alpha_\pi(u))\|<\lambda
\varepsilon_1,
\quad\pi \in \mF_1;
\\
&\|u\cdot\psi_1-\lambda \psi_1\cdot u\|<\lambda^2\varepsilon_1.
\end{align*}
Then by Lemma
\ref{lem:perturb}, there exists a unitary $v_\pi^1$ such that
$\|v_\pi^1 -1\|^\#_{\psi_1\otimes \tr_\pi}< 12\sqrt[4]{\varepsilon_1}$,
$\pi\in \mF_1$, and
$\Ad v_\pi^1\circ\alpha_\pi (u)=u\otimes 1$.
We define $v_\rho^1$, $\rho\not\in \mF_1$, 
in a similar way to the proof of Lemma \ref{lem:perturb}.
Set $\{e_{11}(1), e_{12}(1),
e_{21}(1), e_{22}(1)\}:=\{uu^*, u, u^*,u^*u\}$.
Note that 
$\|[e_{ii}(1),\psi_1]\|<2\varepsilon_1$,
so the first step is complete.

Suppose we have done up to the $n$-th step.
Set $E_n:=\bigvee_{k=1}^n(\{e_{ij}(k)\}_{i,j=1}^2)''$,
$\alpha^{n+1}:=\Ad \bar{v}^n \circ\alpha$, 
and $M_{n+1}:=E_n'\cap M$. 
Then $\alpha^{n+1}$
is a centrally free cocycle 
action on $M_{n+1}\cong\meR_\infty$.
Hence $(M_{n+1}, \alpha^{n+1})$ 
has the property $L'_a$ by Lemma \ref{lem:L'}.
Let $\{w_\el\}_{\el=1}^{4^n}$ be a basis for $E_n^*$
with $\|w_\el\|\leq1$, and 
decompose $\psi_k=\sum_{\el=1}^{4^n} w_\el\otimes \psi_{k \el}$.
Take $u\in M_{n+1}$ satisfying
$uu^*+u^*u=1$, $u^2=0$ and the following conditions:
for any $\pi\in\mF_{n+1}$,
\begin{align*}
&\|(u\otimes 1-\alpha_\pi^{n+1}(u))
\cdot (\psi_1\otimes \tr_\pi)\|<\lambda
\varepsilon_{n+1};
\\
&\|(\psi_1\otimes \tr_\pi)
\cdot(u\otimes 1-\alpha_\pi^{n+1}(u))\|<\lambda
\varepsilon_{n+1};
\\
&\|(u\otimes 1-\alpha_\pi^{n+1}(u))
\cdot
\big{(}(\psi_1\otimes \tr_\pi)\circ \Ad 
\bar{v}_\pi^{n*}\big{)}\|<\lambda
\varepsilon_{n+1};
\\
&\|
\big{(}(\psi_1\otimes \tr_\pi)\circ \Ad \bar{v}_\pi^{n*}\big{)}
\cdot(u\otimes 1-\alpha_\pi^{n+1}(u))
\|<\lambda
\varepsilon_{n+1};
\\
&\|(u\otimes 1)\cdot
\big{(}
(\psi_{1}\otimes \tr_\pi) \circ \Ad \bar{v}_\pi^{n*}
\big{)}
-
\lambda
\big{(}
(\psi_{1}\otimes \tr_\pi)\circ \Ad \bar{v}_\pi^{n*}
\big{)}\cdot (u\oti1)\|
<
\lambda^2 \varepsilon_{n+1};
\\
&\|u\cdot\psi_{k\el}-\lambda \psi_{k\el}\cdot u\|<
4^{-n}\lambda^2 \varepsilon_{n+1},
\quad1\leq k\leq n+1,\ 1\leq \el\leq 4^n.
\end{align*}
Here we have regarded
$\psi_1$ and $(\psi_1\otimes \tr_\pi)\circ\Ad\bar{v}^{n*}_\pi$
as states on $M_{n+1}$ and $M_{n+1}\otimes B(H_\pi)$, respectively.
The last inequality yields 
$\|u\cdot \psi_k-\lambda 
\psi_k\cdot u\|\leq \lambda^2 \varepsilon_{n+1}$,
and in particular,
$\|(u\otimes 1)\cdot (\psi_1\otimes \tr)
-\lambda 
(\psi_1\otimes \tr)\cdot (u\otimes 1)\|\leq \lambda^2 \varepsilon_{n+1}$.
By Lemma \ref{lem:perturb},
there exists a unitary $v^{n+1}_\pi\in M_{n+1}\otimes B(H_\pi)$
for $\pi\in\mF_{n+1}$
such that
\begin{align*}
&\Ad v^{n+1}_{\pi}\circ\alpha_\pi^n(u)=u\otimes 1,
\quad\pi\in\mF_{n+1};
\\
&\|v_\pi^{n+1}-1\|_{\psi_1\otimes \tr_\pi}^\#<12\sqrt[4]{\varepsilon_{n+1}},
\quad\pi\in \mF_{n+1};
\\
&\|v_\pi^{n+1}-1\|_{(\psi_1\otimes \tr_\pi)\circ\Ad \bar{v}_\pi^{n*}}^\#
<12\sqrt[4]{\varepsilon_{n+1}},
\quad\pi\in \mF_{n+1}.
\end{align*}
Set $\{e_{11}(n+1), e_{12}(n+1), e_{21}(n+1),e_{22}(n+1)\}
:=\{uu^*,u,u^*,u^*u\}$.
Then $\|\psi_k\cdot e_{ij}(n+1)-\lambda^{i-j}
e_{ij}(n+1)\cdot\psi_k\|<2\varepsilon_{n+1}$ holds for $1\leq k\leq n+1$.
Define $v^{n+1}$ by extending $v_\pi^{n+1}$, $\pi\in\mF_{n+1}$,
as before. 
Thus we have finished the $(n+1)$-st step, and this completes our induction.

Define $E_\infty:=\bigvee_{k=1}^\infty \{e_{ij}(k)\}_{i,j=1,2}''$.
Since
$\sum_{k=1}^\infty\|\psi_n\cdot e_{ij}(k)-\lambda^{i-j}e_{ij}(k)\cdot\psi_n
\|<\infty $ for all $n\in\N$,
$E_\infty$ is an injective factor of type III$_\lambda$,
and we have the
factorization $M=E_\infty\vee E_\infty'\cap M\cong E_\infty\otimes 
E_\infty'\cap M$ by \cite[Theorem 1.3]{Ar-lambda}. (Also see 
\cite[Lemma XVIII.4.5]{Tak-book}.)
Next we show the convergence of $\{\bar{v}_\pi^n\}_{n=1}^\infty$.
If $\pi\in \mF_n$, we have
\begin{align*}
\|\bar{v}^{n+1}_\pi-\bar{v}_\pi^n\|_{\psi_1\otimes \tr_\pi}
&=\|(v_\pi^{n+1}-1)\bar{v}_\pi^n\|_{\psi\otimes \tr_\pi}
=\|v_\pi^{n+1}-1\|_{(\psi\otimes \tr_\pi)\circ \Ad \bar{v}_\pi^{n*}}\\
&<12\sqrt{2}\sqrt[4]{\varepsilon_{n+1}}
\end{align*}
and
\begin{align*}
\|\left(\bar{v}^{n+1}_\pi-\bar{v}_\pi^n\right)^{*}\|_{\psi_1\otimes \tr_\pi}
&=\|(v_\pi^{n+1}-1)^*\|_{\psi\otimes \tr_\pi}\\
&<12\sqrt{2}\sqrt[4]{\varepsilon_{n+1}}.
\end{align*}
Hence for each $\pi\in\IG$,
$\{\bar{v}_\pi^n\}_{n=1}^\infty$ is a Cauchy sequence
in the strong* topology,
and set $\displaystyle\bar{v}_\pi:=\lim_{n\to\infty} \bar{v}_\pi^n$.
Set $\bar{v}=(\bar{v}_\pi)_\pi\in M\oti \lhG$.
By the choice of $v_\pi^n$,
$\al':=\Ad \bar{v}\circ\alpha$ acts trivially on
$E_\infty$.
Hence $\al'$ is a cocycle action on
$E_\infty'\cap M$ with a 2-cocycle
$u=\bar{v}^{(12)}(\alpha\otimes \id)(\bar{v})(\id\otimes \Delta)(\bar{v}^*)$.
Since
$E_\infty'\cap M$ is of type III,
$u$ is a coboundary by 
Lemma \ref{lem: prop-inf-coboundary}.
Hence $\alpha$ is cocycle conjugate to 
$\id_{E_\infty}\otimes \beta$
for some action $\beta$ of $\bhG$ on $E'_\infty\cap M$.
Since $E_\infty\otimes E_\infty\cong E_\infty
\cong \meR_\lambda$,
$\al\sim \id_{E_\infty}\oti\be
\approx \id_{E_\infty}\oti\id_{E_\infty}\oti\be
\sim\id_{\meR_\lambda}\otimes \alpha$.
\end{proof}

\begin{cor}\label{cor:lam-split}
Let $M\cong \meR_\infty$, $N=M\rti_{\si_T^\vph}\Z$,
$T=-2\pi/\log\la$
and $\theta$ be as before.
Let $\al$ be a centrally free action of $\bhG$ on $\meR_\infty$.
Then the $\bhG\times\T$-action
$\bal\theta$ is cocycle conjugate to
$\id_{\meR_\lambda}\otimes \bal\theta$.
\end{cor}
\begin{proof}
This is immediate from
Lemma \ref{lem: can-cocycle-stable} and
Theorem \ref{thm:lam-split}
when we consider the state of the form $\vph_\la\oti \vph$
on $\meR_\la\oti M$, where
$\varphi_\lambda$ is a periodic
state on $\meR_\lambda$.
\end{proof}

\subsection{Model action splitting}\label{sec:1-model}

\begin{lem}\label{lem:modelpiece}
Let $\al$ be a centrally free cocycle action of $\bhG$ on $M\cong\meR_\infty$.
Then there exists a centralizing sequence
of partial isometries $\{u_n\}_n\subset N$
with $u_nu_n^*+u_n^*u_n=1$, $u_n^2=0$, $\theta_t(u_n)=e^{\sqrt{-1}t}u_n$
for all $t\in\R$
and $\displaystyle\lim_{n\to\infty}\bal(u_n)-u_n\otimes 1=0$
in the $\sigma$-strong* topology.
\end{lem}
\begin{proof}
Since $M$ is properly infinite,
$\alpha$ is cocycle conjugate to an action $\al'$.
Then $\al'\sim\id_{\meR_\lambda}\otimes \alpha'$
by Theorem \ref{thm:lam-split} and
$\meR_\lambda \cong \meR_0\otimes \meR_\lambda$, 
$\alpha'$ is cocycle conjugate to
$\id_{\meR_0}\otimes \alpha'$ via an isomorphism
$\meR_0\otimes M\cong M$.
By Lemma \ref{lem: can-cocycle-stable},
it suffices to show the statements for
$\id_{\meR_0}\otimes \alpha$
and $N=(\meR_0\oti M)\rti_{\si_T^{\tr\otimes \varphi}}\Z$
assuming that $\al$ is an action.
We denote by $U$ the implementing unitary.

Let
$\{v_n\}_{n=1}^\infty
\subset  \meR_0\otimes \C1\subset \meR_0\otimes M$
be a centralizing sequence of partial
isometries with
$v_nv_n^*+v_n^*v_n=1$, $v_n^2=0$,
and $(\id_{\meR_0}\otimes \alpha)(v_n)=v_n\otimes 1$.
Let
$\{w_n\}_{n=1}^\infty\subset\C1\otimes  M$ as 
in Corollary \ref{cor: modular-app}.
Set $u_n:=U^*w_n v_n^*$ for each $n\in\N$.
Since $[w_n,v_n]=0$ and
$Uv_nU^*=\sigma^{\tr\oti\varphi}_T(v_n)=v_n$, we have $u_nu_n^*=v_nv_n^*$,
$u_n^*u_n=v_n^*v_n\in M$,
$u_nu_n^*+u_n^*u_n=1$ and
$u^2_n=0$.
Since $(U^* w_n)_n$ is centralizing,
$\{u_n\}_{n=1}^\infty$ is a centralizing sequence in $N$,
and $\theta_t(u_n)=e^{\sqrt{-1}t}u_n$ for all $t\in\T$.
Take a faithful normal state $\psi$ on $N\otimes B(H_\pi)$.
Then we have
\begin{align*}
&\hspace{14pt}
\left\|\psi \cdot \left(
(\id\otimes \bal_\pi)(u_n)-u_n\otimes 1\right)
\right\| \\
&=
\left\|\psi\cdot
(U^*\oti1)
\left(
[D\varphi\circ \Phi_\pi:D\phi\otimes \tr_\pi]_T^*
(\id\otimes\alpha_\pi)(w_n)(v_n^*\otimes 1)
-(w_n v_n^*\otimes 1)
\right)
\right\|
\\
&=
\left\|\psi\cdot
(U^*\oti1)
\left(
[D\varphi\circ \Phi_\pi:D\phi\otimes \tr_\pi]_T^*
-(w_n \otimes 1)(\id\otimes\alpha_\pi)(w_n^*)
\right)
\right\|
\\
&\to0
\end{align*}
as $n\to\infty$.
In a similar way, we get 
$\displaystyle\lim_{n\to\infty}
\left\|\left(
(\id\otimes \bal_\pi)(u_n)-u_n\otimes 1\right)\cdot \psi
\right\|=0$.
These implies that $\bal_\pi(u_n)-u_n\otimes 1$
converges to 0 $\sigma$-strongly*.
\end{proof}

\begin{thm}\label{thm:modelsplit}
Let $M$, $N$, $\theta$ be as before.
Let $\al$ be a centrally free action of $\bhG$ on $M$.
Let $\beta$ be the
infinite tensor product type action of $\T$ on $\meR_0$
given in Proposition \ref{prop:split}.
Then the $\bhG\times\T$-action
$\bal\th$ is cocycle conjugate to
$\beta\otimes \bal\th$.
\end{thm}
\begin{proof}
The proof is similar to that of Theorem \ref{thm:lam-split}.
Set $\varepsilon_n:=16^{-n}$.
Let $\{\psi_n\}_{n=1}^\infty\subset (M_*)_+$
be a countable dense subset such that $\psi_1$ is a faithful state.
For each $k\in\N$,
we will construct a mutually commuting sequence of $2\times 2$-matrix units
$\{e_{ij}(k)\}_{i,j=1}^2\subs N$,
and unitaries
$v^k, \bar{v}^k\in M\otimes B(H_\pi)$
with the following:
\begin{align*}
&\bar{v}^k=v^k v^{k-1}\cdots v^1;
\\
&\Ad \bar{v}_\pi^n\circ\bal_\pi(e_{ij}(k))=e_{ij}(k)\oti1
\quad\mbox{for all } i,j=1,2,\ 1\leq k\leq n,\ \pi\in\mF_n;
\\
&\|v_\pi^n -1\|_{\psi_1\otimes \tr_\pi}
<12\sqrt[4]{\varepsilon_n}
\quad\mbox{for all }\pi\in \mF_n;
\\
&\|v_\pi^n -1\|_{(\psi_1\otimes \tr_\pi)\circ \Ad \bar{v}^{n-1*}}
<12\sqrt[4]{\varepsilon_n}
\quad\mbox{for all }\pi\in \mF_n;
\\
&\theta_t(e_{ij}(k))
=\Ad
\left(
\begin{array}{cc}
1& 0\\
0&e^{\sqrt{-1}t}\\
\end{array}\right)(e_{ij}(k))
\quad\mbox{for all } t\in\R,\ i,j=1,2,\ k\in\N;
\\
&\|\psi_k\cdot e_{ij}(n)-e_{ij}(n)\cdot\psi_k\|
< 
2\varepsilon_n\quad\mbox{for all } i,j=1,2,\ 1\leq k\leq n.
\end{align*}

By Lemma \ref{lem:modelpiece},
there exists a partial isometry $u\in N$
such that $u^2=0$, $uu^*+u^*u=1$, $\theta_t(u)=e^{\sqrt{-1}t}u$,
$uu^*, u^*u\in M$ and
\begin{align*}
&\|(u\otimes 1-\bal_\pi(u))\cdot (\psi_1\otimes \tr_\pi)\|
<
\varepsilon_1\quad\mbox{for all } \pi \in \mF_1;
\\
&\|(\psi_1\otimes \tr_\pi)\cdot(u\otimes 1-\bal_\pi(u))\|<
\varepsilon_1
\quad\mbox{for all } \pi \in \mF_1;
\\
&\|u\cdot\psi_1-\psi_1\cdot u\|<\varepsilon_1.
\end{align*}

Since $uu^*\otimes 1_\pi,\bal_\pi(uu^*)=\al_\pi(uu^*)\in M\otimes B(H_\pi)$,
we can take $w$ from $M$ in the proof of Lemma \ref{lem:perturb}.
Then $v_\pi^1$ constructed in Lemma \ref{lem:perturb}
is in $M\otimes B(H_\pi)$,
and we have
\begin{align*}
&\|v_\pi^1 -1\|^\#_{\psi_1}< 12\sqrt[4]{\varepsilon_1}
\quad\mbox{for all }
\pi\in \mF_1;
\\
&\Ad v_\pi^1\circ\bal_\pi (u)=u\otimes 1
\quad\mbox{for all }\pi\in \mF_1.
\end{align*}
Set $\{e_{11}(1), e_{12}(1),
e_{21}(1), e_{22}(1)\}:=\{u^*u,u^*,u,uu^*\}$.
Define $v^1\in M\oti\lhG$ by extending $v_\pi^1$, $\pi\in\mF_1$, 
as before. 
Note that
$\|[e_{ij}(1),\psi_1]\|<2\varepsilon_1$ for $i,j=1,2$.
So the first step is complete.

Set $\bal^2:=\Ad v^1\circ\bal$, and $N_2:=\{u\}'\cap N$.
Take $w\in M$ an isometry with $ww^*=e_{11}(1)$.
Set $s_1:=e_{11}(1)w$ and $s_2:=e_{21}(1)w$.
Then $s_is^*_j=e_{ij}(1)$,
$\theta_t(s_1)=s_1$ and $\theta_t(s_2)=e^{-it}s_2$ hold.
Let $\rho(x):=\sum_{i=1}^2 s_ixs_i^*$.
Then $\rho$ is an isomorphism
between $N$ and $N_2$ which intertwines $\theta$.
Then $M_2:=N_2^\theta=\rho(M)$ is the injective factor of type III$_1$,
and
$\theta$ is the dual action for $\sigma^{\varphi'}_T$
where $\vph':=\vph\circ\rho^{-1}\in (M_2)_*$.
Since $\theta$ commutes with $\bal^2$
because of $v_\pi^1\in M\oti B(H_\pi)$,
$\bal^2$ preserves $M_2$.
Note that $v^1\bal(v^1)(\id\oti\De)((v^1)^*)$,
a 2-cocycle of $\bal^2$ is in $N_2$
and fixed by $\theta$, and it is indeed in $M_2$.
This means that $\bal|_{M_2}$ is a cocycle action.
Obviously we have
$Z(\widetilde{N})=Z(\widetilde{N_2})$.
Hence $\bal^2$ has trivial Connes-Takesaki module,
and $\bal^2$ is approximately inner.
By Lemma \ref{lem:canonical},
$\bal^2$ is the canonical extension of $\al^2:=\bal^2|_{M_2}$.
Since $\bal$ is centrally free,
$\bal^2$ is centrally free,
and $\al^2$ is centrally free on $M_2$
by Lemma \ref{lem:canonical}.

Then we can apply Lemma \ref{lem:modelpiece} to $M_2$,
$\alpha^2$, and $\theta$.
The rest of the proof is same as that of Theorem \ref{thm:lam-split}.
\end{proof}

\section{Appendix}\label{sec:appendix}
We discuss relations between the canonical extension of
endomorphisms and homomorphisms.
In this section, we do not assume the amenability of $\bhG$.

\subsection{Canonical extension of homomorphisms}
Let $M$ be a properly infinite factor
and
$H$ a finite dimensional Hilbert space with $\mathrm{dim}\, H=n$.
Let $\tM$ be the canonical core of $M$
\cite[Definition 2.5]{FT}.
We denote by $\Tr_H$ and $\tr_H$ the non-normalized and the normalized
traces on $B(H)$, respectively.
Then we can introduce an isomorphism between
the inclusions $M\subs \tM$ and $M\otimes B(H)\subs \tM\otimes B(H)$
as follows.
Fix isometries $\{v_i\}_{i=1}^n\subs M$ with orthogonal ranges
and $\sum_{i=1}^n v_i v_i^*=1$.
Define $\sigma \in \Mor(\tM\otimes B(H),\tM)$ by
\[
\sigma(x)=\sum_{i,j=1}^n v_ix_{ij}v_j^*.
\]
It is easy to see that $\sigma$
is an isomorphism with
$\sigma^{-1}(x)=\sum_{i,j=1}^n v_i^*xv_j\otimes e_{ij}$.
The map $\si$ derives the following bijection:
\[
\si_*\col \Mor_0(M,M\oti B(K))\ra \End_0(M),
\quad \al\mapsto \si\circ\al.
\]
We can check that $d(\al)=d(\si\circ\al)$
and the standard left inverse of $\rho:=\si\circ\al$
is given by $\ph_\rho=\Ph_\al\circ\si^{-1}$.
Hence
$\Phi^\alpha(x)=\phi_\rho\circ\sigma(x)
=\sum_{i,j=1}^n\ph_\rho(v_ix_{ij}v_j^*)$
holds.

Recall the topology on $\End_0(M)$ introduced in
\cite[Definition 2.1]{M-T-endo-pre}.
We also introduce a topology on $\Mor_0(M,M\oti B(H))$
similarly.

\begin{lem}\label{lem: homeo}
The map $\si_*$ is a homeomorphism.
\end{lem}
\begin{proof}
Take any $\vph\in M_*$.
Assume that $\al^\nu\to\al$
in $\Mor_0(M,M\oti B(H))$ as $\nu\to\infty$,
that is,
we have the norm convergence
$\vph\circ \Ph^{\al^\nu}\to\vph\circ\Ph^{\al}$
in $(M\oti B(H))_*$.
Write $\rho^\nu=\si_*(\al^\nu)$ and $\rho=\si_*(\al)$.
Using $\ph_{\rho^\nu}=\Ph^{\al^\nu}\circ \si^{-1}$
and $\ph_\rho=\Ph^{\al}\circ\si^{-1}$,
we have the norm convergence
$\vph\circ\ph_{\rho^\nu}\to\vph\circ\ph_\rho$,
that is, $\rho^\nu\to\rho$ as $\nu\to\infty$.
Hence $\si_*$ is continuous.
Similarly we can prove that $\si_*^{-1}$ is continuous.
\end{proof}

\begin{lem}\label{lem: connes-cocycle}
Let $\varphi$ be a faithful normal state on $M$.
Then one has
\[
\left[D\vph\circ\Phi^\alpha: D\varphi\otimes \Tr\right]_t=
\sum_{i,j=1}^n
v_i^*[D\varphi\circ \phi_{\si_*(\al)}: D\varphi]_t
\sigma^\varphi_t(v_j)\otimes e_{ij}
\quad\mbox{for all}\ t\in\R.
\]
\end{lem}
\begin{proof}
Set $\rho:=\si_*(\al)$ and a unitary
$u_t:=\sum_{i,j=1}^n v_i^*
[D\varphi\circ \phi_\rho:D\varphi]_t
\sigma^\varphi_t(v_j)\otimes e_{ij}$.
Then $u_t$ is
a $\sigma^{\varphi\otimes\mathrm{Tr}}$-cocycle.
We verify that $u_t$ satisfies the relative modular condition.
Let $\mathbf{D}:=\{z\in \mathbf{C}\mid 0<\mathrm{Im}(z)<1\}$,
and 
\[
A(\mathbf{D}):=\{f(z)\mid f(z)
\mbox{ is analytic on }\mathbf{D},
\mbox{bounded, continuous on }
\ovl{\mathbf{D}}\}.
\]
Take $x,y\in M\otimes B(H)$.
By the relative modular condition for
$[D\varphi\circ \Phi^\alpha:D\varphi]_t$
and
$\sum_{k=1}^n v_kx_{k\el}$ and $\sum_{k=1}^n y_{\el j}v_j^*$,
we can choose $F_{\el}(z)\in \mathcal{A}(\mathbf{D})$ such
that
\[F_\el(t)=\sum_{j,k=1}^n
\varphi\circ\phi_\rho\left([D\varphi\circ\phi_\rho:D\varphi]_t
\sigma_t^\varphi
(v_kx_{k\el})y_{\el j}v_j^*
\right)\quad\mbox{for all }t\in \R
\]
and 
\[
F_\el(t+\sqrt{-1})=\sum_{j,k=1}^n
\varphi\left(y_{\el j}v_j^*[D\varphi\circ\phi_\rho:D\varphi]_t
\sigma_t^\varphi
(v_kx_{k\el})
\right)\quad\mbox{for all }t\in \R.
\]
Set $F(z):=\sum_{\el=1}^n F_\el(z)\in \mathcal{A}(\mathbf{D})$.
Then we have
\begin{align*}
\varphi\circ \Phi^\alpha(u_t\sigma_t^{\varphi\otimes \mathrm{Tr}}(x)y)  &=
\sum_{i,j=1}^n
\varphi\circ \phi_\rho\left(v_i(u_t\sigma_t^{\varphi\otimes 
\mathrm{Tr}}(x)y)_{ij}v_j^*\right) \\
&=
\sum_{i,j,k,\el=1}^n
\varphi\circ \phi_\rho
\left(v_iu_{t,ik}\sigma_t^{\varphi}(x_{k\el})y_{\el j}v_j^*\right) \\
&=
\sum_{i,j,k,\el=1}^n
\varphi\circ \phi_\rho
\left(v_i
v_i^*[D\varphi\circ \phi_\rho:D\varphi]_t
\sigma^\varphi_t(v_k)
\sigma_t^{\varphi}(x_{k\el})y_{\el j}v_j^*\right) \\
&=
\sum_{j,k,\el=1}^n
\varphi\circ \phi_\rho\left(
[D\varphi\circ \phi_\rho:D\varphi]_t
\sigma^\varphi_t(v_k
x_{k\el})y_{\el j}v_j^*\right)  \\
&= F(t),
\end{align*}
and
\begin{align*}
(\varphi\otimes \mathrm{Tr})(yu_t\sigma_t^{\varphi\otimes \mathrm{Tr}}(x))
&=
\sum_{\el=1}^n\varphi
\left((yu_t\sigma_t^{\varphi\otimes \mathrm{Tr}}(x))_{\el\el}\right)
= 
\sum_{j,k,\el=1}^n
\varphi\left(y_{\el j}u_{t,jk}\sigma_t^{\varphi}(x_{k\el})\right)
\\
&= 
\sum_{j,k,\el=1}^n\varphi\left(y_{\el j}
v_j^*[D\varphi\circ \phi_\rho:D\varphi]_t
\sigma_t^\varphi(v_k)
\sigma_t^{\varphi}(x_{k\el})\right)\\
&=F(t+\sqrt{-1}).
\end{align*}
This shows that $u_t$ satisfies the relative modular condition. 
\end{proof}

Let $\sim\col \End(M)_0\ra \End(\tM)$ be the canonical extension
\cite[Theorem 2.4]{Iz-can2}.
We define the map
$\sim\col \Mor_0(M,M\oti B(K))\ra \Mor(\tM,\tM\oti B(K))$
by
\[
\tal=\si^{-1}\circ \wdt{\si\circ \al}
\quad\mbox{for all}\ \al\in \Mor_0(M,M\oti B(K)).
\]
In fact, $\tal$ does not depend on $\si$ as follows.

\begin{thm}
One has the following:
\begin{enumerate}
\item  $\tal(x)=\al(x)$ for all $x\in M$;

\item $\tal(\lambda^\vph(t))
=
d(\al)^{it}
[D\varphi\circ \Phi^\alpha: D\varphi\otimes \Tr_K]_t
(\lambda^\vph(t)\otimes 1)$
for all $t\in\R$.
\end{enumerate}
\end{thm}
\begin{proof}
Set $\rho:=\sigma_*(\alpha)$.
Then by definition, we have
\begin{align*}
&\trho(x)
=
\rho(x)\quad\mbox{for all}\ x\in M,
\\
&\trho(\la^\vph(t))
=
d(\rho)^{it}[D\vph\circ\ph_\rho:D\vph]_t \la^\vph(t)
\quad\mbox{for all}\ t\in\R.
\end{align*}
Since $\si^{-1}\circ\rho=\al$,
(1) follows.
On (2), we have
\begin{align*}
\tal(\la^\vph(t))
&=
\si^{-1}(\trho(\la^\vph(t)))
=
\sum_{k,\el=1}^n
v_k^* \trho(\la^\vph(t)) v_\el\oti e_{k\el}
\\
&=
\sum_{k,\el=1}^n
d(\rho)^{it} v_k^* 
[D\vph\circ\ph_\rho:D\vph]_t\la^\vph(t)
v_\el\oti e_{k\el}
\\
&=
\sum_{k,\el=1}^n
d(\rho)^{it}
(v_k^* [D\vph\circ\ph_\rho:D\vph]_t
\si_t^\vph(v_\el)\oti e_{k\el})(\la^\vph(t)\oti1)
\\
&=
d(\al)^{it}[D\vph\circ\Ph^\al:D\vph\oti\Tr_K]_t
(\la^\vph(t)\oti1).
\quad(\mbox{by Lemma \ref{lem: connes-cocycle}})
\end{align*}
\end{proof}

We say that $\al\in \Mor_0(M,M\oti B(K))$ is \emph{inner}
if there exists a unitary $U\in M\oti B(K)$ such that
$\al=U(\cdot\oti1)U^*$.
Denote by
$\Int(M,M\oti B(K))$,
$\oInt(M,M\oti B(K))$ and $\Cnt(M,M\oti B(K))$
the set of the inner homomorphisms,
the approximately inner homomorphisms and
the centrally trivial homomorphisms in $\Mor_0(M,M\oti B(K))$,
respectively.
(See Definition \ref{defn: app-cent}.)
Then we have the following bijective correspondence.
See \cite{M-T-endo-pre} for the notations used here.

\begin{lem}\label{lem: bijection}
The bijection $\si_*\col \Mor_0(M,M\oti B(K))\ra \End_0(M)$
yields the following bijective maps:
\begin{enumerate}
\item
$\si_*\col \Int(M, M\oti B(K))\ra \Int_{\dim(K)}(M)$;

\item
$\si_*\col \oInt(M, M\oti B(K))\ra \oInt_{\dim(K)}(M)$;

\item
$\si_*\col \Cnt(M,M\oti B(K))\ra \Cnd(M)$.
\end{enumerate}
\end{lem}
\begin{proof}
(1)
Assume that $\al=\Ad U(\cdot\oti1)$ for some unitary $U\in M\oti B(K)$.
Set $\rho:=\si_*(\al)$ and a Hilbert space $\meH\subs M$
which is spanned by $w_k:=\sum_{i=1}^n v_iU_{ik}$, $k=1,\dots,n$.
Then for $x\in M$, we have
\begin{align*}
\rho(x)
&=
\si(\al(x))
=
\si(U(x\oti1)U^*)
=
\sum_{i,j,k=1}^n\si(U_{ik}x U_{jk}^*\oti e_{ij})
\\
&=
\sum_{i,j,k=1}^n v_i U_{ik}x U_{jk}^* v_j^*
=
\sum_{k=1}^n w_k x w_k^*
=
\rho_{\meH}(x).
\end{align*} 
Hence $\rho=\rho_\meH\in \Int_{\dim(K)}(M)$.
Conversely if we have $\rho=\rho_\meH$ with $\dim\meH=n$,
then setting $U_{ik}:=v_i^* w_k$
for some orthonormal basis $\{w_k\}_{k=1}^n\subs\meH$,
we have $\si^{-1}\circ\rho=\Ad U(\cdot\oti1)$.

(2)
This follows from (1) and Lemma \ref{lem: homeo}.

(3)
Assume that $\al\in\Cnt(M,M\oti B(K))$.
Set $\rho:=\si_*(\al)$.
Take an $\om$-centralizing sequence $(x^\nu)_\nu$ in $M$.
Then $\al(x^\nu)-x^\nu\to0$ strongly* as $\nu\to\om$.
Hence $\rho(x^\nu)-\si(x^\nu\oti1)\to0$.
Since $\si(x^\nu\oti1)=\sum_{i,j=1}^n v_i x^\nu v_i^*$,
we see that $\rho(x^\nu)-x^\nu\to0$,
that is, $\rho\in\Cnd(M)$.
The converse can be proved similarly.
\end{proof}

We define the following set:
\[
\Mor_{{\rm CT}}(M, M\oti B(K))
=
\{\al\in \Mor_0(M,M\oti B(K))\mid \si_*(\al)\in \End(M)_{{\rm CT}}\}.
\]
The following lemma shows that
this set does not depend on $\si$.

\begin{lem}
Let $\al\in\Mor_0(M, M\oti B(K))$.
Then the following are equivalent:
\begin{enumerate}
\item
$\al\in\Mor_{{\rm CT}}(M, M\oti B(K))$;

\item
There exists $\ga\in \Aut_\th(Z(\tM))$
such that $\tal(z)=\ga(z)\oti1$
for all $z\in Z(\tM)$.
\end{enumerate}
\end{lem}
\begin{proof}
Assume that $\al\in\Mor_{{\rm CT}}(M, M\oti B(K))$.
Set $\rho:=\si_*(\al)$.
Then $\rho$ has Connes-Takesaki module $\mo(\rho)$.
By definition, $\si^{-1}(z)=z\oti1$ for $z\in Z(\tM)$.
For $z\in Z(\tM)$, we have
\[
\tal(z)=\si^{-1}(\trho(z))=\si^{-1}(\mo(\rho)(z))=\mo(\rho)(z)\oti1.
\]

Conversely, assume that such $\ga$ exists.
Then we have
\[
\trho(z)=\si(\tal(z))=\si(\ga(z)\oti1)=\ga(z).
\]
Hence $\rho$ has the Connes-Takesaki module $\ga$,
that is, $\al\in \Mor_{\rm CT}(M, M\oti B(K))$.
\end{proof}

In this situation, we say that $\al$ has the \emph{Connes-Takesaki module}
$\mo(\al):=\ga$.

\begin{thm}\label{thm: app-cent-prop}
Let $M$ be a properly infinite injective factor.
Then one has the following:
\begin{enumerate}
\item 
$\al\in\oInt(M,M\oti B(K))$

\noindent$\Leftrightarrow$
$\al\in \Mor_{{\rm CT}}(M,M\oti B(K))$
with
$\mo(\al)=\th_{\log(\dim(K)/d(\al))}$;

\item
$\al\in\Cnt(M,M\oti B(K))$

\noindent$\Leftrightarrow$
There exists a unitary
$U\in M_{d(\al),\dim(K)}(\tM)$ such that
$\tal=U(\cdot\oti1)U^*$.

\end{enumerate}
\end{thm}
\begin{proof}
This follows from \cite[Theorem 3.15, 4.12]{M-T-CMP}.
Note that if $\al\in \Cnt(M,M\oti B(K))$, then $d(\al)$ is integer
\cite[Theorem 3.3 (5)]{Iz-can2}.
\end{proof}

We obtain the following corollary.

\begin{cor}\label{cor: app-cent-infinite}
The following statements hold:
\begin{enumerate}
\item
If $M=\meR_{0,1}$, then

\vspace{5pt}
\begin{itemize}
\item
$\al\in \oInt(M,M\oti B(K))$
\\
$\Leftrightarrow$
$\ta\circ\Ph^\al=\ta\oti\tr_K$,
where $\ta$ is a trace on $M$;

\vspace{5pt}
\item
$\al\in\Cnt(M,M\oti B(K))$
\\
$\Leftrightarrow$
there exist $n\in\N$ and a unitary $U\in M\oti M_{\dim(K),n}(\C)$
such that
\[
\al(x)=U(x\oti1)U^*
\quad\mbox{for all}\ x\in M.
\]
\end{itemize}

\item
If $M=\meR_\la$ with $0<\la<1$,
then

\vspace{5pt}
\begin{itemize}
\item
$\al\in \oInt(M,M\oti B(K))$
\\
$\Leftrightarrow$
$[D\vph\circ\Ph^\al:D\vph\oti\tr_K]_T=1$,
where $\vph$ is a generalized trace on $M$
and $T=-2\pi/\log\la$;

\vspace{5pt}
\item
$\al\in\Cnt(M,M\oti B(K))$
\\
$\Leftrightarrow$
there exist $n\in\N$, a unitary $U\in M\oti M_{\dim(K),n}(\C)$
and $\{s_i\}_{i=1}^{n}\subs \R$
such that
\[
\al(x)=U\diag(\si_{s_1}^\vph(x),\dots,\si_{s_{n}}^\vph(x))U^*
\quad\mbox{for all}\ x\in M.
\]
\end{itemize}

\item
If $M=\meR_\infty$,
then

\vspace{5pt}
\begin{itemize}
\item
$\oInt(M,M\oti B(K))=\Mor_0(M,M\oti B(K))$;
\vspace{5pt}
\item
$\al\in\Cnt(M,M\oti B(K))$
\\
$\Leftrightarrow$
there exist $n\in\N$, a unitary $U\in M\oti M_{\dim(K),n}(\C)$
and $\{s_i\}_{i=1}^{n}\subs \R$
such that
\[
\al(x)=U\diag(\si_{s_1}^\vph(x),\dots,\si_{s_n}^\vph(x))U^*
\quad\mbox{for all}\ x\in M.
\]
\end{itemize}
\end{enumerate}

\end{cor}

\subsection{Canonical extension of cocycle actions}

We discuss canonical extension of cocycle actions.
Let $(\al,u)$ be a cocycle action
of $\bhG$ on a factor $M$.
For $\pi\in\IG$, we define the left inverse $\Ph_\pi^\al$
for $\al_\pi$ by
\[
\Ph_\pi^\al(x)=(1\oti T_{\opi,\pi}^*)u_{\opi,\pi}^*
(\al_\opi\oti\id)(x)u_{\opi,\pi}(1\oti T_{\opi,\pi})
\quad\mbox{for all}\ x\in M\oti B(H_\pi),
\]
where $T_{\opi,\pi}$ is an isometry intertwining
$\btr$ and $\opi\oti\pi$ \cite[p.491]{M-T-CMP}.
Then $\al_\pi\circ\Ph_\pi^\al$ is a faithful normal conditional
expectation from $M\oti B(H_\pi)$ onto $\al_\pi(M)$.
Set $d(\pi):=\dim(H_\pi)$.
Recall the diagonal operator $a\in M\oti\lhG$ of $u$
\cite[Definition 5.5]{M-T-CMP}:
\[
(a\oti1)(1\oti\De(e_\btr))=u(1\oti\De(e_\btr)).
\]

\begin{lem}
One has $\Ind(\al_\pi\circ \Ph_\pi^\al)=d(\pi)^2$
for all $\pi\in\IG$.
\end{lem}
\begin{proof}
Set $E_\pi:=\al_\pi\circ\Ph_\pi^\al$ and $d(\pi)=\dim H_\pi$.
Let $\{e_{\pi_{ij}}\}_{i,j=1}^{d(\pi)}$ be a system of
matrix units of $B(H_\pi)$.
We show that
$\{d(\pi)^{1/2}(1\oti e_{\pi_{ij}})a_\pi^*\}_{i,j=1}^{d(\pi)}$
is a quasi basis for $E_\pi$ \cite[Definition 1.2.2]{Wa}.
For any $y\in M$ and $1\leq k,\el\leq d(\pi)$, we have
\begin{align*}
&\Ph_\pi^\al(a_\pi(1\oti e_{\pi_{ji}})(y\oti e_{\pi_{k\el}}))
\\
&=
\de_{ik}
(1\oti T_{\opi,\pi}^*)
u_{\opi,\pi}^*
\al_\opi(a_\pi)
(\al_\opi(y)\oti e_{\pi_{j\el}})
u_{\opi,\pi}
(1\oti T_{\opi,\pi})
\\
&=
\de_{ik}
(1\oti T_{\opi,\pi}^*)
(a_{\opi}^*\oti1)
\al_\opi(a_\pi)
(\al_\opi(y)\oti e_{\pi_{j\el}})
(a_\opi\oti1)
(1\oti T_{\opi,\pi})
\\
&=
\de_{ik}
(1\oti T_{\opi,\pi}^*)
(\al_\opi(y)a_\opi\oti e_{\pi_{j\el}})
(1\oti T_{\opi,\pi})
\quad(\mbox{by \cite[Lemma 5.6(1)]{M-T-CMP}})
\\
&=
d(\pi)^{-1}\de_{ik}
\left(\al_\opi(y)a_\opi\right)_{\opi_{j\el}}.
\end{align*}
Using this equality, we have
\begin{align*}
&\hspace{14pt}\sum_{i,j=1}^{d(\pi)}
d(\pi)(1\oti e_{\pi_{ij}})a_\pi^*
E_\pi(a_\pi(1\oti e_{\pi_{ji}})(y\oti e_{k\el}))
\\
&=
\sum_{i,j=1}^{d(\pi)}
\de_{ik}
(1\oti e_{\pi_{ij}})a_\pi^*
\al_\pi
(
\left(
\al_\opi(y)a_\opi
\right)_{\opi_{j\el}}
)
=
\sum_{j=1}^{d(\pi)}
(1\oti e_{\pi_{kj}})a_\pi^*
\al_\pi
\big{(}
\left(
\al_\opi(y)a_\opi
\right)_{\opi_{j\el}}
\big{)}
\\
&=
\sum_{j=1}^{d(\pi)}
(1\oti \vep_{\pi_k})
(1\oti \vep_{\pi_j}^*\oti \vep_{\opi_j}^*)
(a_\pi^*\oti1)
\al_\pi
\big{(}
\al_\opi(y)a_\opi
\big{)}
(1\oti1\oti\vep_{\pi_\el})
\\
&=
\sum_{j=1}^{d(\pi)}
(1\oti \vep_{\pi_k})
(1\oti \vep_{\pi_j}^*\oti \vep_{\opi_j}^*)
u_{\pi,\opi}^*
\al_\pi
(
\al_\opi(y)a_\opi
)
(1\oti1\oti\vep_{\pi_\el})
\\
&=
\sum_{j=1}^{d(\pi)}
(1\oti \vep_{\pi_k})
(1\oti \vep_{\pi_j}^*\oti \vep_{\opi_j}^*)
(\id\oti\De)(\al(y))
u_{\pi,\opi}^*
\al_\pi(a_\opi)
(1\oti1\oti\vep_{\pi_\el})
\\
&=
\sum_{j=1}^{d(\pi)}
(y\oti \vep_{\pi_k})
(1\oti \vep_{\pi_j}^*\oti \vep_{\opi_j}^*)
u_{\pi,\opi}^*
\al_\pi(a_\opi)
(1\oti1\oti\vep_{\pi_\el})
\\
&=
\sum_{j=1}^{d(\pi)}
(y\oti \vep_{\pi_k})
(1\oti \vep_{\pi_j}^*\oti \vep_{\opi_j}^*)
(a_\pi^*\oti1)
\al_\pi(a_\opi)
(1\oti1\oti\vep_{\pi_\el})
\\
&=
\sum_{j=1}^{d(\pi)}
(y\oti \vep_{\pi_k})
(1\oti \vep_{\pi_j}^*\oti \vep_{\opi_j}^*)
(1\oti1\oti\vep_{\pi_\el})
=
y\oti e_{\pi_{k\el}}.
\end{align*}
Hence
$\{d(\pi)^{1/2}(1\oti e_{\pi_{ij}})a_\pi^*\}_{i,j=1}^{d(\pi)}$
is a quasi basis for $E_\pi$,
and we have
\begin{align*}
\Ind(E_\pi)
&=\sum_{i,j=1}^{d(\pi)}
d(\pi)^{1/2}(1\oti e_{\pi_{ij}})a_\pi^*
\cdot
d(\pi)^{1/2}a_\pi(1\oti e_{\pi_{ij}}^*)
\\
&=
d(\pi)^2(\id\oti\tr_\pi)(a_\pi^*a_\pi)
=d(\pi)^2.\quad(\mbox{by \cite[Lemma 5.6(2)]{M-T-CMP}})
\end{align*}
\end{proof}

\begin{defn}
We say that a cocycle action $\al\in\Mor(M,M\oti\lhG)$
is \emph{standard}
when the left inverse $\Ph_\pi^\al$ is standard
for each $\pi\in\IG$.
\end{defn}

\begin{prop}\label{prop: standard}
Let $\al\in\Mor(M,M\oti\lhG)$ be a cocycle action.
Then the following hold:
\begin{enumerate}
\item
If $\al$ is cocycle conjugate
to a standard cocycle action $\be\in\Mor(M^2,M^2\oti\lhG)$,
then $\al$ is standard.

\item
If $\al$ is free, then $\al$ is standard;

\item
If $\bhG$ is amenable, then $\al$ is standard.

\end{enumerate}
\end{prop}
\begin{proof}
(1)
Let $\pi\in\IG$.
Since the inclusion $\be_\pi(M^2)\subs M^2\oti B(H_\pi)$
is isomorphic to $\al_\pi(M)\subs M\oti B(H_\pi)$,
$[M\oti B(H_\pi):\al_\pi(M)]_0=[\be_\pi(M^2):M^2\oti B(H_\pi)]_0=d(\pi)^2$.
Hence $\al$ is standard.

(2)
For any $\pi\in\IG$,
the expectation $\al_\pi\circ\Ph_\pi^\al$
is minimal because of the irreducibility of
the inclusion $\al_\pi(M)\subs M\oti B(H_\pi)$
\cite[Lemma 2.8]{M-T-CMP}.
Hence $\al$ is standard.

(3)
Since
$[B(\el_2)\oti M\oti B(H_\pi):B(\el_2)\oti \al_\pi(M)]_0
=[M\oti B(H_\pi):\al_\pi(M)]_0$,
we may and do assume that $M$ is properly infinite
by considering $\id\oti \al$.
Then $\al$ is cocycle conjugate to an action
$\be$ on $M$ by Lemma \ref{lem: prop-inf-coboundary}.
By (1), it suffices to show that $\be$ is standard.
We check $E_\pi^{-1}=d(\pi)^2 E_\pi$
on $Q_\pi:=\be_\pi(M)'\cap (M\oti B(H_\pi))$
to use \cite[Theorem 1(2)]{Hi0}.
Take $x\in Q_\pi$.
Then by \cite[p.62 Remark]{Wa}, we have
\[
E_\pi^{-1}(x)
=
\sum_{i,j=1}^{d(\pi)}
d(\pi)^{1/2}(1\oti e_{\pi_{ij}})
x
d(\pi)^{1/2}(1\oti e_{\pi_{ij}}^*)
=
d(\pi)^2((\id\oti\tr_\pi)(x)\oti1).
\]
So,
$E_\pi$ is minimal if and only if the following holds:
\begin{equation}\label{eq: tr-Ph}
(\id\oti\tr_\pi)(x)=\Ph_\pi^\be(x)\in\C.
\end{equation}
If we can find a $\be$-invariant state
$\ps\in M^*$,
the proof is finished.
Indeed, applying $\ps$ to $\Ph_\pi^\be$,
we have
\begin{align*}
\ps(\Ph_\pi^\be(x))
&=
T_{\opi,\pi}^*(\ps\oti\id\oti\id)(\be_\opi(x))T_{\opi,\pi}
=
T_{\opi,\pi}^*(1_\opi\oti (\ps\oti\id)(x))T_{\opi,\pi}
\\
&=
(\ps\oti\tr_\pi)(x)
\end{align*}
Hence (\ref{eq: tr-Ph}) holds.
Such a state $\ps$ is constructed by
using an invariant mean $m\in\lhG^*$.
Take a state $\vph$ on $M^\al$ and set $\ps:=m((\vph\oti\id)(\be(x)))$.
Then we have $(\ps\oti\id)(\be(x))=\ps(x)1$ for all $x\in M$,
that is, $\ps$ is invariant under $\be$.
\end{proof}

\begin{prob}
Is any cocycle action of $\bhG$ on a factor standard?
\end{prob}

Let $\al\in \Mor(M,M\oti\lhG)$ be
a standard cocycle action with a 2-cocycle $u$.
Now for $\pi\in\IG$,
we consider the canonical extension
$\tal_\pi\in\Mor(\tM,\tM\oti B(H_\pi))$.
Collecting $(\tal_\pi)_\pi$,
we obtain a map $\tal\in \Mor(\tM,\tM\oti\lhG)$,
which is called the canonical extension of the action $\al$.
We have the following equalities:
\begin{align*}
&\tal_\pi(x)=\al_\pi(x)
\quad\mbox{for all}\ x\in M;
\\
&\tal_\pi(\la^\vph(t))
=[D\vph\circ\Ph_\pi^\al:D\vph\oti\tr_\pi]_t
(\la^\vph(t)\oti1)
\quad\mbox{for all}\ t\in\R, \vph\in W(M).
\end{align*}

The following two results even for actions of general Kac algebras
are obtained in \cite{Y-can},
where operator valued weight theory is fully used,
but we can directly prove them for the discrete $\bhG$.
We present their proofs for readers' convenience.

Take $\vph\in W(M)$.
For $t\in\R$, we define $w_t=(w_{t,\pi})_\pi\in U(M\oti \lhG)$
by
\[
w_{t,\pi}=[D\vph\circ\Ph_\pi^\al:D\vph\oti\tr_\pi]_t^*.
\]

\begin{lem}\label{lem: Connes-cocycle-alpha}
The unitary $w_t$ satisfies the following:
\[
(w_t\oti1)\al(w_t)u(\id\oti\De)(w_t^*)
=(\si_t^\vph\oti\id)(u).
\]
\end{lem}
\begin{proof}
By the chain rule of Connes' cocycles,
we may and do assume that $\vph$ is a state.
Let $\pi,\rho\in\IG$.
Using the isomorphism $\al_\pi^{-1}\col \al_\pi(M)\ra M$,
we have
\[
\al_\pi(w_{t,\rho})
=[D\vph\circ\Ph_\rho^\al\circ(\al_\pi^{-1}\oti\id)
: D\vph\circ \al_\pi^{-1}\oti \tr_\rho]_t^*.
\]
Since $E_\pi:=\al_\pi\circ\Ph_\pi^\al\col M\oti B(H_\pi)\ra \al_\pi(M)$
is a conditional expectation,
we have
\begin{align*}
\al_\pi(w_{t,\rho})
&=
[D\vph\circ\Ph_\rho^\al\circ(\al_\pi^{-1}\oti\id)\circ(E_\pi\oti\id)
: D\vph\circ \al_\pi^{-1}\circ E_\pi\oti \tr_\rho]_t^*
\\
&=
[D\vph\circ\Ph_\rho^\al\circ(\Ph_\pi^\al\oti\id)
: D\vph\circ \Ph_\pi^\al\oti \tr_\rho]_t^*
\\
&=
[D\vph\circ \Ph_\pi^\al\oti \tr_\rho
:D\vph\circ\Ph_\rho^\al\circ(\Ph_\pi^\al\oti\id)]_t.
\end{align*}
Then we have
\begin{align*}
&\hspace{16pt}(w_{t,\pi}\oti1_\rho)\al_\pi(w_{t,\rho})
\notag\\
&=
([D\vph\oti\tr_\pi:D\vph\circ\Ph_\pi^\al]_t\oti1_\rho)
[D\vph\circ \Ph_\pi^\al\oti \tr_\rho
:D\vph\circ\Ph_\rho^\al\circ(\Ph_\pi^\al\oti\id)]_t
\notag\\
&=
[D\vph\oti\tr_\pi\oti\tr_\rho:D\vph\circ\Ph_\pi^\al\oti\tr_\rho]_t
[D\vph\circ \Ph_\pi^\al\oti \tr_\rho
:D\vph\circ\Ph_\rho^\al\circ(\Ph_\pi^\al\oti\id)]_t
\notag\\
&=
[D\vph\oti\tr_\pi\oti\tr_\rho
:D\vph\circ\Ph_\rho^\al\circ(\Ph_\pi^\al\oti\id)]_t.
\end{align*}
Applying
$(\si_t^\vph\oti\id\oti\id)(u_{\pi,\rho}^*)$
and $u_{\pi,\rho}$ to the both sides,
we have
\begin{align}
&\hspace{16pt}
(\si_t^\vph\oti\id\oti\id)(u_{\pi,\rho}^*)
(w_{t,\pi}\oti1_\rho)\al_\pi(w_{t,\rho})
u_{\pi,\rho}
\notag\\
&=
(\si_t^\vph\oti\id\oti\id)(u_{\pi,\rho}^*)
u_{\pi,\rho}
\notag\\
&\quad\cdot
[D\vph\oti\tr_\pi\oti\tr_\rho\circ\Ad u_{\pi,\rho}
:D\vph\circ\Ph_\rho^\al\circ(\Ph_\pi^\al\oti\id)
\circ\Ad u_{\pi,\rho}]_t
\notag\\
&=
[D\vph\oti\tr_\pi\oti\tr_\rho
:D\vph\circ\Ph_\rho^\al\circ(\Ph_\pi^\al\oti\id)
\circ\Ad u_{\pi,\rho}]_t.
\label{eq: pi-rho-Phi}
\end{align}

Recall the following formula \cite[Lemma 2.5]{M-T-CMP}:
for $X\in M\oti B(H_\pi)\oti B(H_\rho)$,
\[
\Ph_\rho^\al\circ(\Ph_\pi^\al\oti\id)(u_{\pi,\rho}Xu_{\pi,\rho}^*)
=
\sum_{\si\prec \pi\cdot\rho}\sum_{S\in \ONB(\si,\pi\cdot\rho)}
\frac{d(\si)}{d(\pi)d(\rho)}
\Ph_\si^\al((1\oti S^*)X(1\oti S)).
\]
Hence for $S\in \ONB(\si,\pi\cdot\rho)$,
we have
\begin{align*}
\Ph_\rho^\al\big{(}(\Ph_\pi^\al\oti\id)
(u_{\pi,\rho}(1\oti SS^*)Xu_{\pi,\rho}^*)\big{)}
&=
\frac{d(\si)}{d(\pi)d(\rho)}
\Ph_\si^\al((1\oti S^*)X(1\oti S))
\\
&=
\Ph_\rho^\al\big{(}(\Ph_\pi^\al\oti\id)
(u_{\pi,\rho}X(1\oti SS^*)u_{\pi,\rho}^*)\big{)}.
\end{align*}
In particular, $1\oti SS^*$ is in the centralizer
of $\vph\circ \Ph_\rho^\al\circ(\Ph_\pi^\al\oti\id)\circ\Ad u_{\pi,\rho}$.
Trivially, it is also in the centralizer of $\vph\oti\tr_\pi\oti\tr_\rho$.
Hence we see that
the both sides of (\ref{eq: pi-rho-Phi}) commutes with
$1\oti SS^*$,
and we have
\begin{align}
(\ref{eq: pi-rho-Phi})
&=
\sum_{\si\prec\pi\cdot\rho}
\sum_{S\in \ONB(\si,\pi\cdot\rho)}
[D\vph\oti\tr_\pi\oti\tr_\rho
:D\vph\circ\Ph_\rho^\al\circ(\Ph_\pi^\al\oti\id)]_t
(1\oti SS^*)
\notag\\
&=
\sum_{\si\prec\pi\cdot\rho}
\sum_{S\in \ONB(\si,\pi\cdot\rho)}
[D(\vph\oti\tr_\pi\oti\tr_\rho)_{1\oti SS^*}
\notag
\\
&\hspace{100pt}
:D\big{(}\vph\circ\Ph_\rho^\al
\circ(\Ph_\pi^\al\oti\id)
\circ\Ad u_{\pi,\rho}\big{)}_{1\oti SS^*}]_t,
\label{eq: si-pi-rho-S}
\end{align}
where the last cocycles are evaluated in
$\big{(}M\oti B(H_\pi)\oti B(H_\rho)\big{)}_{1\oti SS^*}$.

Let $\Th_S\col B(H_\si)\ra \big{(}B(H_\pi)\oti B(H_\rho)\big{)}_{SS^*}$
be the isomorphism defined by $\Th_S(x)=SxS^*$ for $x\in B(H_\si)$.
Using
\begin{align*}
&
(\tr_\pi\oti \tr_\rho)_{SS^*}=
\frac{d(\si)}{d(\pi)d(\rho)}\tr_\si\circ\Th_S^{-1}
\\
&
\big{(}\vph\circ\Ph_\rho^\al\circ(\Ph_\pi^\al\oti\id)
\circ\Ad u_{\pi,\rho}\big{)}_{1\oti SS^*}
=
\frac{d(\si)}{d(\pi)d(\rho)} \vph\circ \Ph_\si^{\al} \circ(\id\oti\Th_S^{-1}),
\end{align*}
we have
\begin{align*}
(\ref{eq: si-pi-rho-S})
&=
\sum_{\si\prec\pi\cdot\rho}
\sum_{S\in \ONB(\si,\pi\cdot\rho)}
[D\vph\oti\tr_\si\circ\Th_S^{-1}
:D\vph\circ\Ph_\si^\al \circ (\id\oti \Th_S^{-1})]_t
\\
&=
\sum_{\si\prec\pi\cdot\rho}
\sum_{S\in \ONB(\si,\pi\cdot\rho)}
(\id\oti \Th_S)
\big{(}
[D\vph\oti\tr_\si
:D\vph\circ\Ph_\si^\al]_t
\big{)}
\\
&=
\sum_{\si\prec\pi\cdot\rho}
\sum_{S\in \ONB(\si,\pi\cdot\rho)}
(\id\oti\De)(w_{t})(1\oti SS^*)
=(\id\oti{}_\pi \De_\rho)(w_t).
\end{align*}
Thus we get
\[
(\si_t^\vph\oti\id\oti\id)(u_{\pi,\rho}^*)
(w_{t,\pi}\oti1_\rho)\al_\pi(w_{t,\rho})
u_{\pi,\rho}
=(\id\oti{}_\pi \De_\rho)(w_t).
\]
\end{proof}

\begin{thm}\label{thm: canonical-ext-action}
Let $(\al,u)$ be a standard cocycle action of $\bhG$ on a
factor $M$.
Then the canonical extension $(\tal,u)$ is a cocycle action
on $\tM$.
\end{thm}
\begin{proof}
We will check
$(\tal\oti\id)\circ\tal=\Ad u\circ(\id\oti\De)\circ\tal$.
We have $\tal=\al$ on $M$, and that is trivial.
For $t\in\R$, $\al(\la^\vph(t))=w_{t}^*(\la^\vph(t)\oti1)$.
The previous lemma yields
\begin{align*}
(\tal\oti\id)(\tal(\la^\vph(t)))
&=
(\tal\oti\id)(w_{t}^*(\la^\vph(t)\oti1))
\\
&=
(\al\oti\id)(w_{t}^*)(w_{t}^*\oti1)
(\la^\vph(t)\oti1\oti1)
\\
&=
u
(\id\oti\De)(w_{t}^*)
(\si_t^\vph\oti\id\oti\id)(u^*)
(\la^\vph(t)\oti1\oti1)
\\
&=
u(\id\oti\De)(\tal(\la^\vph(t)))u^*.
\end{align*}
\end{proof}

\begin{lem}\label{lem: trace-inv}
Let $(\al,u)$ be a standard cocycle action of $\bhG$ on $M$.
The canonical trace $\ta$ on $\tM$ is invariant under $\tal$,
that is, $\ta\circ\Ph_\pi^\tal=\ta\oti\tr_\pi$
for all $\pi\in\IG$.
\end{lem}
\begin{proof}
Let $\vph\in W(M)$.
Take a positive operator $h$ affiliated in $\tM_{\hvph}$
such that $h^{it}=\la^\vph(t)$.
Then the canonical trace is given by $\ta:=\hvph_{h^{-1}}$,
which does not depend on the choice of the weight $\vph$.
Let $T_\th\col\tM\ra M$ be the averaging operator valued weight for $\th$.
Then $\hvph=\vph\circ T_\th$.
Since $\th$ commutes with $\tal$,
we have
\[
[D\hvph\circ\Ph_\pi^\tal:D\hvph\oti\tr_\pi]_t
=[D\vph\circ\Ph_\pi^\al\circ(T_\th\oti\id):D\vph\circ T_\th\oti\tr_\pi]_t
=[D\vph\circ\Ph_\pi^\al:D\vph\oti\tr_\pi]_t.
\]
This implies
\begin{align*}
&\hspace{16pt}
[D\ta\circ\Ph_\pi^\tal:D\ta\oti\tr_\pi]_t
\\
&=
[D\ta\circ\Ph_\pi^\tal:D\hvph\circ\Ph_\pi^\hal]_t
[D\hvph\circ\Ph_\pi^\tal:D\hvph\oti\tr_\pi]_t
[D\hvph\oti\tr_\pi:D\ta\oti\tr_\pi]_t
\\
&=
\tal_\pi(h^{-it})[D\vph\circ\Ph_\pi^\al:D\vph\oti\tr_\pi]_t
(h^{it}\oti1)
=1.
\end{align*}
\end{proof}

Since $\tal$ commutes with $\th$,
$\tal$ extends to an action on $\tM\rti_\th\R$.
We call it the \emph{second canonical extension}
and denote that by $\wdt{\tal}$.

\begin{cor}\label{cor: second-conjugate}
Let $\al\in\Mor(M,M\oti\lhG)$ be a standard action.
The second canonical extension $\wdt{\tal}$
is cocycle conjugate to $\id_{B(\el_2)}\oti\al$.
\end{cor}
\begin{proof}
Let $\vph$ be a faithful normal semifinite weight on $M$.
We regard $\tM=M\rti_{\si^\vph}\R$.
Define $w(\cdot)\in U(L^\infty(\R)\oti M\oti\lhG)$
by $w(t)=w_{-t}$ for $t\in\R$.
Then $w$ is an $\id_{B(L^2(\R))}\oti\al$-cocycle.
By Takesaki duality, there exists a canonical isomorphism
$\tM\rti_\th\R\cong B(L^2(\R))\oti M$
intertwining the actions $\wdt{\tal}$ and $\Ad w\circ(\id\oti\al)$.
\end{proof}

\ifx\undefined\bysame
\newcommand{\bysame}{\leavevmode\hbox to3em{\hrulefill}\,}
\fi

\end{document}